\documentclass[a4paper,leqno,12pt]{amsart}
\usepackage{csquotes}
\usepackage{amsmath,amsthm}
\usepackage{amssymb}
\usepackage{xspace}
\usepackage{bm}
\usepackage{amstext}
\usepackage{comment}
\usepackage{amsfonts}
\usepackage{graphicx}
\usepackage[mathscr]{euscript}
\usepackage{amscd}
\usepackage{latexsym}
\usepackage{amssymb}
\usepackage{enumerate}
\usepackage{mathrsfs}
\usepackage[cmtip,all]{xy}
\usepackage{xcolor}
\usepackage{standalone}
\usepackage{tikz}
\usetikzlibrary{arrows.meta}
\usepackage{tikz-cd}
\usetikzlibrary{intersections}
\usetikzlibrary{calc}
\usetikzlibrary{decorations.markings}
\usetikzlibrary{arrows}
\usetikzlibrary{shapes,arrows,shapes.multipart}

\usepackage{xcolor}

\theoremstyle{plain}
\swapnumbers
    \newtheorem{thm}[figure]{Theorem}
    \newtheorem{prop}[figure]{Proposition}
    \newtheorem{lemma}[figure]{Lemma}
    
    \newtheorem{corollary}[figure]{Corollary}
        
    \newtheorem{subsec}[figure]{}
    
    \newtheorem*{thma}{Theorem A}
    
    \newtheorem*{thmc}{Theorem C}
    \newtheorem*{thmd}{Theorem D}
        \newtheorem*{thme}{Theorem E}

\theoremstyle{definition}

\theoremstyle{remark}
        \newtheorem{remark}[figure]{Remark}
        
        \newtheorem{example}[figure]{Example}

    \newtheorem{notation}[figure]{Notation}
\renewcommand{\thefigure}{\arabic{section}.\arabic{figure}}
\newenvironment{mysubsect}[2][]
{\begin{subsec}\begin{upshape}\begin{bfseries}{#2\vsn.}
\end{bfseries}{#1}}
{\end{upshape}\end{subsec}}
\newcommand{\sect}{\setcounter{figure}{0}\section}
\newcommand{\cU}{\mathcal{U}}

\newcommand{\vsn}{\vspace{2 mm}}

\newenvironment{myeq}[1][]
{\stepcounter{figure}\begin{equation}\tag{\thefigure}{#1}}
{\end{equation}}
\newcommand{\myodiag}[2][]
{\stepcounter{figure}\begin{equation}
     \tag{\thefigure}{#1}\vcenter{\xymatrix@R=10pt@C=1pt{#2}}\end{equation}}
\newcommand{\mypdiag}[2][]
{\stepcounter{figure}\begin{equation}
     \tag{\thefigure}{#1}\vcenter{\xymatrix@R=-1pt@C=5pt{#2}}\end{equation}}
\newcommand{\myqdiag}[2][]
{\stepcounter{figure}\begin{equation}
     \tag{\thefigure}{#1}\vcenter{\xymatrix@R=5pt@C=10pt{#2}}\end{equation}}
\newcommand{\myrdiag}[2][]
{\stepcounter{figure}\begin{equation}
     \tag{\thefigure}{#1}\vcenter{\xymatrix@R=3pt@C=25pt{#2}}\end{equation}}
\newcommand{\mysdiag}[2][]
{\stepcounter{figure}\begin{equation}
     \tag{\thefigure}{#1}\vcenter{\xymatrix@R=0pt@C=20pt{#2}}\end{equation}}
\newcommand{\mytdiag}[2][]
{\stepcounter{figure}\begin{equation}
     \tag{\thefigure}{#1}\vcenter{\xymatrix@R=0pt@C=-27pt{#2}}\end{equation}}
\newcommand{\myudiag}[2][]
{\stepcounter{figure}\begin{equation}
     \tag{\thefigure}{#1}\vcenter{\xymatrix@R=9pt@C=20pt{#2}}\end{equation}}
\newcommand{\myvdiag}[2][]
{\stepcounter{figure}\begin{equation}
     \tag{\thefigure}{#1}\vcenter{\xymatrix@R=16pt@C=36pt{#2}}\end{equation}}
\newcommand{\mywdiag}[2][]
{\stepcounter{figure}\begin{equation}
     \tag{\thefigure}{#1}\vcenter{\xymatrix@R=15pt@C=26pt{#2}}\end{equation}}
\newcommand{\myxdiag}[2][]
{\stepcounter{figure}\begin{equation}
     \tag{\thefigure}{#1}\vcenter{\xymatrix@R=10pt@C=2pt{#2}}\end{equation}}
\newcommand{\myydiag}[2][]
{\stepcounter{figure}\begin{equation}
     \tag{\thefigure}{#1}\vcenter{\xymatrix@R=10pt@C=12pt{#2}}\end{equation}}
\newcommand{\RO}{\operatorname{RO}}
\newcommand{\uM}{\underline{M}}
\newcommand{\uH}{\underline{H}}

\newcommand{\uC}{\underline{C}}
\newcommand{\C}{\mathbb{C}}
\newcommand{\Fun}{\operatorname{Fun}}

\newcommand{\tr}{\operatorname{tr}}

\newcommand{\ho}{\operatorname{ho}}
\newcommand{\Hom}{\operatorname{Hom}}

\newcommand{\Id}{\operatorname{Id}}
\newcommand{\Image}{\operatorname{Im}}

\newcommand{\Ker}{\operatorname{Ker}}
\newcommand{\Coker}{\operatorname{Coker}}

\newcommand{\bE}{{E_{C_2}\Zz}}

\newcommand{\res}{\operatorname{res}}
\newcommand{\Z}{\mathbb{Z}}

    \newtheorem*{corb}{Corollary B}

\newcommand{\cP}{\mathcal{P}}
\newcommand{\cF}{\mathcal{F}}
\newcommand{\cG}{\mathcal{G}}
\newcommand{\tH}{\widetilde{H}}
\newcommand{\tHG}{\tH^{\bigstar}_{\cG}}

\newcommand{\cFh}{\cF \langle H \rangle}
\newcommand{\cFH}[1]{\cF \langle H\sb{#1} \rangle}

\newcommand{\uZ}{\underline{\Z}}

\newcommand{\bMp}{\mathbb{M}_{p}(\beta)}
\newcommand{\Zp}{\mathbb{Z}/p}
\newcommand{\uFp}{\underline{\Fp}}

\newcommand{\Fp}{\mathbb{F}_p}
\newcommand{\uZp}{\underline{\Zp}}

\newcommand{\F}{\mathbb{F}_{2}}
\newcommand{\uF}{\underline{\F}}
\newcommand{\Zz}{\mathbb{Z}/{2}}

\newcommand{\bM}{\mathbb{M}}
\newcommand{\K}{\mathcal{K}_4}
\newcommand{\bs}{\bigstar}
\newcommand{\cA}{\mathcal{A}}

\newcommand{\Sp}{\mathrm{Sp}}

\begin{document}

\title{ $\RO(C_p \times C_p)$-graded cohomology of universal spaces and the coefficient ring}
%

%
%

\author{Surojit Ghosh}
\address{Department of Mathematics, Indian Institute of Technology, Roorkee, Uttarakhand-247667, India}
\email{surojit.ghosh@ma.iitr.ac.in; surojitghosh89@gmail.com}

\author{Ankit Kumar}
\address{Department of Mathematics, Indian Institute of Technology, Roorkee, Uttarakhand-247667, India}
\email{ankit\_k@ma.iitr.ac.in}

\date{\today}
%
%
%

\begin{abstract}

We compute the $\RO(C_p \times C_p)$-graded Bredon cohomology of equivariant universal and classifying spaces associated to families of subgroups, with coefficients in the constant Mackey functor $\underline{\mathbb{F}_p}$. An explicit description of the resulting coefficient ring, including its multiplicative structure, is obtained. These computations are then applied to the study of lifts of cohomology operations via the Bredon cohomology of equivariant complex projective spaces.

\end{abstract}

\maketitle

\setcounter{section}{-1}
\sect{Introduction}

Equivariant stable homotopy theory studies spaces and spectra equipped with a group action, with the goal of understanding the associated equivariant (co)homology theories. Unlike the classical non-equivariant case, equivariant cohomology is naturally graded by the real representation ring $\RO(G)$ of a finite group $G$. Its computation reveals new algebraic structures arising from fixed points, transfers, and norms. Explicit descriptions of $\RO(G)$-graded Bredon cohomology remain scarce, and their determination forms a central problem in the field.  

A natural point of departure is the computation of the coefficient ring, i.e.\ the Bredon cohomology of a point with constant Mackey functor coefficients. For $G=C_p$, these were established in the classical work of Stong and Lewis~\cite{Sto86}. Subsequent advances include the additive structure of equivariant cohomology for $G=C_{pq}$ with both $\underline{A}$ and $\underline{\Z}$-coefficients~\cite{BG19}, as well as computations for cyclic groups of square-free order that identify large parts of the multiplicative structure~\cite{BG24}. For general cyclic groups, selected graded pieces of the cohomology ring were obtained in~\cite{BD24}. These works rely on tools such as the Tate square~\cite{GM95}, computations of geometric fixed points, Bredon cohomology techniques, and fixed point methods. Together, they point toward a general challenge: to extend knowledge of the coefficient ring beyond cyclic groups.

For the Klein four group $C_2 \times C_2$ (henceforth, we denote it by $\K$), the cohomology groups were computed in~\cite{HK17}, while the Mackey functor structure and ring structure of the positive cone of $\pi_{\bigstar}H\uF$ were determined in~\cite{EB20}. In this work, we carry this program forward by determining the multiplicative structure of the coefficient ring for $C_p \times C_p$ (for both $p=2$ and odd prime $p$) by using the Tate square. Our method relies on a careful analysis of the Bredon cohomology of the universal spaces entering the Tate square. Beyond their immediate computational role, these universal spaces are of independent interest, as they control the structure of power operations in equivariant homotopy theory.  

\medskip

This perspective connects directly to power operations on equivariant $E_\infty$-ring spectra. For $N_\infty$-operads, the group of power operations is governed by the $E$-cohomology of equivariant classifying spaces of symmetric groups. In particular, explicit computations of power operations for $H\uFp$ and $MU^G$ with $G$ a finite Abelian group play a central role. As a first step, one is led to compute the $\RO(C_p \times C_p)$-cohomology of certain graph complexes for both even and odd primes, which in turn determine the $C_p$-equivariant cohomology of $B_{C_p}\Zp$, the $C_p$-equivariant classifying space of $C_p$, providing a foundation for the study of power operations in equivariant homotopy theory.  

In this direction, our first main theorem identifies the cohomology of a universal space:

\begin{thma}
\begin{enumerate}
    \item For an odd prime $p$, the polynomial part of the $\RO(\cG)$–graded Bredon cohomology of $E_{C_p}\Zp$ is 
\[
\mathrm{P}(\widetilde{H}^{\bigstar}_{\cG}(E_{C_p}\Zp_{+};\uFp))=\dfrac{
\Fp\big[
a_{\beta},u_{\beta},\kappa_{\beta},
a_{\alpha_\ell},\kappa_{\alpha_\ell},
v_{S},z_{Q},w_{Q},v_{p-1}^{\pm}
\big]
}{\mathfrak{J}},
\]
where $0\le \ell\le p-1$, and where $S$ and $Q$ range over the subsets of 
$[p-1]=\{0,1,\ldots,p-1\}$ of having cardinality at least one and at least three, respectively.  
Here $\mathfrak{J}$ denotes the ideal generated by the relations listed in Theorem~\ref{main}.

\item For $p=2$, there is an isomorphism of graded rings
\[
\widetilde{H}^{\bigstar}_{\cG}(\bE_+; \uF)\;\cong\;
\frac{
    \Bigl[
        \F[a_\beta, u_\beta]
        \;\oplus \bigoplus_{j,k \geq 1}\;
        \F\bigl\langle \Sigma^{-1}\frac{1}{a_\beta^{j}u_\beta^{k}}  \bigr\rangle
    \Bigr]
    \,[\,a_{\alpha_0}, u_{\alpha_0}, a_{\alpha_1}, u_{\alpha_1}, v^{\pm}\,]
}{
    \left( 
        v\,u_\beta - u_{\alpha_0} u_{\alpha_1},\;
        v\,a_\beta - (a_{\alpha_0}u_{\alpha_1} + u_{\alpha_0}a_{\alpha_1})
    \right)
}
\]
where $v$ is the invertible class in degree 
\(
\alpha_0 + \alpha_1 - \beta - 1,
\)
as constructed in Proposition~\ref{inv}. This appears as a result in  \cite{DV25}. 

\end{enumerate}

\end{thma}
See Theorem~\ref{main} and Theorem~\ref{main2}.  
\medskip

Along the way, we also obtain the following corollary, computing the $\RO(C_p)$-graded Bredon cohomology of the classifying space $B_{C_p}\Zp$:  

\begin{corb}\begin{enumerate}
 \item For an odd $p$, the polynomial part of the $\RO(C_p)$--graded Bredon cohomology of \( B_{C_p}\Zp \) is given by
\[
\mathrm{P}(\widetilde{H}^{\bigstar_{\beta}}_{C_p}(B_{C_p}\Zp_{+};\uFp))
\;=\;
\Fp[a_{\beta},u_{\beta},\kappa_{\beta},\nu,c,b,\eta,\xi,\omega]/\sim,
\]
modulo the ideal generated by the classes mentioned in Proposition~\ref{poly-BCpZp}. 

    \item There is an isomorphism:
\[
\widetilde{H}^{\bigstar_{\beta}}_{C_2}(B_{C_2} \Z/2_+; \uF) \;\cong\;
\frac{
    \left[ \F[a_\beta, u_\beta] 
    \oplus \bigoplus_{j,k \geq 1}
   \F \langle{\Sigma^{-1}\tfrac{1}{a^j_{\beta}u^k_{\beta}}}\rangle
    \right][c, b]
}{
    (c^2 = a_\beta c + u_\beta b)
}.
\]
\end{enumerate}

\end{corb}
See Proposition~\ref{poly-BCpZp} and Proposition~\ref{cohomologyofclassfying}. It is worth noting that part (2) of this is a classical result that appeared in \cite{HK01}.

Once the structure of $E_{C_p}\Zp$ is in hand, we incorporate this into the Tate square to determine the coefficient ring:

\begin{thmc}
\begin{enumerate}
    \item For an odd prime $p$, the polynomial part of the $\RO(\cG)$--graded Bredon cohomology of a point with coefficients in the 
constant Mackey functor $\uFp$ is given in Theorem~\ref{cohomologyofpointoddp}.
    \item For $p=2$, the $\RO(\K)$--graded Bredon cohomology of a point with coefficients in the 
constant Mackey functor $\uF$ is given explicitly in Theorem~\ref{cohomologyofpoint}.  
\end{enumerate}

\end{thmc}

\medskip

With the coefficient ring established, a natural next step is to analyze the module structure of Bredon cohomology. For a $G$-space $X$ and a Green functor $\uM$, the Bredon cohomology $\uH^\bs_G(X;\uM)$ carries a canonical module structure over $\uH^\bs_G(S^0;\uM)$. The freeness of this module is intimately tied to the $G$-cell structure of $X$. A rich body of work explores this theme: for $G=C_p$, freeness results for complex projective spaces were established in~\cite{Lew88}; the case $G=C_2$ has been treated extensively in~\cite{DHM24}; and for $G=C_n$ with $n$ square-free, freeness theorems with $\underline{A}$- and $\uZ$-coefficients are proved in~\cite{BG19,BG24}, with applications to the cohomology of projective spaces and Grassmannians. Further freeness theorems for finite $C_p$-spaces appear in~\cite{BG21}, and extensions to general cyclic groups in~\cite{BD24}. Additional developments include computations for connected sums of projective spaces~\cite{BDK22} and a description of the equivariant cohomology of infinite complex projective space~\cite{BDK}.  

In this context, our next main theorem describes the Bredon cohomology of smash powers of the equivariant infinite complex projective space, taken over a \emph{complete universe}~$\cU$:

\begin{thmd}
\begin{enumerate}
\item For an odd prime $p$ and $r\geq 2$, the additive structure of the Bredon cohomology of the $r$-fold smash power of the equivariant projective space is given by:
\begin{align*}
    \widetilde{H}^{\bigstar}_{\cG}\!\left({\cP(\cU)^{\wedge (r)}};\uFp\right)
        \;\cong\; 
        \bigoplus_{k=0}^{\infty}\bigoplus_{ i=0 }^{p-1}\bigoplus_{ j=0 }^{p-1}\widetilde{H}^{\bigstar-j\Sigma_{\ell=0}^{p-1}\alpha_{\ell}-i\beta-k\rho_{\C}}_{\cG}\left(\cP(\cU)^{\wedge (r-1)};\uFp\right),
\end{align*}
 and $r=1$ is given as
\begin{align*}
    \widetilde{H}^{\bigstar}_{\cG}\left(\cP(\cU);\uFp\right)= \bigoplus_{t=0}^{\infty} \bigoplus_{i=0}^{p-1}\bigoplus_{j=0}^{p-1}\widetilde{H}^{\bigstar-j\Sigma_{\ell=0}^{p-1}\alpha_{\ell}-i\beta-k\rho_{\C}}_{\cG}\left(S^0;\uFp\right) 
\end{align*}
with $(i,j,k)\neq (0,0,0)$.    
    \item For $p=2$ and $r\geq 2$, the additive structure of the Bredon cohomology of the $r$-fold smash power of the equivariant projective space is given by:
\begin{align*}
        \widetilde{H}^{\bigstar}_{\K}\!\left({\cP(\cU)^{\wedge (r)}};\uF\right)
        \;\cong\;
        \bigoplus_{i=1}^{  \infty}&
        \widetilde{H}^{\bigstar-i\rho_{\C}}_{\K}\!\left({\cP(\cU)^{\wedge (r-1)}};\uF\right) \\
        \bigoplus_{i=0}^{  \infty}&\Big(\widetilde{H}^{\bigstar-i\rho_{\C}-2\alpha_0}_{\K}\!\left({\cP(\cU)^{\wedge (r-1)}};\uF\right) \\
        &\oplus \widetilde{H}^{\bigstar-i\rho_{\C}-2\alpha_1-2\beta}_{\K}\!\left({\cP(\cU)^{\wedge (r-1)}};\uF\right) \\
        &\oplus \widetilde{H}^{\bigstar-i\rho_{\C}-2\alpha_0-2\alpha_1-2\beta}_{\K}\!\left({\cP(\cU)^{\wedge (r-1)}};\uF\right)\Big),
\end{align*}
with the base case $r=1$ is given by
 \begin{align*}
\widetilde{H}^{\bigstar}_{\K}\!\left({\cP(\cU)};\uF\right)
        \;\cong\;
        \bigoplus_{i=1}^{  \infty}&
        \widetilde{H}^{\bigstar-i\rho_{\C}}_{\K}\!\left({S^0};\uF\right) 
         \bigoplus_{i=0}^{  \infty} \Big(\widetilde{H}^{\bigstar-i\rho_{\C}-2\alpha_0}_{\K}\!\left({{S^0}};\uF\right) \\
        &\oplus \widetilde{H}^{\bigstar-i\rho_{\C}-2\alpha_1-2\beta}_{\K}\!\left({{S^0}};\uF\right) 
        \oplus \widetilde{H}^{\bigstar-i\rho_{\C}-2\alpha_0-2\alpha_1-2\beta}_{\K}\!\left({{S^0}};\uF\right)\Big),
\end{align*}
\end{enumerate}
\end{thmd}

\medskip

Finally, we turn to a third direction of inquiry: the structure of equivariant Steenrod operations. In the classical (non-equivariant) setting, Steenrod operations provide a fundamental layer of structure on cohomology. In the equivariant context, however, explicit descriptions are subtle and comparatively rare. In general, for $E \in \Sp^G$, the $\RO(G)$-graded $E$-cohomology operations are given by $\pi^G_{\bigstar}(F(E,E))$, while the algebra of co-operations is $\pi^G_{\bigstar}(E \wedge E)$. When $E = H\uFp$, these are denoted by $\cA^{\bigstar}_G$ and $\cA^G_{\bigstar}$, respectively.  

For any subgroup $H \leq G$, there is a restriction map
\[
\res^G_H : \cA^{\bigstar}_G \longrightarrow \cA^{\bs}_H.
\]
Since $\RO(H)$ is a retract of $\RO(G)$, one may ask when an operation defined in 
$\cA^{\bs}_H$ can be lifted to an operation in $\cA^{\bs}_G$. Concretely, a cohomology 
operation $\theta \colon H\uFp \to \Sigma^\alpha H\uFp$ (of degree $\alpha \in \RO(H)$) 
is said to be \emph{liftable} if there exists an operation $\tilde{\theta} \in \cA^{\bs}_G$ 
of the same degree such that, after restricting from $G$ to $H$, the action of 
$\tilde{\theta}$ coincides with that of $\theta$.  

The central question is therefore: which $H$-equivariant cohomology operations admit lifts to $G$-equivariant operations? For cyclic groups of prime order, Caruso~\cite{Car99} proved that the only nontrivial liftable operation is the Bockstein homomorphism, using explicit computations of Bredon cohomology of Eilenberg--Mac\,Lane spaces and the structure of the coefficient ring. More recently, in~\cite{BDK} the author gave an alternative proof of Caruso’s result in the case $G=C_p$.  

Finally we investigate the analogous question for the group $C_p\times C_p$, and prove the following:  

\begin{thme}
Let $\theta$ be a cohomology operation of degree $2r(p-1)$ and $2r$ for odd $p$ and $p=2$ case, respectively which does not involve the B\"ockstein $\beta$.  
Then there does not exist an equivariant cohomology operation $\tilde{\theta}$ lifting $\theta$.
\end{thme}

See Theorem~\ref{nonlift}.

\begin{mysubsect}{Organization} We organize the rest of the paper as follows.  
In Section~2, first we work on odd prime $p$ and compute the polynomial part of $\RO(\cG)$-graded Bredon cohomology of the equivariant universal space $E_{C_p}\Zp$ in Theorem~\ref{main}.  
As a consequence, we obtain the corresponding $\RO(C_p)$-graded Bredon cohomology of the classifying space $B_{C_p}\Zp$ in Proposition~\ref{poly-BCpZp}. In Section~3, we compute $p=2$ case with the $\RO(\K)$-graded cohomology of $\bE$ in Theorem~\ref{main2} and $\RO(C_2)$-graded cohomology of classfying space $B_{C_2}\Zz$ in Proposition~\ref{cohomologyofclassfying}.  

In Section~4 and Section~5, we apply the Tate square formalism to $E_{C_p}\Zp$ and $E_{C_2}\Zz$ respectively, in order to compute the $\RO(\cG)$ and $\RO(\K)$-graded Bredon cohomology rings of a point with coefficients in $\uFp$ and $\uF$, respectively, as stated in Theorem~\ref{cohomologyofpointoddp} and Theorem~\ref{cohomologyofpoint} respectively.

Finally, in Section~6, we determine the additive structure of the Bredon cohomology of equivariant complex projective space for a certain complete universe and 
Further, the cohomology of the $r$-fold smash product of the equivariant infinite complex projective space in Theorem~\ref{cohomologyofcp}.  
The section concludes with Theorem~\ref{nonlift}, where we establish the non-liftability of certain Steenrod squares in the equivariant setting.  
\end{mysubsect}

\begin{mysubsect}{Notation} 
We use the following notation throughout the paper.  
\begin{enumerate}
\item We will use the symbol $p$ to denote an arbitrary prime number, except when a particular value of $p$ is explicitly specified.

\item We denote by $G$ a finite group, and by $C_p \times C_p$ the elementary abelian group of rank two at the prime $p$.
For odd primes $p$, we write $\mathcal G$ for this group, whereas for $p=2$ we denote it by $\K$.

\item We write $\sigma$ for the sign representation of group $C_2$ and $\xi$ for the two dimensional real irreducible $C_p$-representation given by rotation by the angle $2\pi/p$.

\item For group $C_p \times C_p=\langle{a,b\colon a^p=b^p=1}\rangle$, all $p$-order subgroup are given by $H_{i}=\langle{ab^i}\rangle$ and $K=\langle{b}\rangle$ where $0 \leq i \leq p-1$. 
\item Associated to each $p$-order subgroup $H_i$ of $\mathcal{G}$, we denote its irreducible real representations by $\alpha_i^{\, j}$ and for the subgroup $K$, we write its irreducible representations as $\beta^{\, j}$ for $1 \le j \le (p-1)/2$.  
Accordingly, the corresponding irreducible representations of $\mathcal{K}_4$ are denoted by $\alpha_0$, $\alpha_1$, and $\beta$.

\item We will use $
\bigstar$ for the $\RO(G)$-grading and $\ast$ to indicate the integer grading. In particular for $\cG,$ we consider $\bs=\Sigma_{i=0}^{p-1}a_i\alpha_i+b\beta+c$, 
where $a_i,b,c\in \Z$ (see Proposition \ref{exponentequivalence}) and $\bigstar_{\beta}$ if $a_i$'s are zero.
\item We denote by $\bM_{p}(\xi)$ and $\bM_{2}(\sigma)$ be the cohomology of a point for group $C_p$ (odd $p$) and $C_2$ respectively, that is,
\[
\bM_{p}(\xi) := \bigoplus_{m,n \in \Z}\widetilde{H}^{m+n\xi}_{C_p}(S^0;\uFp),\quad \bM_{2}(\sigma) := \bigoplus_{m,n\in \Z}\widetilde{H}^{m+n\sigma}_{C_2}(S^0;\uF)
\]
as a graded rings, where $\xi$ and $\sigma$ be the real irreducible representations of $C_p$ and $C_2$ respectively. We may also use $\bs_\xi$ to denote the $\RO(C_p)$-grading of the form $m+n\xi$.

\item Square brackets $[\,]$ indicate polynomial generators, while angle brackets $\langle \,\rangle$ indicate additive generators. For example,
\[
\Fp\langle a, b \rangle[c]
\]
denotes the algebra with additive generators $a, b$ and polynomial generator $c$, so its elements are of the form $lac^i+mbc^j$ for all $i,j \geq 0$ and $l,m\in \Fp$.
\item For the negative suspension classes of the form  $\Sigma^{-1}\frac{\kappa^{\epsilon}_{\beta}}{a^j_{\beta}u^k_{\beta}}$, $\Sigma^{-1}\frac{\kappa^{\epsilon}_{\alpha_i}}{a^j_{\alpha_i}u^k_{\alpha_i}}$, we always consider $j,k\geq1$ and $\epsilon=0,1$. 
   
\item We write $[i]$ to denote the set $\{0,1,\cdots,i\}$.
\item The notation $\cU$ is used for the complete $\cG$-universe, while $\rho_{\C}$ is used for the complex regular representation of the corresponding group.
\item For a map $q$, we use $\Image q$ to denote the image of $q$.
\item The notation $\mathrm{P}(\widetilde{H}^{\bigstar}_{G}(X_+;\uFp))$ is used to denote the polynomial part of $\widetilde{H}^{\bigstar}_{G}(X_+;\uFp)$. For example, $\mathrm{P}(\widetilde{H}^{\bigstar}_{C_2}(S^0;\uF))=\F[a_{\sigma},u_{\sigma}]$ where $\widetilde{H}^{\bigstar}_{C_2}(S^0;\uF)=\F[a_{\sigma},u_{\sigma}]\bigoplus \F\langle{\Sigma^{-1}\frac{1}{a^{k}_{\beta}u^{k}_{\beta}}}\rangle$. 
\item $2^\mathbb{N}$ denote the collection of all subsets of natural numbers $\mathbb{N}$.
\end{enumerate}
\end{mysubsect}

\paragraph{\textbf{Acknowledgement.}} The first author (SG) would like to thank Guoqi Yan for helpful suggestions. The research of the second author (AK) is supported by the University Grants Commission, India.

\sect{Preliminaries on Bredon Cohomology}
Bredon cohomology is the natural extension of singular cohomology to the equivariant setting, where a group $G$ acts on a space $X$. We recall some standard constructions of $G$--spaces.  
For an orthogonal representation $V$ of a finite group $G$, one has the unit disk $D(V)$ and the unit sphere $S(V)$ with the natural $G$-action, and the one-point compactification $S^V$. The equivariant analogue of CW complexes are \emph{$G$-CW complexes}, built inductively from cells of the form $G/H \times D^n$ for subgroups $H \leq G$, attached along $G/H \times S^{n-1}$

In this framework, one replaces constant abelian coefficient groups with a more structured notion: Mackey functors. Furthermore, grading is promoted from the integers $\Z$ to the real representation ring $\RO(G)$, allowing one to track cohomological degrees twisted by representations of $G$. For more details, one can check \cite{May96}.  

\vsn

\begin{mysubsect}{Mackey functors} The foundational algebraic object in this theory is the \emph{Burnside category} $\mathrm{Burn}_G$, whose objects are finite $G$-sets of the form $G/H$, and morphisms are spans of $G$-maps, up to isomorphism and group completion. Explicitly, a morphism $G/H \to G/K$ is a formal linear combination of diagrams
\[
G/H \xleftarrow{f} S \xrightarrow{g} G/K
\]
where $S$ is a finite $G$-set and $f$, $g$ are $G$-maps. The composition in $\mathrm{Burn}_G$ is given by pullback along these spans.

A \emph{Mackey functor} $\uM$ for $G$ is a contravariant additive functor:
\[
\uM: \mathrm{Burn}_G^{\mathrm{op}} \to \mathrm{Ab}.
\]
Equivalently, a Mackey functor consists of a pair of compatible systems of restriction and transfer maps:
\[
\res_K^H: \uM(G/H) \to \uM(G/K), \quad \mathrm{for } \quad K \leq H,
\]
\[
\mathrm{tr}_K^H: \uM(G/K) \to \uM(G/H), \quad \mathrm{for } \quad K \leq H,
\]
satisfying natural coherence conditions such as the double coset formula. The category of Mackey functors $\mathrm{Mack}_G$ is an abelian symmetric monoidal category, with \emph{box product} and internal Hom given by Day convolution.

\vsn

\begin{example}
A basic example is the \emph{constant Mackey functor} $\uC$, defined by $\uC(G/H) = C$ for an abelian group $C$, with restriction maps given by identity and transfers $\mathrm{tr}_K^H$ given by multiplication by $[H:K]$ for $K \leq H$. Its dual $\uC^\ast$ is obtained by interchanging restrictions and transfers, with $\uC^\ast(G/H)=\uC(G/H)$ on orbits.  

For instance, when $G=\K$ (the Klein four group) and $C=\F$, the corresponding Lewis diagrams are:
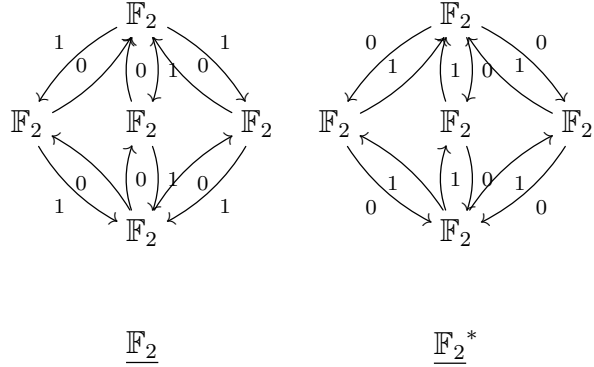
\[
\begin{array}{cc}
\xymatrix{
  & \F \ar@/^.5pc/[d]^{1} \ar@/^.5pc/[dr]^{1} \ar@/_.5pc/[dl]_{1} \\
  \F \ar@/_.5pc/[dr]_{1} \ar@/_.5pc/[ur]^{0}   
    &  \F \ar@/^.5pc/[u]_{0}  \ar@/^.5pc/[d]^{1} 
    &  \F \ar@/^.5pc/[ul]_{0} \ar@/^.5pc/[dl]^{1} \\
  & \F \ar@/^.5pc/[u]_{0} \ar@/_.5pc/[ul]^{0} \ar@/^.5pc/[ur]_{0}\\
  & \uF &
}
&
\xymatrix{
  & \F \ar@/^.5pc/[d]^{0} \ar@/^.5pc/[dr]^{0} \ar@/_.5pc/[dl]_{0} \\
  \F \ar@/_.5pc/[dr]_{0} \ar@/_.5pc/[ur]^{1}   
    &  \F \ar@/^.5pc/[u]_{1}  \ar@/^.5pc/[d]^{0} 
    &  \F \ar@/^.5pc/[ul]_{1} \ar@/^.5pc/[dl]^{0} \\
  & \F \ar@/^.5pc/[u]_{1} \ar@/_.5pc/[ul]^{1} \ar@/^.5pc/[ur]_{1}\\
  & \uF^{\ast} &
}
\end{array}
\]
Here the top $\F$ is the value at $\K/\K$, the three middle $\F$'s correspond to $\K/H_0,\K/H_1,\K/K$ (from left to right), and the bottom $\F$ is the value at $\K/e$. Downward arrows indicate restrictions and upward arrows indicate transfers.
\end{example}

Following examples of Mackey functors will appear in Proposition \ref{mackey}.

\begin{example}
    For the group $C_p$ (where $p$ can be an even or odd prime), the Mackey functor $\langle{\uFp}\rangle$ is defined by
\[
\langle{\uFp}\rangle(C_p/C_p)=\Fp,\qquad
\langle{\uFp}\rangle(C_p/e)=0,
\]
with restriction and transfer maps
\[
\operatorname{res}^{C_p}_{e}=0,
\qquad
\operatorname{tr}^{C_p}_{e}=0.
\]

\end{example}

\begin{example}
    For the group $C_2$, the Mackey functor $\langle{\Lambda}\rangle$ is defined by
\[
\langle{\Lambda}\rangle(C_2/C_2)=0,\qquad
\langle{\Lambda}\rangle(C_2/e)=\F,
\]
with restriction and transfer maps
\[
\operatorname{res}^{C_2}_{e}=0,
\qquad
\operatorname{tr}^{C_2}_{e}=0.
\]

\end{example}

\vsn
For a $G$-space $X$ and a Mackey functor $\uM$, the $\RO(G)$-graded Bredon cohomology groups are defined via equivariant stable homotopy theory as:
\[
\tH_G^\alpha(X_+; \uM) := [X_+, S^\alpha \wedge H\uM]^G,
\]
where $\alpha \in \RO(G)$ is a virtual real representation of $G$, $S^\alpha$ is the associated representation sphere, and ${H}\uM$ is the Eilenberg–Mac Lane $G$-spectrum representing Bredon cohomology with coefficients in $\uM$. 

For $X = S^0$, these cohomology groups define the coefficient ring of the theory:
\[
\tH_G^\bigstar(S^0; \uM),
\]
which inherits a graded ring structure from the smash product. This ring often serves as the base over which the cohomology of other $G$-spaces becomes a module via the map $X_+ \to S^0.$

The $\RO(G)$ graded cohomology theories can be formulated in a Mackey
functor--valued setting.  The following definition makes this explicit.

Let $X$ be a based $G$ space, let $\underline{M}$ denote a Mackey functor,
and  $\alpha \in \RO(G)$.  
The \emph{Mackey functor valued cohomology} of $X$ with coefficients in $\underline{M}$ is
the Mackey functor
\[
\underline{H}^\alpha_G(X;\underline{M})(G/K) \;=\;
\widetilde{H}^\alpha_G(G/K_+ \wedge X;\,\underline{M}),
\]
defined for each subgroup $K \leq G$. The restriction and transfer maps are induced from the
appropriate maps of $G$-spectra.
\end{mysubsect}

\vsn

\begin{mysubsect}{Some special classes in Bredon cohomology}
The $\RO(G)$-graded cohomology ring $\tH_G^\bigstar(S^0; \uFp)$ contains distinguished classes associated to representations \cite{HHR15}:
\begin{enumerate}
    \item For any real $G$-representation $V$ with no nontrivial fixed point, there is a canonical \emph{Euler class}
    \[
    a_V \in \tilde{H}_G^V(S^0; \uFp),
    \]
    represented by the inclusion $S^0 \hookrightarrow S^V$.
    
    \item If $V$ is oriented representation, there exists a \emph{Thom class}
    \[
    u_V \in \tH_G^{V - \dim V}(S^0; \uFp),
    \]
    arising from the cofiber sequence $S(V)_{+} \to D(V)_{+} \to S^V$.
\end{enumerate}

These classes satisfy natural multiplicative identities:
\[
a_V \cdot a_W = a_{V \oplus W}, \quad u_V \cdot u_W = u_{V \oplus W}.
\]

The behaviour of these generators under restriction and transfer maps is controlled by the structure of the representation $V$ over various subgroups $H \leq G$, which can significantly impact the internal structure of the cohomology ring.

\vsn 

\begin{remark}\label{cpgrad}
   For $G=C_p$, one observes that it suffices to consider the grading of the form $m+n \xi$ with $m,n \in \Z$ as $S^\alpha \wedge H\uFp \simeq H\uFp$ whenever both the dimensions of $\alpha$ and $\alpha^{C_p}$ are zero (See \cite[Appendix B]{FL04}). We use $\bigstar_{\xi}$ to denote this grading.
\end{remark}

We record the following computations for later use in the subsequent section
\begin{prop}{\cite{BG21}}{\label{c_2}}\begin{enumerate}
    \item For an odd prime $p$, the cohomology of a point in $\uFp$ coefficient is given by
    \[
    \bM_{p}(\xi)=\dfrac{\Fp[a_{\xi},\kappa_{\xi},u_{\xi}]}{(\kappa^2_{\xi})}\oplus \bigoplus_{j,k\geq 1} \Fp\langle{\Sigma^{-1}\frac{\kappa^{\epsilon}_{\beta}}{a^j_{\xi}u^k_{\xi}}}\rangle,\quad \epsilon=0,1.
    \]
    with $|\kappa_{\xi}|=\xi-1$, and $\xi$ is an $2$--dimensional real irreducible representation of group $C_p$.
    \item The $\RO(C_2)$–graded cohomology of a point with coefficients in $\uF$ is
\[
\bM_{2}(\sigma)=\F[a_{\sigma},u_{\sigma}] \oplus \bigoplus_{j,k \ge 1} \F\langle{\Sigma^{-1}\frac{1}{a_\sigma^j u_\sigma^k}\rangle}
\]
   where $\sigma$ denotes the sign representation of group $C_2$. 
\end{enumerate} 
\end{prop}

\begin{prop}{\cite{BG21}}
\begin{enumerate}
  \item For an odd prime $p$, the Bredon cohomology of free $C_p$--space $EC_p$ is given by
  \[
\widetilde{H}^{\bigstar_{\xi}}_{C_p}(E{C_p}_{+};\uFp)=\frac{\Fp[a_{\xi},\kappa_{\xi},u^{\pm}_{\xi}]}{(\kappa^2_{\xi})}.
\]

    \item For an $p=2$, the Bredon cohomology of free $C_2$--space $EC_2$ is given by
\[
\widetilde{H}^{\bigstar}_{C_2}(E{C_2}_{+};\uF)=\F[a_{\sigma},u^{\pm}_{\sigma}].
\]
\end{enumerate}
    
\end{prop}

In fact, the Mackey functor structures of 
$\underline{H}^{\!\bigstar}_{C_p}(S^{0};\uFp)$ and $\underline{H}^{\!\bigstar}_{C_2}(S^{0};\uF)$ were described explicitly by Ferland in \cite{Fer99}.

\begin{prop}{\label{mackey}}
\begin{enumerate}
     
\item For an odd prime $p$
    \begin{myeq}
\underline{H}^{\alpha}_{C_p}(S^{0};\uFp) =
\begin{cases}
\uFp & \text{if } |\alpha| = 0,\; |\alpha^{C_{p}}| \leq 0, \\[6pt]
\uFp^{*} & \text{if } |\alpha| = 0,\; |\alpha^{C_{p}}| > 0, \\[6pt]

\langle \Fp \rangle & \text{if } |\alpha| > 0,\; |\alpha^{C_{p}}| \leq 0 \quad \text{and $|\alpha|$ even} \\[6pt]
\langle \Fp \rangle & \text{if } |\alpha| < 0,\; |\alpha^{C_{p}}| > 1 \quad \text{and $|\alpha|$ odd} \\[6pt]
0 & \text{otherwise.}
\end{cases}
\end{myeq}

\item For $p=2$

    \begin{myeq}
\underline{H}^{\alpha}_{C_2}(S^{0};\uF) =
\begin{cases}
\uF & \text{if } |\alpha| = 0,\; |\alpha^{C_{2}}| \leq 0, \\[6pt]
\uF^{*} & \text{if } |\alpha| = 0,\; |\alpha^{C_{2}}| > 0, \\[6pt]
\langle \Lambda \rangle & \text{if } |\alpha| = 0,\; |\alpha^{C_{2}}| = 1, \\[6pt]
\langle \F \rangle & \text{if } |\alpha| > 0,\; |\alpha^{C_{2}}| \leq 0, \\[6pt]
\langle \F \rangle & \text{if } |\alpha| < 0,\; |\alpha^{C_{2}}| > 1, \\[6pt]
0 & \text{otherwise.}
\end{cases}
\end{myeq}
\end{enumerate}
\end{prop}

\end{mysubsect}
\begin{mysubsect}{Equivariant universal spaces}
 A \emph{family} $\cF$ of subgroups of $G$ is a collection of subgroups closed under conjugation and under taking subgroups. For any such family, one associates a $G$-space $E\cF$ satisfying the fixed-point condition \cite{May96}:
\[
E\cF^H = 
\begin{cases}
\simeq  \ast & \text{if } H \in \cF, \\
\emptyset & \text{otherwise}.
\end{cases}
\]
The $G$-space $E\cF$ is characterized up to $G$-homotopy equivalence by this universal property. It is initial among $G$-CW complexes whose isotropy lies in $\cF$, and plays a key role in equivariant homotopy theory, particularly in isotropy separation techniques.

One fundamental cofibre sequence associated with such a family is:
\[
E\cF_+ \to S^0 \to \widetilde{E\cF},
\]
which gives rise to a long exact sequence in Bredon cohomology, and can be interpreted as separating $\cF$-isotropy from the rest of the space.

Suppose now that $\cF \subset \cF'$ are families of subgroups such that $\cF' \setminus \cF = \{H\}$ for some subgroup $H \leq G$. Then there is a natural cofibre sequence of $G$-spaces:
\begin{myeq}
E\cF_+ \to E\cF'_+ \to E\cFh,    
\end{myeq}

where $E\cFh$ is a $G$-space characterized by its fixed-point behavior:
\[
E\cFh^K \simeq 
\begin{cases}
S^0 & \text{if } K = H, \\
\ast & \text{otherwise}.
\end{cases}
\]
Note that $\cFh$ is not in general a family, but the associated space $E\cFh$ is still well-defined and useful in constructing cell decompositions of universal spaces via successive cofiberings.

\vsn

Now, consider two finite groups $G$ and $\Pi$. Define a family $\cF_{gr}$ of subgroups of $G \times \Pi$ by

\begin{myeq}{\label{family}}
\cF_{gr} := \left\{ H \leq G \times \Pi \mid H \cap (1 \times \Pi) = e \right\}.    
\end{myeq}

The associated universal space $E_{G}\Pi := E\cF_{gr}$ is a $G \times \Pi$-space with free $\Pi$-action. The quotient
\[
B_G \Pi := E_G \Pi / \Pi
\]
is called the \emph{$G$-equivariant classifying space} for principal $\Pi$-bundles. It generalises the classical classifying space in the equivariant context and models the homotopy theory of $\Pi$-bundles in a $G$-equivariant setting.
\end{mysubsect}

 For a subgroup $H \leq G \times \Pi$, we write $\cF(H)$ to denote the family of all subgroups contained in $H$. For an abelian group  $G \times \Pi$, each family $\cF(H)$ corresponds canonically to the subgroup lattice of quotient group $G \times \Pi/H$, and we can identify
\[
E\cF(H) \simeq E(G \times \Pi / H).
\]

For the elementary abelian group $\cG = \langle a,b \colon a^p=b^p=1\rangle $, all subgroups of order $p$ are given by
\[
H_i=\langle a\, b^{\,i}\rangle \qquad (0 \le i \le p-1),
\]
together with the subgroup
\[
K=\langle b\rangle .
\]

Thus $\cG$ has exactly $p+1$ distinct subgroups of order $p$.
 For each $i$, there is a canonical cofibre sequence of $\cG$-spaces:
\begin{myeq}\label{cofseq3}
E{\cG}_+ \xrightarrow{q_i} E\cF(H_i)_+ \to E\cFH{i},
\end{myeq}
where the rightmost term is defined as the homotopy cofiber of the map $q_i$. 

Whereas the family, $\cF_{gr}$ for $\cG$, defined in \eqref{family} is given by 
\[
\mathcal{F}_{\mathrm{gr}} = \{e, H_{i}\colon 0 \leq i \leq p-1\},
\]

To compute the $\RO(C_p\times C_p)$--graded Bredon cohomology of $E_{C_{p}}{\Zp}:=E\cF_{gr}$, we follow the following tower of cofiber sequences: 

\begin{myeq}\label{tower}
\xymatrix@C=0.5cm{E\cF(H_0)_+ ={E \cF_{0}}_+ \ar[r] & {E\cF_{1}}_+ \ar[d] \ar[r] & \cdots {E\cF_{i-1}}_+ \ar[d] \ar[r]^{j_i} & {E\cF_{i}}_+ \ar[d] \ar[r] & \cdots {E\cF_{p-1}}_+ = {E_{C_p}\Z/p}_+ \ar[d] \\ &  E\cFH{1} &  E\cFH{i-1} & E\cFH{i} & E\cFH{p-1}}
\end{myeq}
Note that the family $\cF_{i}$ is given by $\{e,H_0,\cdots,H_{i}\}$.

For an odd prime \(p\), let \(\alpha_i^{\,j}\) and \(\beta^{\,j}\) denote the two-dimensional real irreducible representations of \(\mathcal G\) with kernels \(H_i\) and \(K := \langle b\rangle\), respectively, for \(1 \le j \le (p-1)/2\). 

\begin{prop}\label{exponentequivalence}
Let $p$ be an odd prime. Then for each integer $1 \le j \le (p-1)/2$ and each $0 \le i \le p-1$, there are equivalences in $\ho\,\Sp^{\cG}$ given by
\[
S^{\alpha_i^{\,j}} \wedge H\uFp \;\simeq\; S^{\alpha_i} \wedge H\uFp
\quad\text{and}\quad
S^{\beta^{\,j}} \wedge H\uFp \;\simeq\; S^{\beta} \wedge H\uFp .
\]
\end{prop}

\begin{proof}
We prove the statement for $\alpha_0$; the remaining cases are analogous.

It suffices to show that
\[
\underline{\pi_k}\bigl(S^{\alpha_0^{\,j}-\alpha_0}\wedge H\uFp\bigr)
\;\cong\;
\underline{H}^{\alpha_0^{\,j}-\alpha_0-k}_{\cG}(S^0;\uFp)
\]
is concentrated in degree $k=0$. By the uniqueness of equivariant Eilenberg--Mac Lane spectra, this implies the desired equivalence.

By Remark~\ref{cpgrad}, it is enough to compute the values of the Mackey functor
$\underline{H}^{\alpha_0^{\,j}-\alpha_0-k}_{\cG}(S^0;\uFp)$
on the orbits $\cG/\cG$ and $\cG/H$, where $H$ is a subgroup of order $p$. We will show that for $k=0$ it takes the value $\Fp$ on both orbits, with restriction map the identity and transfer map zero, and that it vanishes for all $k\neq 0$.

Consider the long exact sequence in $\RO(\cG)$-graded cohomology induced by the cofiber sequence
\[
S(\alpha_0)_+ \longrightarrow S^0 \longrightarrow S^{\alpha_0} :
\]
\[
\begin{aligned}
0 \to\;& 
\uH^{\alpha^j_0-2-\alpha_0}_{\cG}(S^0;\uFp)
\to \uH^{\alpha^j_0-2}_{\cG}(S^0;\uFp)
\to \uH^{\alpha^j_0-2}_{\cG}(S(\alpha_0)_{+};\uFp) \\
\to\;&
\uH^{\alpha^j_0-1-\alpha_0}_{\cG}(S^0;\uFp)
\to \uH^{\alpha^j_0-1}_{\cG}(S^0;\uFp)
\to \uH^{\alpha^j_0-1}_{\cG}(S(\alpha_0)_{+};\uFp) \\
\to\;&
\uH^{\alpha^j_0-\alpha_0}_{\cG}(S^0;\uFp)
\to \uH^{\alpha^j_0}_{\cG}(S^0;\uFp)
\to 0 .
\end{aligned}
\]

From the vanishing of the remaining terms in the relevant degrees, we see that
\[
\underline{H}^{\alpha^j_0-\alpha_0-k}_{\cG}(S^0;\uFp)=0
\qquad
\text{for all } k\in \Z\setminus\{0,1,2\}.
\]

For $k=1,2$, it suffices to show that
\[
\underline{H}^{\alpha^j_0-\alpha_0-k}_{\cG}(S^0;\uFp)(\cG/\cG)=0 .
\]
This follows since the map
\[
\uH^{\alpha^j_0-k}_{\cG}(S^0;\uFp)
\longrightarrow
\uH^{\alpha^j_0-k}_{\cG}(S(\alpha_0)_+;\uFp)
\]
is an isomorphism at the $\cG/\cG$-level for $k=1,2$, using the compatibility of the restriction and transfer maps.

For $k=0$, we have
\[
\uH^{\alpha^j_0-\alpha_0}_{\cG}(S^0;\uFp)(\cG/\cG)
\;\cong\;
\uH^{\alpha^j_0}_{\cG}(S^0;\uFp)(\cG/\cG)
=
\Fp\{a_{\alpha^j_0}\}.
\]

For a subgroup $H\neq H_0$ of order $p$, consider the following exact sequences:
\[
\begin{array}{cc}
\xymatrix{
\cG/\cG \colon & \Fp\{\kappa_{\alpha_0}\} \ar[r]^{} \ar@/^.5pc/[d]^{1} &\Fp \ar[r]^{} \ar@/^.5pc/[d]^{}& \widetilde{H}^{\alpha^j_{0}-\alpha_0}_{\cG}(S^0;\uFp) \ar[r]^{1} \ar@/^.5pc/[d]^{} &\Fp\{a_{\alpha^j_0}\} \ar@/^.5pc/[d]^{1}\to 0   \\
\cG/H  \colon &  \Fp\{\kappa_{\xi}\} \ar@/^.5pc/[u]_{0} \ar[r]^{1} &\Fp \ar[r]^{0} \ar@/^.5pc/[u]_{} &\Fp \ar[r]^{1} \ar@/^.5pc/[u]_{} &\Fp\{a_{\xi}\} \ar@/^.5pc/[u]_{0}\to 0   
}
\end{array}
\]
where $\xi=\res^{\cG}_H(\alpha_0)$. Commutativity shows that the restriction map is the identity and the transfer map is zero at $\uH_{\cG}^{\alpha^j_0-\alpha_0}(S^0;\uFp)$.

For $H=H_0$, we have the exact sequences
\[
\begin{array}{cc}
\xymatrix{
\cG/\cG \colon & \Fp\{\kappa_{\alpha_0}\} \ar[r]^{1} \ar@/^.5pc/[d]^{0} &\Fp \ar[r]^{} \ar@/^.5pc/[d]^{}& \widetilde{H}^{\alpha^j_{0}-\alpha_0}_{\cG}(S^0;\uFp) \ar[r]^{1} \ar@/^.5pc/[d]^{} &\Fp\{a_{\alpha^j_0}\} \ar@/^.5pc/[d]^{0}\ar[r]^{} & 0   \\
\cG/H_0  \colon &  0 \ar@/^.5pc/[u]_{0} \ar[r]^{0} &\Fp \ar[r]^{1} \ar@/^.5pc/[u]_{} &\Fp \ar[r]^{0} \ar@/^.5pc/[u]_{} & \quad 0 \ar@/^.5pc/[u]_{0} \quad\ar[r]^{}    
    &   0
}
\end{array}
\]

This shows that $\tr_{H_0}^{\cG}=0$, while the identity
\[
\res_e^{H_0}\res_{H_0}^{\cG}
=
\res_e^{H_1}\res_{H_1}^{\cG}
=\Id
\]
implies that $\res_{H_0}^{\cG}$ is the identity. This completes the proof.
\end{proof}

\begin{remark}
    
Since the representation ring of $\cG$ is given by
\[
\RO(\cG)
=
\Z\{\,1,\alpha_i^{\,j},\beta^{\,j}
\mid
0\le i\le p-1,\;
1\le j\le (p-1)/2
\},
\]
the preceding proposition shows that it suffices to consider gradings of the form
\[
\bigstar
=
\sum_{i=0}^{p-1} a_i\alpha_i + b\beta + c,
\qquad
a_i,b,c\in\Z.
\]

In contrast, for the Klein four group $\K$, there are three nontrivial one-dimensional real representations $\alpha_0$, $\alpha_1$, and $\beta$, with kernels $H_0$, $H_1$, and $K$, respectively. Consequently,
\[
\RO(\K)
=
\Z\{\,1,\alpha_0,\alpha_1,\beta\,\}.
\]

Since the cases $p=2$ and odd $p$ exhibit fundamentally different behavior, we begin by treating the case where $p$ is odd.

\end{remark}

\sect{Bredon cohomology of \(E_{C_p}\mathbb{Z}/p\) for odd primes \(p\)}

In this section, we compute the \(\RO(\mathcal G)\)-graded Bredon cohomology of certain equivariant universal spaces. 
To this end, we begin by analyzing the spaces \(E\mathcal F(H_i)\) associated to the family of subgroups generated by \(H_i\).

For the computation of 
\(\widetilde{H}^{\bigstar}_{\mathcal G}(E\mathcal F(H_i)_{+};\uFp)\), 
we use the identification
\[
E\mathcal F(H_i) \;\cong\; E\mathcal G/H_i \;\simeq\; S(\infty \alpha_i),
\]
which provides a natural skeletal filtration of \(E\mathcal F(H_i)\).
The associated graded pieces of this filtration satisfy
\[
E\mathcal F(H_i)^{(s)} / E\mathcal F(H_i)^{(s-1)}
\;\cong\;
\bigvee_{e \in I(s)} \mathcal G/H_i{}_+ \wedge S^s,
\]
where \(I(s)\) indexes the \(s\)-cells.

Applying the homotopy fixed point spectral sequence \cite{HM17}, where the $E_1$--page is given by: 
\[E_1^{s,t}(V)=\pi_{t-s}^{\cG}\Fun(E\cF(H_i)^{(s)}/E\cF(H_i)^{(s-1)}, \Sigma^{-V} H\uFp)\]
for $V\in \RO(\cG)$.  
By making some identification using adjunctions, one can calculate the $E_2$--page as:
\[
E^{s,V}_2= H^s (B(\cG/H_i); \pi^{H_i}_{\res_{H_i}(V)}(H\uFp)) \Rightarrow H^{s - V}_\cG(E\cF(H_i); \uFp),
\]
with differentials $d_r: E^{s,V}_r \to E^{s+r,V-r+1}_r$ for $r\geq 2.$ Since the action of $\cG/H_i$ on the coefficient is trivial (as everything is in $\Fp$). Therefore, the $E_2$--page using Proposition~\ref{c_2} can be written as 
\[E_2^{\ast, \bigstar}=\bMp\otimes\frac{\Fp[t,y]}{(t^2)} \otimes\Fp[u_{\alpha_i}^{\pm}, u_{\alpha_j - \beta}^\pm:j \neq i]\]
where the bidegrees are given by
 $|t|=(1,0)$, $|y|=(2,0)$, $|u_\beta|=(0,{\beta}-|\beta|)$, $|a_{\beta}|=(0,{\beta})$, $|\kappa_{\beta}|=(0,{\beta}-1)$, $|u_{\alpha_i}|= (0,{\alpha_i} -|{\alpha_i}|)$ and $|u_{\alpha_j -\beta}|= (0,{\alpha_j} -\beta).$

One derives that all the differentials are zero and hence the spectral sequence collapses at $E_2$--page. Since the spectral sequence is over $\Fp$, all the extensions are trivial. Thus,

\begin{myeq}\label{comh1}
\widetilde{H}^{\bigstar}_{\cG}(E\cF(H_i)_+;\uFp) \cong \bMp \otimes \dfrac{\Fp[t,y,u^{\pm}_{\alpha_i},u^{\pm}_{\alpha_j-\beta}\colon j \neq i]}{(t^2)}.
\end{myeq}

A similar computation can be used to compute the Bredon cohomology of free space $E\cG$, whose $E_2$--page is given by
\[E_2^{s, V}=H^s(B\cG;\pi_{\dim(V)}(H\Fp)) \Rightarrow \widetilde{H}^{s-V}_{\cG}(E{\cG}_+;\uFp)\]
with differential $d_2:E_2^{s, V} \to E_2^{s+2, V-1}$.
Since $\pi_{\dim(V)}(H\Fp)$ is concentrated in $\dim(V)=0$. Thus, the spectral sequence collapses at the  $E_2$-page. Therefore,
\begin{myeq}\label{cal1}
\widetilde{H}^{\bigstar}_{\cG}(E{\cG}_+;\uFp) \cong \dfrac{\Fp[t_0, t_1,y_0,y_1,u_{\alpha_i}^\pm, u_{\beta}^\pm\colon 0\leq i \leq p-1]}{(t^2_0,t^2_1)} 
\end{myeq}
where the bidegrees are given by
 $|t_l|=(1,0)$, $|y_l|=(2,0)$ for $l=0,1$.

Before we proceed to understand the Bredon cohomology of $E\cFH{i}$, which can be calculated with the help of cofiber sequence $\eqref{cofseq3}$, we describe map $q_i$ on cohomology as follow:
\begin{prop}\label{q_i-map}
The induced map 
\[
q^{\bigstar}_{i}:
\widetilde{H}^{\bigstar}_{\cG}(E\cF(H_i)_{+};\uFp)
\longrightarrow 
\widetilde{H}^{\bigstar}_{\cG}(E\cG_{+};\uFp)
\]
acts on generators by
\[
\begin{alignedat}{2}
t &\longmapsto t_0 - i t_1, &\qquad 
y &\longmapsto y_0 - i y_1, \\[2pt]
u_{\beta} &\longmapsto u_{\beta}, &
u_{\alpha_i} &\longmapsto u_{\alpha_i}, \\[2pt]
u_{\alpha_j-\beta}^{\pm} &\longmapsto (u_{\alpha_j}u_{\beta}^{-1})^{\pm}, & & (j\neq i), \\[4pt]
a_{\beta} &\longmapsto y_1 u_{\beta}, &
\kappa_{\beta} &\longmapsto t_1 u_{\beta}, \\[3pt]
 \Sigma^{-1}\!\dfrac{\kappa_{\beta}}{u_{\beta}^{j}a_{\beta}^{k}} &\longmapsto 0, &
\Sigma^{-1}\!\dfrac{1}{u_{\beta}^{j}a_{\beta}^{k}} &\longmapsto 0.
\end{alignedat}
\]
\end{prop}

 \begin{proof}
 Let $H_i$ be the size $p$ subgroup of $\cG$ as defined earlier. We consider  $\overline{(1,b)}$ be a fixed generator of $\cG/H_i$. 

Since $\Fp$ is trivial $\cG$--module, hence one can identifies 
\[\widetilde{H}_{\cG}^1(E\cF(H_i)_{+},\uFp)=\Hom(\cG/H_i,\Fp), \quad \widetilde{H}_{\cG}^1(E\cG_{+},\uFp)=\Hom(\cG/H_i,\Fp),
\]
which gives the following commutative square diagram
\[
\begin{tikzcd}[row sep=large, column sep=large]
\Hom(\cG/H_i,\Fp) \arrow[r, "q^{\ast}_{i}"] \arrow[d, "\cong" '] 
  & \Hom(\cG,\Fp) \arrow[d, "\cong"] \\
\Fp\{t\} \arrow[r] 
  & \Fp\{t_0,t_1\}.
\end{tikzcd}
\]
Here the left vertical isomorphism $\Hom(G/H_i,\Fp)\cong \Fp\{t\}$ sends a morphism $\psi:G/H_i\to\Fp$ to $\psi(\overline{(1,b)})$ and right vertical isomorphism sends $\varphi:\cG\to\Fp$ to $(\varphi(1,b),\varphi(a,1))$.

Note that computing the lower horizontal map is equivalent to calculating $0 \leq l \leq p-1$ such that $aH_i=b^lH_i$ where $aH_i=\{a^2b^i,a^3b^{2i},\cdots, a\}$ and $b^lH_i=\{ab^{i+l},a^2b^{2i+l},\cdots, b^l\}$, which gives $l=-i$. Hence, the lower horizontal map is given by $t\mapsto t_0-it_1$. Action on $y$ follows from the fact that $\beta(x)=y$ where $\beta$ is the B\"{o}ckstein in ordinary cohomology.

For the action of other classes, consider the following commutative diagram of rings 
\begin{myeq}\label{locsquare}
\xymatrix{ & \widetilde{H}^\bigstar_\cG(E\cF(H_i)_{+};\uFp) \ar[dl]_{q_i^\bigstar}  \\ \widetilde{H}^\bigstar_\cG(E\cG_+;\uFp)   & & \widetilde{H}^\bigstar_\cG(S^0;\uFp) \ar[dl]^{\pi_j^\bigstar}\ar[ll]_{q^\bigstar} \ar[ul]_{\pi_i^\bigstar} \\ & \widetilde{H}^\bigstar_\cG(E\cF(H_j)_{+};\uFp) \ar[ul]^{q_j^\bigstar} }
\end{myeq}
The action of ring homomorphism $q_i^\bigstar$ on other classes follow from the fact that $a_{\alpha_i},a_{\alpha_j},a_{\beta},u_{\alpha_i},u_{\alpha_j},u_{\beta},\kappa_{\alpha_i},\kappa_{\alpha_j}$, and 
$\kappa_{\beta}$ maps nontrivially under $q^{\bigstar}$, whereas the negative suspension classes map to zero due to degree reasons.
 \end{proof}

From the preceding discussion, the classes 
$a_{\alpha_i}, \kappa_{\alpha_i}, a_{\beta}, \kappa_{\beta}\in 
\widetilde{H}^{\bigstar}_{\cG}(S^{0};\uFp)$ act on 
$\widetilde{H}^{\bigstar}_{\cG}(E\cG_{+};\uFp)$ by multiplication with the classes  
$(y_{0}-iy_{1})u_{\alpha_i}$, $(t_{0}-it_{1})u_{\alpha_i}$, $y_{1}u_{\beta}$, and 
$t_{1}u_{\beta}$, respectively.  
Using these identifications, we may replace the integer–graded generators by the 
corresponding $\RO(\cG)$–graded classes.  
This yields the re–expressed descriptions of  
$\widetilde{H}^{\bigstar}_{\cG}(E\cG_{+};\uFp)$ and 
$\widetilde{H}^{\bigstar}_{\cG}(E\cF(H_i)_{+};\uFp)$ that will be used in the subsequent computations.

\begin{prop}\label{cohomologyofeg}
For an odd prime $p$,
\begin{enumerate}

\item The $\RO(\cG)$–graded Bredon cohomology of the free 
$\cG$–spaces $E\cG$ is given by:
\[
\widetilde{H}^{\bigstar}_{\cG}(E\cG_{+};\uFp)
\;\cong\;
\dfrac{
\Fp\!\left[
a_{\alpha_j},a_{\beta},
\kappa_{\alpha_j},\kappa_{\beta},
u^{\pm}_{\alpha_j},u^{\pm}_{\beta} \colon 0 \le j \le p-1
\right]
}{\!\!\sim},
\]
where the relations $\sim$ are induced from the following substitutions
\begin{align*}
a_{\alpha_j} &= (y_{0}-jy_{1})\,u_{\alpha_j}, &
\kappa_{\alpha_j} &= (t_{0}-jt_{1})\,u_{\alpha_i}, \\
a_{\beta} &= y_{1}\,u_{\beta}, &
\kappa_{\beta} &= t_{1}\,u_{\beta}, \\
t_{0}^{2} &= t_{1}^{2}=0.
\end{align*}

\item For each $0 \le i \le p-1$, the $\RO(\cG)$–graded Bredon cohomology of $E\cF(H_i)$ is
\[
\widetilde{H}^{\bigstar}_{\cG}(E\cF(H_i)_{+};\uFp)
\;\cong\;
\bMp \otimes
\dfrac{\Fp\!\left[
a_{\alpha_i},\kappa_{\alpha_i},
u^{\pm}_{\alpha_i},
u^{\pm}_{\alpha_j-\beta}\;:\; j\neq i
\right]}{(\kappa_{\alpha_i}^{2})}.
\]
\end{enumerate}
\end{prop}

We now turn to the computation of the Bredon cohomology of the $E\cFH{i}$.  
These calculations are controlled by the cofiber sequence \eqref{cofseq3}, and in 
particular by the computations of the cokernel and the kernel of the induced map 

\[
q_i^\bigstar: \tHG(E\cF(H_i)_+; \uFp) \to \tHG(E\cG_+; \uFp).
\]
Before we calculate them, we observe the following algebraic expression which make it easier: 
\begin{prop}
\label{alglemma}
Let \(p\) be a prime.  As an abelian group, there is an isomorphism
\[
\frac{\Fp[x^{\pm}, y^{\pm}]}{\Fp\left[\frac{x}{y}, \frac{y}{x}, x, y\right]} \cong x^{-1} \Fp\left[x^{-1}, \left(\frac{y}{x}\right)^{\pm}\right].
\]

\end{prop}

\begin{proof}
For $m, n \in\Z$, define a map $\phi: \Fp[x^{\pm}, y^{\pm}] \to x^{-1} \Fp[x^{-1}, (y/x)^{\pm}]$,
\[
\phi(x^ny^m)=
\begin{cases}
    0, & \text{if } m,n\geq 0\\
    \frac{1}{x^ny^m}, & \text{if } m,n<0\\
   y^{-(m+n)}(\frac{x}{y})^n ,& \text{if } m<0;n\geq0;m+n<0 \\
   x^{-(n+m)}(\frac{x}{y})^m ,& \text{if } n<0;m\geq0;m+n<0 \\
   0, & \text{otherwise.} 
\end{cases}
\]
It is straightforward to verify that $\phi$ is a surjective homomorphism with kernel given by the subalgebra $\Fp[\frac{x}{y}, \frac{y}{x}, x, y]$. 
\end{proof}

This proposition enables an explicit computation of $\Coker q_i^{\bigstar - 1}$. Since the multiplicative structure is trivial on the  $\Sigma^{-1}\Coker q_i^{\bigstar - 1}$, we obtain the following description of the cohomology ring:

\begin{prop}\label{pcohefbhi}
For \(0 \le i \le p-1\), the reduced Bredon cohomology
\(\widetilde{H}^{\bigstar}_{\mathcal G}(E\cFH{i};\uFp)\)
is given by
\begin{align*}
\widetilde{H}^{\bigstar}_{\mathcal G}(E\cFH{i};\uFp)
\;=\;&
\bigoplus_{k_1,k_2 \ge 1}
\Fp\Bigl\langle
\Sigma^{-1}\frac{\kappa_{\beta}^{\epsilon}}{a_{\beta}^{k_1}u_{\beta}^{k_2}}
\Bigr\rangle
\,
\frac{
\Fp\bigl[
a_{\alpha_i}, \kappa_{\alpha_i},
u_{\alpha_i}^{\pm},
u_{\alpha_j-\beta}^{\pm} \colon j \neq i
\bigr]
}{
(\kappa_{\alpha_i}^2)
}
\\[6pt]
&\;\oplus\;
\Sigma^{-1}u_{\beta}^{-1}
\Fp\bigl[
\kappa_{\alpha_i}, \kappa_{\alpha_j}, \kappa_{\beta},
a_{\alpha_i}, a_{\alpha_j}, a_{\beta},
\left(\tfrac{u_{\alpha_j}}{u_{\beta}}\right)^{\pm},
u_{\alpha_i}^{\pm}, u_{\beta}^{-1}
\colon j \neq i
\bigr]
\mod \Image q_i^{\bigstar}.
\end{align*}
\end{prop}

\begin{proof}
Consider the long exact sequence in cohomology associated to the cofibre sequence \eqref{cofseq3}:
\[
\cdots \;\longrightarrow\; 
\tHG(E\cFH{i}; \uFp) \;\longrightarrow\; 
\tHG(E\cF(H_i)_+; \uFp) \xrightarrow{\,q_i^\bigstar\,} 
\tHG(E\cG_+; \uFp) 
\;\longrightarrow\; \cdots.
\]
By Propositions~\ref{q_i-map}, \ref{cohomologyofeg}, and \ref{alglemma}, the maps in this sequence can be described explicitly.  
The desired description of $\tHG(E\cFH{i}; \uFp)$ then follows by computing the kernel and cokernel of 
\(
q_i^\bigstar.
\)
\end{proof}

To study the $\RO(\cG)$--graded Bredon cohomology of $E\cF_{i}$, we make use of the relationship between the universal spaces for the families considered earlier, captured by the following structural lemma:

\begin{lemma}\label{pbfamily}
Let $\cF_1, \cF_2$ be families of subgroups of a finite group $\cG$. Then the diagram
\[
\xymatrix{
E(\cF_1 \cap \cF_2)_+ \ar[r] \ar[d] & E\cF_1{}_+ \ar[d] \\
E\cF_2{}_+ \ar[r] & E(\cF_1 \cup \cF_2)_+
}
\]
is a pullback in the category of genuine $\cG$-spectra.
\end{lemma}

\begin{proof}
This follows from the fact that geometric fixed points detect weak equivalences and commute with suspension spectra. The maps on fibers are equivalences, hence the square is a homotopy pullback.
\end{proof}

As a consequence, one obtains a diagram of long exact sequences in the Bredon cohomology.

   \begin{myeq}{\label{pullp}}
\xymatrix@C=.5cm
{
\cdots  \widetilde{H}^{\bigstar}_{\cG}({E\cF_{i}}_+;\uFp) \ar[d]_-{\iota^{\bigstar}_i}\ar[r]^-{j^{\bigstar}_i}  &  \widetilde{H}^{\bigstar}_{\cG}({E\cF_{i-1}}_+;\uFp) \ar[d]^-{\eta^{\bigstar}_{i-1}}\ar[r]^-{\partial^{\bigstar}_{i-1}} &  \widetilde{H}^{\bigstar+1}_{\cG}(E\cFH{i};\uFp) \ar[d]^{=}  \cdots\\
\cdots \widetilde{H}^{\bigstar}_{\cG}(E\cF(H_i)_+;\uFp) \ar[r]^-{q^{\bigstar}_i} & \widetilde{H}^{\bigstar}_{\cG}(E\cG_+;\uFp) \ar[r]^-{\bar{\partial}_{i-1}^{\bigstar}} & \widetilde{H}^{\bigstar+1}_{\cG}(E\cFH{i};\uFp)  \cdots
}.
\end{myeq}

\begin{prop}{\label{integerlevel}}
    At the integer level, the boundary map $\partial^{\bigstar}_{0}:\widetilde{H}^{\bigstar}_{\cG}(E\cF(H_0)_{+};\uFp)\to\widetilde{H}^{\bigstar+1}_{\cG}(E\cFH{1};\uFp)$ is non-trivial. Moreover, it map classes 
    $\frac{\kappa_{\alpha_0}}{u_{\alpha_0}} \mapsto \Sigma^{-1}\frac{\kappa_{\beta}}{u_{\beta}}$ and 
    $\frac{a_{\alpha_0}}{u_{\alpha_0}} \mapsto \Sigma^{-1}\frac{a_{\beta}}{u_{\beta}}$.   
\end{prop}
\begin{proof}
    At degree $1$, the classes $\frac{\kappa_{\alpha_0}}{u_{\alpha_0}}\in \widetilde{H}^{1}_{\cG}(E\cF(H_0)_{+};\uFp)$ and  $\frac{\kappa_{\alpha_1}}{u_{\alpha_1}}\in \widetilde{H}^{1}_{\cG}(E\cF(H_1)_{+};\uFp)$ maps to themself in $\widetilde{H}^{1}_{\cG}(E\cG;\uFp)$ which are not equal, via morphisms $q^{\bigstar}_0 (= \eta_0^\bs)$ and $q^{\bigstar}_1$ respectively, which forces class $\frac{\kappa_{\alpha_0}}{u_{\alpha_0}}$ maps to nonzero class $\Sigma^{-1}\frac{\kappa_{\beta}}{u_{\beta}}$ due to the commutativity of diagram~\eqref{pullp} and degree reasons. Analogous strategy works for the class $\frac{a_{\alpha_0}}{u_{\alpha_0}}$.
\end{proof}

Since calculating the cohomology of $E\cF_{i}$ is equivalent to computing the kernel and cokernel of the corresponding boundary map $\partial^{\bigstar}_{i-1}$. The following proposition computes the kernel and cokernel of the boundary map for $i=1$:
\begin{prop}{\label{boundary_0}}
    The kernel and cokernel of the boundary map $\partial^{\bigstar}_{0}:\widetilde{H}^{\bigstar}_{\cG}(E\cF(H_0);\uFp)\to \widetilde{H}^{\bigstar+1}_{\cG}(E\cFH{1};\uFp) $ are given by:
    \begin{enumerate}
        \item \begin{align*}
\Ker \partial_{0}^{\bigstar}=\bigoplus_{k_1,k_2\geq 1}\Fp\langle{\Sigma^{-1}\frac{\kappa^{\epsilon}_{\beta}}{a^{k_1}_{\beta}u^{k_2}_{\beta}}\colon \epsilon=0,1}\rangle\frac{[a_{\alpha_0},\kappa_{\alpha_0},u^{\pm}_{\alpha_0},u^{\pm}_{\alpha_j-\beta}\colon j\neq 0]}{(\kappa^2_{\alpha_0})}\\
\bigoplus \Fp[a_{\beta},\kappa_{\beta},u_{\beta},a_{\alpha_0},\kappa_{\alpha_0},u_{\alpha_0},a_{\alpha_1},\kappa_{\alpha_1},u_{\alpha_1},({u_{\alpha_0}u_{\alpha_1-\beta}})^{\pm},{u^{\pm}_{\alpha_j-\beta}}\colon j\neq 0,1]/\sim 
\end{align*}
where $a_{\alpha_1}:=(\frac{a_{\alpha_0}}{u_{\alpha_0}}u_{\beta}-a_{\beta})u_{\alpha_1-\beta}$ and $\kappa_{\alpha_1}:=(\frac{\kappa_{\alpha_0}}{u_{\alpha_0}}u_{\beta}-\kappa_{\beta})u_{\alpha_1-\beta}$, 
subject to the relations induced from $\widetilde{H}^{\bigstar}_{\cG}(E\cF(H_0)_{+};\uFp)$. 

\item \begin{align*}    
\Coker \partial^{\bigstar}_{0}=\bigoplus_{k_1,k_2\geq 1}\Fp\langle{\Sigma^{-1}\frac{\kappa^{\epsilon}_{\beta}}{a^{k_1}_{\beta}u^{k_2}_{\beta}}\colon \epsilon=0,1}\rangle\frac{[a_{\alpha_1},\kappa_{\alpha_1},u^{\pm}_{\alpha_1},u^{\pm}_{\alpha_j-\beta}\colon j\neq 1]}{(\kappa^2_{\alpha_1})}\\
\bigoplus \Sigma^{-1}u^{-1}_{\beta}\Fp[\kappa_{\alpha_0},\kappa_{\alpha_1},\kappa_{\alpha_j},\frac{\kappa_{\beta}}{u_{\beta}},a_{\alpha_0},a_{\alpha_1},a_{\alpha_j},\frac{{a_{\beta}}}{u_{\beta}},\frac{u_{\alpha_0}}{u_{\beta}},\frac{u_{\alpha_1}}{u_{\beta}},(\frac{u_{\alpha_0}u_{\alpha_1}}{u_{\beta}})^{\pm},(\frac{u_{\alpha_j}}{u_{\beta}})^{\pm},u^{-1}_{\beta}\colon j\neq0, 1]/\sim
\end{align*}
subject to the relations induced from $\widetilde{H}^{\bigstar}_{\cG}(E\cFH{1};\uFp)$ and modulo $\Image \partial^{\bigstar}_0$.
    \end{enumerate}
\end{prop}
\begin{proof}
Computation of $\Ker \partial^{\bigstar}_0$ follows from the fact that $x \in \Ker \partial^{\bigstar}_0$ if and only if there exists $y\in \widetilde{H}^{\bigstar}_{\cG}(E\cF(H_1)_{+};\uFp)$ such that $q^{\bigstar}_{0}(x)=q^{\bigstar}_{1}(y)$, which gives $\Ker \partial^{\bigstar}_{0}=(q^{\bigstar}_0)^{-1}(\Image q^{\bigstar}_{0} \cap \Image q^{\bigstar}_{1})$.

For the cokernel $\Coker(\partial^{\bigstar}_{0})$, note that 
$\partial^{\bigstar}_{0}$ acts on the relevant classes from Proposition~\ref{integerlevel} as
\[
\partial_0^\bigstar\!\left( \frac{a_{\alpha_0}}{u_{\alpha_0}} \right)
    = \Sigma^{-1} \frac{a_\beta}{u_\beta},
\qquad
\partial_0^\bigstar\!\left( \frac{\kappa_{\alpha_0}}{u_{\alpha_0}} \right)
    = \Sigma^{-1} \frac{\kappa_\beta}{u_\beta},
\]
\[
\partial_0^\bigstar(u_{\alpha_0}^{-n})
    = \Sigma^{-1} u_{\alpha_0}^{-n},
\qquad
\partial_0^\bigstar(u_{\alpha_1 - \beta}^{\,n})
    = \Sigma^{-1} u_\beta^{-n} u_{\alpha_1}^{n}, \quad n\geq 1.
\]
From this data, the claimed description of 
$\Coker(\partial_0^\bigstar)$ follows immediately.
\end{proof}

From Proposition~\ref{boundary_0}, Proposition~\ref{coker} and diagram~\eqref{pullp}, one sees that the Bredon cohomology of \(E\cF_{i}\) is given by as the summand of classes from $\Ker \partial^{\bigstar}_{i-1}$ and $\Coker \partial^{\bigstar-1}_{i-1}$, but following proposition tells that classes from $\Coker \partial^{\bigstar-1}_{i-1}$ always vanish through $\partial^{\bigstar}_{i}$.
\begin{prop}\label{coker}
The cokernel of the boundary map
\[
\partial^{\bigstar}_{\,i-1} :
\widetilde{H}^{\bigstar}_{\cG}(E\cF_{i-1+};\uFp)
\;\longrightarrow\;
\widetilde{H}^{\bigstar+1}_{\cG}(E\cF(H_i)_{+};\uFp)
\]
is contained in the kernel of the boundary map
\[
\partial^{\bigstar}_{\,i} :
\widetilde{H}^{\bigstar}_{\cG}(E\cF_{i+};\uFp)
\;\longrightarrow\;
\widetilde{H}^{\bigstar+1}_{\cG}(E\cF(H_{i+1})_{+};\uFp).
\]
that is,
\[
\Coker(\partial^{\bigstar}_{\,i-1})
\;\subseteq\;
\Ker(\partial^{\bigstar}_{\,i}).
\]
\end{prop}

\begin{proof}
    For $x\in \Coker \partial^{\bigstar-1}_{i-1}$, implies $j^{\bigstar}_{i-1} x=0$, then proof follows from the fact that  $\partial^{\bigstar}_{i}$ factor through $j^{\bigstar}_{i-1}$.       
\end{proof}

 Because explicitly computing the cokernel is lengthy and technical, in the odd \(p\) case we will concentrate exclusively on the polynomial part from the kernel. The corresponding part of the Bredon cohomology $\widetilde{H}^{\bigstar}_{\cG}({E\cF_{i}}_{+};\uFp)$ will be denoted by $\mathrm{P}(\widetilde{H}^{\bigstar}_{\cG}({E\cF_{i}}_{+};\uFp))$, the polynomial part of $\widetilde{H}^{\bigstar}_{\cG}({E\cF_{i}}_{+};\uFp)$.

     To proceed further, we introduce the following new classes in $\widetilde{H}^{\bigstar}_{\cG}({E\cF_{i}}_{+};\uFp)$.

\begin{prop}\label{invertiable}
For each nonempty subset $S\subseteq [i]= \{0,1,\dots,i\}$ with $i\ge 1$, there exists a class
\[
v_{S}\in \widetilde{H}^{\bigstar}_{\cG}(E\cF_{i+} ;\uFp).
\]
satisfying the following properties;
\begin{enumerate}
     \item $v_{S_1}v_{S_2}=\begin{cases}
        v_{S_1\cup S_2}u_{\beta}, & S_1 \cap S_2=\phi\\
        v_{S_1\cup S_2}v_{S_1\cap S_2}, & S_1 \cap S_2\neq\phi
    \end{cases}
    $
    \item $v_{S}a_{\beta}
    - m\,v_{S\setminus\{j_1\}}a_{\alpha_{j_1}}
    + m\,v_{S\setminus\{j_2\}}a_{\alpha_{j_2}}$,
    for $j_1\neq j_2\in S$ and  $m\in \Fp$ satisfying $m^{-1}=j_2-j_1$;
    
    \item $v_{S}\kappa_{\beta}
    - m\,v_{S\setminus\{i\}}\kappa_{\alpha_{i}}
    +m\,v_{S\setminus\{j\}}\kappa_{\alpha_{j}}$,
    under the same conditions as above;
  
\end{enumerate}
Moreover, for $S=[i]$, the corresponding class
$v_i := v_S$ is invertible in $\widetilde{H}^{\bigstar}_{\cG}(E\cF_{i+};\uFp)$.
\end{prop}
\begin{remark}
For a singleton subset $\{j\}\subseteq [i]$ we declare $v_{\{j\}} := u_{\alpha_j}$. 
\end{remark}
\begin{proof}[Proof of Proposition~\ref{invertiable}]
We construct the class $v_S$ inductively on $i$.

\medskip
\noindent\textbf{Base case.}
For $i=1$, the subsets of $\{0,1\}$ give rise to classes
\[
v_{\{0\}} := u_{\alpha_0}, \qquad 
v_{\{1\}} := u_{\alpha_1},
\]
and a class $v_1$ whose image in 
$\widetilde{H}^{\bigstar}_{\cG}(E\cF(H_0)_{+};\uFp)$ is 
$u_{\alpha_0}\,u_{\alpha_1-\beta}$.
All of these live in $\widetilde{H}^{\bigstar}_{\cG}(E\cF_{1+};\uFp)$ with $u_{\alpha_0}u_{\alpha_1-\beta}$ is invertiable as describe in Proposition~\ref{boundary_0}.

\medskip
\noindent\textbf{Inductive step.}
Assume the classes $v_S$ have been constructed for level $i-1\ge 1$; that is,
for every subset
\[
S'=\{i-1 \ge j_k > j_{k-1} > \cdots > j_0 \ge 0\} \subseteq [i-1],
\]
there exists a class $v_S' \in \widetilde{H}^{\bigstar}_{\cG}(E\cF_{(i-1)+};\uFp)$.

Now let
\[
S=\{i \ge l_m > l_{m-1} > \cdots > l_0 \ge 0\} \subseteq [i].
\]
We define $v_S$ as follows:
\[
v_S =
\begin{cases}
\text{the preimage of } v_S \in \widetilde{H}^{\bigstar}_{\cG}(E\cF_{(i-1)+};\uFp),
&\text{if } l_m \neq i,\\[6pt]
u_{\alpha_i - \beta}\, v_{\{l_{m-1},\dots,l_0\}},
&\text{if } l_m = i.
\end{cases}
\]
the latter expression lies in $\Ker(\partial_{i-1}^{\bigstar})$, so the class is
well-defined in $\widetilde{H}^{\bigstar}_{\cG}(E\cF_{i+};\uFp)$.

\medskip

For $S=[i]$, the corresponding element $v_{i}$ is the preimage of class $u_{\alpha_i-\beta}v_{i-1}$. Since $u_{\alpha_i-\beta}v_{i-1}$ and its inverse belong in $\Ker \partial^{\bigstar}_{i-1}$.
Thus $v_i$ is invertible as well.
\end{proof}

In the same way, the next proposition constructs new classes corresponding to a subset of $[i]$.

\begin{prop}\label{wz}
Let $2 \le i \le p-1$, and $Q \subseteq [i]$ be a subset of cardinality 
at least three. Then there exist classes
\[
w_{Q},\; z_{Q} \;\in\; \widetilde{H}^{\bigstar}_{\cG}(E\cF_{i+};\uFp).
\]
\end{prop}
\begin{remark}
In the above proposition, we can define these classes for a set of cardinality two as follows:
\[
\begin{cases}
w_{\{s<t\}}=\kappa_{\alpha_s}\kappa_{\alpha_t}\\
z_{\{s<t\}}=\kappa_{\alpha_s}a_{\alpha_t}-\kappa_{\alpha_t}a_{\alpha_s}
\end{cases}
\]
\end{remark}
\begin{proof}[Proof of Proposition~\ref{wz}]
Choose consecutive elements $s_0 < s_1$ of $Q$ with $s_0 = \min(Q)$.  
Consider the classes in 
$\widetilde{H}^{\bigstar}_{\cG}(E\cG_{+};\uFp)$ given by
\[
\kappa_{\alpha_{s_0}}\kappa_{\alpha_{s_1}}
\prod_{\substack{l\in Q \\ l\ne s_0,s_1}}
\frac{u_{\alpha_l}}{u_{\beta}},
\qquad
\left(\kappa_{\alpha_{s_0}} a_{\alpha_{s_1}}
      - \kappa_{\alpha_{s_1}} a_{\alpha_{s_0}}\right)
\prod_{\substack{l\in Q \\ l\ne s_0,s_1}}
\frac{u_{\alpha_l}}{u_{\beta}}.
\]
Define $w_Q$ and $z_Q$ to be their respective preimages under the canonical map from 
$\widetilde{H}^{\bigstar}_{\cG}(E\cF_{i+};\uFp)$ to $\widetilde{H}^{\bigstar}_{\cG}(E\cG_{+};\uFp)$.
  
The existence of these preimages follows by induction on $i$,  
analogous to the construction of the classes in Proposition~\ref{invertiable}.
The base case occurs at $i=2$, where for $Q=[2]$ the corresponding classes
are the preimages of
\[
\kappa_{\alpha_0}\kappa_{\alpha_1}u_{\alpha_2-\beta},
\qquad
(\kappa_{\alpha_0}a_{\alpha_1} - \kappa_{\alpha_1}a_{\alpha_0})\,u_{\alpha_2-\beta}.
\]
\end{proof}

Combining the newly introduced classes with Proposition~\ref{coker} enables a computation of the kernel of the boundary map $\partial^{\bigstar}_{i}$ for $1\le i\le p-2$, and hence the corresponding Bredon cohomology $\mathrm{P}(\widetilde{H}^{\bigstar}_{\cG}(E{\cF_{i}}_{+};\uFp))$: 
\begin{thm}{\label{main}}
    The polynomial part of the $\RO(\cG)$–graded Bredon cohomology of ${E\cF_{i}}$ is given by
\[
\mathrm{P}(\widetilde{H}^{\bigstar}_{\cG}(E{\cF_{i}}_{+};\uFp))
\;\cong\;
\dfrac{\Fp\big[
a_{\beta},u_{\beta},\kappa_{\beta},
a_{\alpha_\ell},\kappa_{\alpha_\ell},a_{\alpha_j},\kappa_{\alpha_j},
v_{S},z_{Q},w_{Q},v_{i}^{\pm},u^{\pm}_{\alpha_j-\beta}
\colon \ell \in [i], j \in [p-1]\setminus [i]\big]}{\mathfrak{J}_i},
\]
Moreover, for $i=p-1$
\[
\mathrm{P}(\widetilde{H}^{\bigstar}_{\cG}(E_{C_p}\Zp_{+};\uFp))
\;\cong\;
\dfrac{\Fp\big[
a_{\beta},u_{\beta},\kappa_{\beta},
a_{\alpha_\ell},\kappa_{\alpha_\ell}
,v_{S},z_{Q},w_{Q},v_{p-1}^{\pm}
\colon \ell \in [i]\big]}{\mathfrak{J}_{p-1}},
\]

where \(S\) and \(Q\) range over the subsets of 
\([i]\) of cardinality at least one and at least three, respectively. 
$\mathfrak{J}_i$ denotes the relations of classes which can be pulled back from the map $\widetilde{H}^{\bigstar}_{\cG}(E{\cF_{i}}_{+};\uFp) \to \widetilde{H}^{\bigstar}_{\cG}(E\cG_{+};\uFp)$:
\begin{align*}
a_{\alpha_l} &\mapsto (y_{0}-ly_{1})\,u_{\alpha_l}, &
\kappa_{\alpha_l} &\mapsto (t_{0}-lt_{1})\,u_{\alpha_l}, \\
a_{\beta} &\mapsto y_{1}\,u_{\beta}, &
\kappa_{\beta} &\mapsto t_{1}\,u_{\beta}, \\
w_{Q}&\mapsto\kappa_{\alpha_{s_0}}\kappa_{\alpha_{s_1}}
\prod_{\substack{d\in Q\setminus\{s_0,s_1\} }}
\frac{u_{\alpha_d}}{u_{\beta}}, &
z_Q&\mapsto\left(\kappa_{\alpha_{s_0}} a_{\alpha_{s_1}}
      - \kappa_{\alpha_{s_1}} a_{\alpha_{s_0}}\right)
\prod_{\substack{d\in Q\setminus\{s_0,s_1\} }}
\frac{u_{\alpha_d}}{u_{\beta}}
\\
 v_{S}&\mapsto\frac{\prod_{j \in S}u_{\alpha_j}}{u^{|S|-1}_{\beta}}
\end{align*}
\end{thm}
\begin{proof}
The above expression follows from the similar fact used in Proposition~\ref{boundary_0}, that any class lie in $\Ker \partial^{\bigstar}_{i-1}$ if and only if there exists $y\in \widetilde{H}^{\bigstar}_{\cG}(E\cF(H_i)_{+};\uFp)$ such that $\eta^{\bigstar}_{i-1}(x)=q^{\bigstar}_{i}(y)$, which gives $\Ker \partial^{\bigstar}_{i-1}=(\eta^{\bigstar}_{i-1})^{-1}(\Image \eta^{\bigstar}_{i-1} \cap \Image q^{\bigstar}_{i})$. Proposition~\ref{invertiable} and Proposition~\ref{wz} gives the existence of such classes, whereas the relations $\mathfrak{J}_i$ follows from the fact that the map 
    \[\widetilde{H}^{\bigstar}_{\cG}({E\cF_{i}}_{+};\uFp) \to \widetilde{H}^{\bigstar}_{\cG}({E\cG}_{+};\uFp)
    \] 
    is an injective map for the polynomial part $\mathrm{P}(\widetilde{H}^{\bigstar}_{\cG}({E\cF_{i}}_{+};\uFp))$.
\end{proof}

\begin{mysubsect}{Bredon cohomology of $B_{C_p}\Zp$}
The canonical projection
\[
\pi\colon C_p\times C_p \longrightarrow C_p\times C_p / (1 \times C_p) 
\]
induces a pullback homomorphism on representation rings
\[
\pi^\ast\colon \RO(C_p\times C_p/(1 \times C_p)) \longrightarrow \RO(C_p\times C_p).
\]
This allow us to relate cohomology theories along $\pi$:

\begin{prop}\label{res.}
Let $\alpha \in \mathrm{im}(\pi^\ast)$, then there is a natural equivalence
\[
\widetilde{H}^{\pi^\ast(\alpha)}_{C_p\times C_p}(X; \uFp) \;\cong\; \widetilde{H}^\alpha_{C_p}(X/(1 \times C_p); \uFp),
\]
where the right-hand side is the $\RO(C_p)$-graded Bredon cohomology of the quotient $X/(1\times C_p)$.
\end{prop}

\begin{proof}
From the definition of Bredon cohomology via $G$-spectra, we have:
\[
[X, S^{\pi^\ast(\alpha)} \wedge H\uFp]^{C_p\times C_p} 
\cong [X, \pi^\ast(S^\alpha \wedge H\uFp)]^{C_p\times C_p} 
\cong [X/(1\times C_p), S^\alpha \wedge H\uFp]^{C_p}.
\]
This gives the result.
\end{proof}

\vsn

In particular, if we take $\bigstar_{\beta} \in \Z\{1, \beta\} = \RO(\cG/(1\times C_p))$, then for $X = E_{C_p}\Zp_{+}$ using Proposition~\ref{res.} we obtain:
\[
\widetilde{H}^{\bigstar_\beta}_{C_p}({B_{C_p} \Zp}_+; \uFp) \;\cong\; \widetilde{H}^{\bigstar_{\beta}}_{\cG}(E_{C_p}\Zp_{+}; \uFp).
\]
 Since the polynomial part of $\widetilde{H}^{\bigstar}_{\cG}(E_{C_p}\Zp_{+};\uFp)$  is now known explicitly in Theorem~\ref{main}, we may use this description to compute the polynomial part of the Bredon cohomology of $B_{C_p}\Zp$.

\begin{prop}\label{poly-BCpZp}
The polynomial part of the $RO(C_p)$--graded Bredon cohomology of \( B_{C_p}\Zp \) is given by
\[
\mathrm{P}(\widetilde{H}^{\bigstar_{\beta}}_{C_p}(B_{C_p}\Zp_{+};\uFp))
\;=\;
\Fp[a_{\beta},u_{\beta},\kappa_{\beta},\nu,c,b,\eta,\xi,\omega]/\sim,
\]
modulo the ideal generated by the following classes :
\begin{enumerate}
    \item $c^p-u_{\beta}b-a^{p-1}_{\beta}c^{p-1}$
    \item $\nu^2,\xi^2,\eta^2,\omega^2,\kappa^2_{\beta}$
    \item $\nu \xi,\nu \eta,\kappa_{\beta}\xi,\xi\eta$
\end{enumerate}
\end{prop}

\begin{proof}
Define the following classes:
\[
\begin{aligned}
\nu &\,=\, \frac{\kappa_{\alpha_0} v_{\{p-1,p-2,\ldots,1\}}}{v_{p-1}}, 
&\qquad
c &\,=\, \frac{a_{\alpha_0} v_{\{p-1,p-2,\ldots,1\}}}{v_{p-1}}, \\[4pt]
b &\,=\, \frac{a_{\alpha_0} a_{\alpha_1} \cdots a_{\alpha_{p-1}}}{v_{p-1}}, 
&\qquad
\xi &\,=\, \frac{w_{[p-1]}}{v_{p-1}}, \\[4pt]
\eta &\,=\, \frac{z_{[p-1]}}{v_{p-1}}, 
&\qquad
\omega &\,=\, \frac{\kappa_{\alpha_0} a_{\alpha_1} \cdots a_{\alpha_{p-1}}}{v_{p-1}}.
\end{aligned}
\]
The above polynomial algebra generates all $\bigstar_{\beta} \in \mathbb{Z}\{1,\beta\}$--graded classes in
\(
\mathrm{P}(
\widetilde{H}^{\bigstar_{\beta}}_{\cG}(E_{C_p}\Zp_{+};\uFp)),
\)
and hence describes the polynomial part of
\(
\widetilde{H}^{\bigstar_{\beta}}_{C_p}(B_{C_p}\Zp_{+};\uFp).
\)
\end{proof}

\end{mysubsect}

\sect{Bredon cohomology of $E_{C_2}\Z/2$}
    
We now compute the case $p = 2$ in a similar manner. Here $\K$ will denote the Klein-four group $C_2\times C_2$.
In this situation, the nontrivial irreducible representations are denoted by 
$\alpha_0$, $\alpha_1$, and $\beta$, whose kernels are the subgroups 
$H_0 = \langle a \rangle$, $H_1 = \langle ab \rangle$, and $K = \langle b \rangle$, respectively. 
In contrast to the case of odd primes $p$, all these are one-dimensional real representations. 

To study the $\RO(\K):=\Z\{1,\alpha_0,\alpha_1,\beta\}$--graded Bredon cohomology of $E_{C_2}\Z/2$, we make use of the relationship between the universal spaces as described in Lemma~\ref{pbfamily}. Using the long exact sequence associated 
for the families $\mathcal{F}_{0} := \{1, H_0\}$ and $\mathcal{F}_{1} := \{1, H_1\}$ in Lemma~\ref{pbfamily}, 
One obtains the following diagram in which both rows are exact sequences, which is analogous to \eqref{pullp}:

   \begin{myeq}{\label{pull2}}
\xymatrix@C=.5cm
{
\cdots  \widetilde{H}^{\bigstar}_{\K}({E_{C_2}\Zz}_+;\uF) \ar[d]_-{\iota^{\bigstar}_1}\ar[r]^-{j^{\bigstar}_1}  &  \widetilde{H}^{\bigstar}_{\cG}({E\cF(H_0)}_+;\uF) \ar[d]^-{q^{\bigstar}_{0}}\ar[r]^-{\partial^{\bigstar}} &  \widetilde{H}^{\bigstar+1}_{\K}(E\cFH{1};\uF) \ar[d]^{=}  \cdots\\
\cdots \widetilde{H}^{\bigstar}_{\K}(E\cF(H_1)_+;\uF) \ar[r]^-{q^{\bigstar}_1} & \widetilde{H}^{\bigstar}_{\K}({E\K}_+;\uF) \ar[r]^-{\bar{\partial}^{\bigstar}} & \widetilde{H}^{\bigstar+1}_{\K}(E\cFH{1};\uF)  \cdots
}.
\end{myeq}

Proceeding analogously, one obtains the following computations for the 
Bredon cohomology of the spaces $E\K$, $E\mathcal{F}(H_i)$, and 
$E\cFH{i}$:

\begin{prop}\label{resuni}
    The Bredon cohomology of $E\K,E\cF(H_i)$ and $E\cFH{i}$ are given by:
    \begin{enumerate}
        \item \[
        \widetilde{H}^{\bigstar}_{\K}({E\K}_{+};\uF)=\dfrac{\F[a_{\alpha_0},a_{\alpha_1},a_{\beta},u^{\pm}_{\alpha_0},u^{\pm}_{\alpha_1},u^{\pm}_{\beta}]}{(a_{\alpha_0}u_{\alpha_1}u_{\beta}+u_{\alpha_0}a_{\alpha_1}u_{\beta}+u_{\alpha_0}u_{\alpha_1}a_{\beta})}
        \]
        \item For $i,j\in \{0,1\},$
        \[
        \widetilde{H}^{\bigstar}_{\K}(E\cF(H_i)_{+};\uF)=\bM_{2}(\beta)\otimes\F[a_{\alpha_i},u^{\pm}_{\alpha_i},u^{\pm}_{\alpha_j-\beta}\colon j\neq i],
        \]

\item For $i,j\in \{0,1\},$
\begin{align*}
& \widetilde{H}^{\bigstar}_{\K}(E\cFH{i};\uF)\cong \F\left\langle{\Sigma^{-1}\frac{1}{a^{k_1}_{\beta}u^{k_2}_{\beta}}\colon k_1,k_2\geq 1}\right\rangle\\  
&\bigoplus \Sigma^{-1} u_\beta^{-1}\,\F\!\left[
a_{\alpha_i},\, a_{\alpha_j},\, a_\beta,\,
\Bigl(\tfrac{u_{\alpha_j}}{u_\beta}\Bigr)^{\pm},\,
u_{\alpha_i}^{\pm},\, u_\beta^{-1}
\colon j\neq i\right]\mod \Image q^{\bigstar}_{i}.
\end{align*}

    \end{enumerate}
    
\end{prop}

As proved for odd prime $p$ in Proposition~\ref{integerlevel}, one derived $j^{\ast}_{1}$ (in the diagram \eqref{pull2}) is a zero morphism at the integer level. As a consequence, we have the following description of the $\Z$-graded cohomology of $E_{C_2}\Zz$ as follows: 
\begin{prop}
     The integer graded Bredon cohomology of the equivariant universal space is given by:
     \[
 \widetilde{H}^{*}_{\cG}(\bE_{+};\uF)\cong\Sigma^{-1}\frac{a_{\alpha_0}a_{\alpha_1}}{u_{\alpha_0}u_{\alpha_1}}\F[\frac{a_{\beta}}{u_{\beta}},\frac{a_{\alpha_0}}{u_{\alpha_0}},\frac{a_{\alpha_1}}{u_{\alpha_1}}]/\sim \cong \Sigma^{-1}\frac{a_{\alpha_0}a_{\alpha_1}}{u_{\alpha_0}u_{\alpha_1}}\F[\frac{a_{\alpha_0}}{u_{\alpha_0}},\frac{a_{\alpha_1}}{u_{\alpha_1}}]
 \]
 where $\sim$ is induced from the relation $a_{\alpha_0}u_{\alpha_1}u_{\beta}+a_{\alpha_1}u_{\alpha_0}u_{\beta}+a_{\beta_0}u_{\alpha_1}u_{\alpha_0}=0$.
 \end{prop}

The following proposition establishes the existence of an invertible class $v$, 
which plays the role analogous to the class appearing in Proposition~\ref{invertiable}.

\begin{prop}\label{inv}
There exists an invertible class $v \in \widetilde{H}^\bigstar_{\K}(\bE_+; \uF)$ of degree $\bigstar = \alpha_0 + \alpha_1 - \beta - 1$, which maps under $j_1^\bigstar$ to $u_{\alpha_0} u_{\alpha_1 - \beta}$. Moreover, the following identities hold:
\begin{enumerate}
\item $v u_\beta = u_{\alpha_0} u_{\alpha_1}$,
\item $v a_\beta = a_{\alpha_0} u_{\alpha_1} + u_{\alpha_0} a_{\alpha_1}$.
\end{enumerate}
\end{prop}
\begin{proof}
For the degree $\bigstar = k(\alpha_0 + \alpha_1 - 1 - \beta)$, one checks that the class
\[
u_{\alpha_0}^{\,k} u_{\alpha_1-\beta}^{\,k}
\]
lies in $\Ker(\partial^{\bigstar})$ in~\eqref{pull2} for every $k \in \Z$. Consequently, there exists an invertible class $v$, defined to be the preimage of $u_{\alpha_0}u_{\alpha_1-\beta}$ under the map $j^{\bigstar}_1$.

For the relations, observe that $j^{\bigstar}_1$ is an isomorphism in degrees 
$\bigstar = \alpha_0+\alpha_1-2$ and $\bigstar = \alpha_0+\alpha_1-1$, since the cohomology of $E\cFH{1}$ vanishes in the corresponding degrees in~\eqref{pull2}. Under this identification via $j^{\bigstar}_1$, the generators correspond as follows:
\[
u_{\alpha_0} \mapsto u_{\alpha_0}, \qquad
u_{\alpha_1} \mapsto u_{\alpha_1-\beta}u_{\beta}, \qquad
a_{\alpha_0} \mapsto a_{\alpha_0}, \qquad
a_{\alpha_1} \mapsto \left(\frac{a_{\alpha_0}}{u_{\alpha_0}}u_{\beta}-a_{\beta}\right) u_{\alpha_1-\beta}.
\]
\end{proof}

For the computation of $\widetilde{H}^{\bigstar}_{\K}(E_{C_2}\Zz_{+};\uF)$, one can consider the following short exact sequence:
\begin{myeq}\label{sesec}
    0 \to \Coker \partial^{\bigstar-1} \to \widetilde{H}^{\bigstar}_{\K}(\bE_{+};\uF) \to \Ker \partial^{\bigstar} \to 0 
\end{myeq}
as a vector space over $\F$. The following proposition gives both sides objects:

\begin{prop}\label{ker2}
The Kernel and Cokernel of the connecting homomorphism $\partial^\bigstar:\widetilde{H}^{\bigstar}_{\K}(E\cF(H_0)_{+};\uF) \to \widetilde{H}^{\bigstar+1}_{\K}(E\cFH{1};\uF)$ from the short exact sequence in \eqref{sesec} are given as follow:

\begin{enumerate}
    \item The kernel is:
    \begin{align*}
    \Ker(\partial^\bigstar) =\; & 
    \F[a_{\beta}, u_{\beta}, u_{\alpha_0}, u_{\alpha_1}, a_{\alpha_0}, ((a_{\alpha_0}/u_{\alpha_0})u_\beta - a_\beta)u_{\alpha_1 - \beta}, (u_{\alpha_0}u_{\alpha_1 - \beta})^{\pm}] \\
    & \bigoplus\; \F\left\langle \Sigma^{-1}\frac{1}{a^{k_1}_\beta u^{k_2}_\beta} \,\colon\, {k_1},{k_2} \geq 1 \right\rangle [a_{\alpha_0}, u_{\alpha_0}^{\pm}, u_{\alpha_1 - \beta}^{\pm}].
    \end{align*}
    
    \item The cokernel is:
    \begin{align*}
    \Coker(\partial^\bigstar) =\; & 
    \F\left\langle \Sigma^{-1}\frac{1}{a_\beta^j u_\beta^k} \,\colon\, j,k \geq 1 \right\rangle [a_{\alpha_1}, u_{\alpha_1}^{\pm}, u_{\alpha_0 - \beta}^{\pm}] \\
    & \bigoplus\; \Sigma^{-1} u_\beta^{-1}  \F\left[
    a_{\alpha_0}, a_{\alpha_1}, \frac{a_\beta}{u_\beta}, \frac{u_{\alpha_1}}{u_\beta}, \frac{u_{\alpha_0}}{u_\beta}, \left(\frac{u_{\alpha_0}u_{\alpha_1}}{u_\beta}\right)^{\pm}, u_\beta^{-1}
    \right] /\sim
    \end{align*}
\end{enumerate}
where $\sim$ denotes the equivalence relations arising from the image of $q_1^\bigstar$ and the internal structure of $\widetilde{H}^\bigstar_{\K}({E\K}_+; \uF)$.
\end{prop}

\begin{proof}
The proof proceeds in the same manner as in Proposition~\ref{boundary_0}.
 \end{proof}

\vsn

Using the square-zero extension structure of the ring, we arrive at the complete description of the cohomology of $\bE$:

\begin{thm}\label{main2}
There is an isomorphism of graded rings:
\[
\widetilde{H}^{\bigstar}_{\cG}(\bE_+; \uF) \cong 
\frac{
    \left[ \F[a_\beta, u_\beta] 
    \oplus 
    \F\langle{\Sigma^{-1}\frac{1}{a^j_{\beta}u^k_{\beta}}:j,k\geq 1}\rangle    \right]
    [a_{\alpha_0}, u_{\alpha_0}, a_{\alpha_1}, u_{\alpha_1}, v^{\pm}]
}{
    \left( vu_\beta - u_{\alpha_0} u_{\alpha_1},\; va_\beta - (a_{\alpha_0}u_{\alpha_1} + u_{\alpha_0}a_{\alpha_1}) \right)
}
\]
where $v$ is the invertible class in degree $ \alpha_0 + \alpha_1 - \beta - 1$, as constructed in Proposition~\ref{inv}.
\end{thm}

\begin{proof}
The proof proceeds by showing explicit isomorphisms of $\Ker(\partial^\bigstar)$ and $\Coker(\partial^\bigstar)$ as module over the ring 
$\F[a_{\alpha_0}, a_{\alpha_1}, u_{\alpha_0}, u_{\alpha_1}, v^{\pm}]$ via the identifications:
\begin{enumerate}
    \item For the kernel:
    \[
    \F\left\langle
        \frac{u_{\alpha_1}^{ \ell}}{v^\ell} \cdot \Sigma^{-1}\frac{1}{a_\beta^j u_\beta^{k +  \ell}} 
    \right\rangle [a_{\alpha_0}, u_{\alpha_0}, v^{\pm}]
    \oplus 
    \F[a_\beta, u_\beta, a_{\alpha_0}, a_{\alpha_1}, u_{\alpha_0}, u_{\alpha_1}, v^{\pm}].
    \]
    for $\ell\geq 0$ and $j,k\geq 1$.
    
    This follows from a basis-level correspondence under $\partial^\bigstar$ via the relations:
    \[
    \Sigma^{-1} \frac{1}{a_\beta^j u_\beta^k} \mapsto \Sigma^{-1} \frac{1}{a_\beta^j u_\beta^k},\quad
    u^a_{\alpha_1} \cdot \Sigma^{-1} \frac{1}{a_\beta^j u_\beta^{k+a}} \mapsto u^a_{\alpha_1-\beta} \cdot \Sigma^{-1} \frac{1}{a_\beta^j u_\beta^k},\quad
     v^a \cdot \Sigma^{-1} \frac{1}{a_\beta^j u_\beta^k} \mapsto u^{a}_{\alpha_0}u^{a}_{\alpha_1-\beta} \cdot \Sigma^{-1} \frac{1}{a_\beta^j u_\beta^k}.
    \]

    \item For the cokernel:
    \begin{align*}
        \F\left\langle
            \frac{u^k_{\alpha_1}a^{k+a}_{\alpha_0} }{v^{k+a}}  \cdot \Sigma^{-1}\frac{1}{a_\beta^{j +a+ k} u_\beta^{k}} 
        \right\rangle [a_{\alpha_1}, u_{\alpha_1}, v^{\pm}] 
        \\
        \oplus\;
        \F\left\langle 
            u_{\alpha_0}^a u_{\alpha_1}^b \left( \frac{u_{\alpha_1} a_{\alpha_0}}{v} \right)^{a+b+m+n} \cdot \Sigma^{-1} \frac{1}{a_\beta^{n+a+b} u_\beta^{n+m+a+b}} 
        \right\rangle [a_{\alpha_0}, a_{\alpha_1}, v^{\pm}].
    \end{align*}
    for $j,k,n\geq 1$ and $a,b,m\geq 0$.
    
    This follows from similar correspondence:
    \begin{align*}
    \Sigma^{-1} \frac{1}{a_\beta^j u_\beta^k} \mapsto \left( \frac{u_{\alpha_1}a_{\alpha_0}}{v} \right)^k \cdot \Sigma^{-1} \frac{1}{a_\beta^{j+k} u_\beta^k}, \\
    \Sigma^{-1} u^{a}_{\alpha_0}u^{b}_{\alpha_1}u_\beta^{-(a+b+n+m)} a_\beta^m \mapsto u^{a}_{\alpha_0}u^{b}_{\alpha_1}\left( \frac{u_{\alpha_1} a_{\alpha_0}}{v} \right)^{m+n+a+b} \cdot \Sigma^{-1} \frac{1}{a_\beta^{n+a+b} u_\beta^{m+n+a+b}}    
    \end{align*}
    
\end{enumerate}
\end{proof}

\vsn

Using Proposition~\ref{res.}, we obtain an analogue of Proposition~\ref{poly-BCpZp} for the case $p = 2$.  
In this situation, all classes of degree $\bigstar_{\beta} \in \mathbb{Z}\{1,\beta\}$ are explicitly determined.

\begin{corollary}{\label{cohomologyofclassfying}}
There is an isomorphism of $\RO(C_2)$-graded $\F$-algebras:
\[
\widetilde{H}^{\bigstar_{\beta}}_{C_2}(B_{C_2} \Z/2_+; \uF) \;\cong\;
\frac{
    \left[ \F[a_\beta, u_\beta] 
    \oplus 
   \bigoplus_{j,k\geq 1}\F \langle{\Sigma^{-1}\frac{1}{a^j_{\beta}u^k_{\beta}}}\rangle
    \right][c, b]
}{
    (c^2 = a_\beta c + u_\beta b)
}.
\]
\end{corollary}

\begin{proof}
In $\widetilde{H}^\bigstar_{\cG}(\bE_+; \uF)$ we have:
\[
a_\beta = v^{-1}(u_{\alpha_0} a_{\alpha_1} + a_{\alpha_0} u_{\alpha_1}), \quad
u_\beta = v^{-1} u_{\alpha_0} u_{\alpha_1}.
\]
Now define
\[
c := v^{-1} u_{\alpha_0} a_{\alpha_1}, \quad
b := v^{-1} a_{\alpha_0} a_{\alpha_1}.
\]
Then a direct computation yields:
\[
c^2 = (v^{-1} u_{\alpha_0} a_{\alpha_1})^2 = v^{-2} u_{\alpha_0}^2 a_{\alpha_1}^2,
\]
while
\[
a_\beta c + u_\beta b = v^{-1}(u_{\alpha_0} a_{\alpha_1} + a_{\alpha_0} u_{\alpha_1}) \cdot v^{-1} u_{\alpha_0} a_{\alpha_1}
+ v^{-1} u_{\alpha_0} u_{\alpha_1} \cdot v^{-1} a_{\alpha_0} a_{\alpha_1}.
\]
Both sides reduce to the same expression, so the claimed relation holds.
\end{proof}

\sect{The ring structure of $\widetilde{H}_{C_p\times C_p }^\bs(S^0; \uFp)$}

In this section, we compute the $\RO(\cG)$--graded Bredon cohomology of a point with coefficients in the constant Mackey functor $\uFp$.  
Our approach follows the \emph{Tate square} formalism~\cite{GM95}, as input from the cohomology of the equivariant universal space $E_{C_p}\Zp$ from Theorem~\ref{main} and Theorem~\ref{main2}.  
We begin with the case of odd primes $p$, where Theorem~\ref{main} describes the polynomial part $\mathrm{P}(\widetilde{H}^{\bigstar}_{\cG}(E_{C_p}\Zp_{+};\uFp))$ of the ring $\widetilde{H}^{\bigstar}_{\cG}(E_{C_p}\Zp_{+};\uFp)$.  
For $p = 2$, we subsequently give an explicit description of all the classes in the cohomology of a point.

\begin{mysubsect}{For an odd prime $p$} 

The following is a homotopy pullback diagram of $\cG$–spectra:

\begin{myeq}\label{tatep}
\xymatrix@C=0.4cm{
  E_{C_p}\Zp_+ \wedge H\uFp \ar[r]^-{f} \ar[d]_\simeq 
    & H\uFp \ar[d] \ar[r] 
    & \widetilde{E_{C_p}\Zp} \wedge H\uFp \ar[d] \\
  E_{C_p}\Zp_+ \wedge \Fun(E_{C_p}\Zp_+, H\uFp) \ar[r] 
    & \Fun(E_{C_p}\Zp_+, H\uFp) \ar[r]^-{\mathcal{L}} 
    & \widetilde{E_{C_p}\Zp} \wedge \Fun(E_{C_p}\Zp_+, H\uFp)
}
\end{myeq}
It is known as \emph{Tate} diagram. The same approach has been employed in a variety of computations, some of the
most significant examples are $\pi_{\bigstar}H\uZ$ for $G=C_2$ in \cite{Gre17}, $\pi_{\bigstar}H\uZp$ for $G=C_p$ in \cite{BG21}.

Here, the left vertical map is an equivalence of $\cG$–spectra \cite{GM95} and the bottom row is obtained by applying the function spectrum $\Fun(E_{C_p}\Zp_+, -)$ to the top row~\eqref{tatep}.
So the homotopy groups of $E_{C_p}\Zp_+ \wedge H\uFp$ can be computed as the homotopy pullback of the bottom row. 

 We start our computation from $\Fun(E_{C_p}\Zp_{+},H\uFp)$, $\pi_{\bigstar}\Fun(E_{C_p}\Zp_{+},\uFp)$ is given by Theorem~\ref{main}, along with non-polynomial classes which are denoted by the  factor $R$ 
    \begin{myeq}
\pi_{\bigstar}\Fun(E_{C_p}\Zp_{+},\uFp)= \dfrac{\Fp\big[
a_{\beta},u_{\beta},\kappa_{\beta},
a_{\alpha_\ell},\kappa_{\alpha_\ell}
,v_{S},z_{Q},w_{Q},v_{p-1}^{\pm}
\mid \ell \in [i]\big]}{\mathfrak{J}_{p-1}}\bigoplus R.
    \end{myeq}

Since \(
\widetilde{E_{C_p}\Zp}\simeq S^{\infty V}
\) for $V= \bigoplus_{0\leq r \leq p-1}\alpha_r$, 
\[
\pi_{\bigstar}\big(\widetilde{E_{C_p}\Zp}\wedge \Fun(E_{C_p}\Zp_{+},H\uFp)\big)
 \;=\; 
\pi_{\bigstar}\Fun(E_{C_p}\Zp_{+},H\uFp)\big[a^{\pm1}_{\alpha_r} \;:\; 0\le r\le p-1\big].
\]

 Here \(\mathcal{L}_{\bigstar}\) is localisation at the multiplicative set generated by the classes 
\(\{a_{\alpha_r}\mid 0\le r\le p-1\}\cup\{1\}\). Consequently, those elements of 
\(
\pi_{\bigstar}\big(\widetilde{E_{C_p}\Zp}\wedge \Fun(E_{C_p}\Zp_{+},H\uFp)\big)
\)
which do not lie in the image of \(\mathcal{L}_{\bigstar}\) are precisely those containing at least one factor \(a_{\alpha_r}\) in the denominator.  

 Using the equivalence of the left vertical map in the Tate square, we obtain the following:
\begin{prop}
    There is an isomorphism
    \begin{align*}
\pi_{\bigstar}\big(E_{C_p}\Zp_{+}\wedge H\uFp\big)
\;&\cong\;
\Bigg[\frac{\Fp[a_{\beta},\kappa_{\beta},u_{\beta}]}{(\kappa_{\beta}^{2})}
\;\otimes\;
\bigoplus_{1\le i\le p-1} Q_i\Bigg] 
\quad\text{(modulo $\mathrm{Im}(\ell_{\bigstar})$)}\\
&\bigoplus  R_{C} \bigoplus \Ker \mathcal{L}_{\bigstar}
    \end{align*}

where each \(Q_i\) is described by
\[
Q_{i}
\;=\;
\bigoplus_{\substack{|S_i| = i \\ S_i \subset [p-1]}}
\Fp\Big\langle 
\Sigma^{-1}\!\left(\frac{1}{\prod_{s\in S_i} a_{\alpha_s}}\right)
\Big\rangle
\big[
a_{\alpha_{n}},\,
a^{-1}_{\alpha_j},\,
\kappa_{\alpha_\ell},\,
v_{S},\,
z_{Q},\,
w_{Q},\,
v_{p-1}^{\pm}
\big],
\]
with
\[
j\in S_i, \qquad 
n\in [p-1]\setminus S_i, \qquad 
0\le \ell\le p-1,
\]
and \(S\) and \(Q\) range over the nonempty subsets of \([p-1]\), with \(Q\) required to have cardinality at least three and $R_{C}$ denote the cokernel of $\mathcal{L}_{\bigstar}$ corresponding to the summand $R$. 
\end{prop}

 Again, in order to determine  
\(
\pi_{\bigstar}\big(\widetilde{E_{C_p}\Zp}\wedge H\uFp\big),
\)
It suffices to compute it in degrees 
\(
\bigstar_{\beta}\in \mathbb{Z}\{1,\beta\}.
\)
As a consequence of the upper cofiber sequence in~\eqref{tatep}, the computation of the map
\[
f_{\bigstar_{\beta}} \colon 
\pi_{\bigstar_{\beta}}\!\left(E_{C_p}\Zp_{+}\wedge H\uFp\right) 
\longrightarrow 
\pi_{\bigstar_{\beta}} H\uFp
\]
provides the required information and yields the desired result, where the expression of the domain ring can be obtained by separating all the $\bigstar_\beta$--degree classes from the above expression and 
working modulo the relations. Since $\Ker\mathcal{L}_{\bigstar_{\beta}}\subset \widetilde{H}^{\bigstar_{\beta}}_{\cG}(E_{C_p}\Zp_{+};\uFp)\cong \widetilde{H}^{\bigstar_{\beta}}_{C_p}(B_{C_p}\Zp_{+};\uFp)$, where the $\RO(C_p)$--graded Bredon cohomology of $B_{C_p}\Zp$ is given as
\begin{prop}\cite[Proposition 6]{HK22}{\label{FullBC_p}}
For an odd prime $p$, the Bredon cohomology $\widetilde{H}^{\bigstar_{\beta}}_{C_{p}}(B_{C_p}\Zp_{+};\uFp)$ is isomorphic to 
\begin{align*}
    \left[\frac{\Fp[a_{\beta},u_{\beta},\kappa_{\beta}]}{(\kappa^{2}_{\beta})} \bigoplus_{j_1,j_2 >0}\Fp\langle\Sigma^{-1}\frac{\kappa^{\epsilon}_{\beta}}{a^{j_1}_{\beta}u^{j_2}_{\beta}}\rangle\right][b]\bigotimes \Lambda_{\Fp}[\omega] 
    \bigoplus  \Biggl[\Sigma^{\beta-2}\Biggl[\Fp\langle{u_{\beta},\kappa_{\beta}}\rangle[a_{\beta},u_{\beta}]\\  \bigoplus_{j_1,j_2 >0}\Fp\langle\Sigma^{-1}\frac{\kappa^{\epsilon}_{\beta}}{a^{j_1}_{\beta}u^{j_2}_{\beta}}\rangle\bigoplus\Sigma^{-1}\frac{1}{a_{\beta}}\Fp[\frac{1}{a_{\beta}}] \Biggr][c,b]\Biggr]/(u_{\beta}b-c^p+a^{p-1}_{\beta}c) \bigotimes \Fp\langle c,\nu\rangle
\end{align*}
     where $|b|=(p-1)\beta+2,$ $|\omega|=(p-1)\beta+1,$ $|c|=\beta$ and $|\nu|=\beta-1$.
\end{prop}

\begin{prop}{\label{kernel}}
    The kernel of the localisation map $\mathcal{L}_{\bigstar}$, when restricted to the $\bigstar_\beta\in\Z\{1,\beta\}$ degree is trivial, i.e.,
    \[
    \Ker \mathcal{L}_{\bigstar_\beta}=0.
    \]
\end{prop}
\begin{proof}
  Note that from the equivalence of the the left vertical arrow in the Tate square \eqref{tatep}, we may identify \(\pi_{\bigstar}E_{C_p}\Zp_{+}\wedge H\uFp \cong \Ker \mathcal{L}_{\bigstar} \oplus \Sigma^{-1}\Coker \mathcal{L}_{\bs+1}. \) Thus the  composition 
  \[
  \Ker \mathcal{L}_{\bigstar} \hookrightarrow \pi_{\bigstar}E_{C_p}\Zp_{+}\wedge H\uFp \to \pi_{\bigstar}\Fun(E_{C_p}\Zp_{+}, H\uFp)
  \]
  is injective.  In particular, if
$x\in \Ker \mathcal{L}_{\bigstar_\beta}$ is nonzero, then its image
$f_{\bigstar_\beta}(x)$ in
$\pi_{\bigstar_\beta}\Fun(E_{C_p}\Zp_{+},H\uFp)$
is also nonzero.

Suppose first that $f_{\bigstar_\beta}(x)$ is represented by a
polynomial class.  Since the middle vertical map in
\eqref{pullp} acts as the identity on polynomial classes, the image of
$f_{\bigstar_\beta}(x)$ under localisation is nonzero.  This contradicts
the commutativity of the left square in \eqref{pullp}, as $x$ lies in the
kernel of $\mathcal{L}_{\bigstar_\beta}$.  Hence this case cannot occur.

It remains to consider the case in which
$f_{\bigstar_\beta}(x)$ is a nonpolynomial class of the form \(
\Sigma^{-1}\frac{\kappa_\beta^{\epsilon}}{a_\beta^{j}u_\beta^{k}}.
\) By Proposition~\ref{FullBC_p}, such a class can map only to one of the
following elements:
\begin{align*}
\Bigl\{
& b^{\ell}\omega^{\epsilon_1}\Sigma^{-1}
  \frac{\kappa_\beta^{\epsilon_0}}{a_\beta^{j_1}u_\beta^{k_1}},\;
  c^{n}b^{\ell}\Sigma^{\beta-3}
  \frac{\kappa_\beta^{\epsilon_0}}{a_\beta^{j_1}u_\beta^{k_1}}, \\
& \nu c^{m}b^{\ell}\Sigma^{\beta-3}
  \frac{\kappa_\beta^{\epsilon_0}}{a_\beta^{j_1}u_\beta^{k_1}},\;
  c^{n}b^{\ell}\Sigma^{\beta-3}\frac{1}{a_\beta^{j_1}},\;
  \nu c^{m}b^{\ell}\Sigma^{\beta-3}\frac{1}{a_\beta^{j_1}}
\Bigr\},
\end{align*}
where $\epsilon_0,\epsilon_1\in\{0,1\}$, $n,j_1,k_1\ge 1$, and
$\ell,m\ge 0$.
 If the vertical map sends
$\Sigma^{-1}\frac{\kappa_\beta^{\epsilon}}{a_\beta^{j}u_\beta^{k}}$
to itself, then localisation again produces a nonzero class, contradicting
$x\in\Ker \mathcal{L}_{\bigstar_\beta}$.  For all remaining possibilities,
degree considerations force $j_1>j$.  Using the multiplicative structure,
this implies that the vertical map would send the zero class to a
nonzero class, which is impossible.

Thus no nonzero element of $\Ker \mathcal{L}_{\bigstar_\beta}$ can exist, and
the claim follows.
\end{proof}

Hence the expression of $\pi_{\bigstar_{\beta}}\!\left(E_{C_p}\Zp{}_{+}\wedge H\uFp\right)$ is given by
\begin{align*}
\pi_{\bigstar_{\beta}}\!\left(E_{C_p}\Zp{}_{+}\wedge H\uFp\right)
\;&=\;
\Bigg[\dfrac{\Fp[a_{\beta},\kappa_{\beta},u_{\beta}]}{(\kappa_{\beta}^{2})}
\;\otimes\;
\bigoplus_{1\le i \le p-1} R_{i}\Bigg] \mod(\Image \mathcal{L}_{\bigstar}) \bigoplus R_{C_{\beta}} 
\end{align*}
where $R_{i}$ is given by
\begin{align*}
R_{i}
&=
\bigoplus_{\substack{|S_i| = i \\ S_i \subset [p-1]}}
\Bigg[
\Fp\Bigg\langle
\Sigma^{-1}\frac{v_{S_i}}{\prod_{s\in S_i} a_{\alpha_s}},\;
\Sigma^{-1}\frac{\kappa_{\alpha_{m_1}}\, v_{S_i\setminus\{m_1\}}}{\prod_{s\in S_i} a_{\alpha_s}},\;
\Sigma^{-1}\frac{\kappa_{\alpha_{m_1}}\kappa_{\alpha_{m_2}}\, v_{S_i\setminus\{m_1,m_2\}}}{\prod_{s\in S_i} a_{\alpha_s}},\\[4pt]
&\qquad\qquad
\Sigma^{-1}\frac{w_{L_i}\, v_{S_i\setminus L_i}}{\prod_{s\in S_i} a_{\alpha_s}},\;
\Sigma^{-1}\frac{z_{L_i}\, v_{S_i\setminus L_i}}{\prod_{s\in S_i} a_{\alpha_s}},\;
\Sigma^{-1}\frac{\kappa_{\alpha_{m_1}}\, z_{D_i}\, v_{S_i\setminus(D_i\cup\{m_1\})}}{\prod_{s\in S_i} a_{\alpha_s}}
\Bigg\rangle
\Bigg]
\\[6pt]
&\qquad\qquad\otimes\;
\Fp\Bigg[
\frac{v_{E_i}}{\prod_{s\in E_i} a_{\alpha_s}},\;
\frac{\kappa_{\alpha_j}}{a_{\alpha_j}},\;
\frac{w_{T_i}}{\prod_{s\in T_i} a_{\alpha_s}},\;
\frac{z_{T_i}}{\prod_{s\in T_i} a_{\alpha_s}}
\Bigg],
\end{align*}
where 
\[
L_i, D_i, E_i, T_i \subset S_i,\qquad 
|L_i|, |D_i|, |T_i| \ge 3,\qquad 
j,m_1, m_2 \in S_i \text{ with } m_1 \notin D_i ,
\]
$R_{C_{\beta}}$ is the corresponding part of $R_{C}$ obtained by seprating out the classes in $\bigstar_{\beta}$ degrees.
\begin{remark}
\begin{enumerate}
    \item Although additional classes can appear in intermediate expressions, they can all be shown using the relations to coincide with the classes listed above.
    \item For $R_1$ and $R_2$, fewer terms occur, since the corresponding sets contain no subsets of cardinality at least three.
\end{enumerate}
\end{remark}

\item \begin{prop}{\label{fmap}}
The induced ring homomorphism 
$f_{\bigstar_{\beta}}$
is trivial, where using Proposition~\ref{res.}, the codomain is
\[
\pi_{\bigstar_{\beta}} H\uFp \;=\; \frac{\Fp[a_{\beta},\kappa_{\beta},u_{\beta}]}{(\kappa_{\beta}^{2})}\bigoplus_{j,k\geq 1} \Fp\{\Sigma^{-1}\frac{\kappa^{\epsilon}_{\beta}}{a^{j}_{\beta}u^{k}_{\beta}}\}
\]
\end{prop}

\begin{proof} 
$f_{\bigstar_{\beta}}$ is trivial as no class can hit $\pi_{\bigstar_{\beta}}H\uFp$ as through the middle vertical map~\eqref{tatep} all the classes map nontrivially, but all the classes in $\pi_{\bigstar_{\beta}}E_{C_p}\Zp_{+}\wedge H\uFp$ come from $\Coker \mathcal{L}_{\bigstar_{\beta}+1}$, hence it contradict the commutativity of the left square in \eqref{tatep}. 
\end{proof}

As a consequence of the above proposition and the top cofiber sequence in the Tate square \eqref{tatep}, we obtain
\[
\pi_{\bigstar}\bigl(\widetilde{E_{C_p}\Zp}\wedge H\uFp\bigr)
   =
   \Biggl[
   \pi_{\bigstar_{\beta}}H\uFp
      \;\bigoplus\;
      \Sigma\, \pi_{\bigstar_{\beta}-1}\bigl(E_{C_p}\Zp_{+}\wedge H\uFp\bigr)
   \Biggr]
   \;\otimes\;
   \Fp[a_{\alpha_r}^{\pm1} \mid 0\le r \le p-1].
\]

 In the next few propositions, we construct explicit new classes in
\(\Ker(\partial_{\bigstar})\), analogous to the classes described in 
Propositions~\ref{invertiable} and~\ref{wz}.

First, we prove the existence of the class $k^{(i)}_{S_i}$, which is analogous to the previously defined class $v_{S}$ in Proposition~\ref{invertiable} for the subgroup $H_i$. 
\begin{prop}{\label{kclass}}
Let $0\leq i \leq p-1$, for each order subset $S_{i}$ of $[p]\setminus \{i\}$ of having cardinality atleast two, there exists a class $k^{(i)}_{S_{i}}$ in $\widetilde{H}^{\bigstar}_{\cG}(S^0;\uFp)$ with the following property
 \begin{myeq}{\label{kproperty}}
k^{(i)}_{S_i}u^{|S_i|-1}_{\alpha_i}-\prod_{j\in S_i}u_{\alpha_j}=0
    \end{myeq}

Moreover, it acts on $\widetilde{H}^{\bigstar}_{\cG}(E_{C_p}\Zp_{+};\uFp)$ by the multiplication of class $ v_{S_i\setminus\{p\}}
    \left(\frac{v_{[p-1]\setminus\{i\}}}{v_{p-1}}\right)^{|S_i|-1}$.
   
\end{prop}
\begin{remark}
    For convenience, we use $\alpha_p$ for $\beta$ representation.
\end{remark}
\begin{proof}
    Here, we prove the existence for the case $i=0$; the other instances follow by the same argument.

\medskip
\noindent\textbf{Existence.}
Let \( S_0\) be any nonempty subset of $[p]\setminus\{0\}$. If $p\notin S_0$, then consider the following degree
\[
\Sigma_{j\in S_0}\alpha_{j}
   -2-(|S_0|-1)\alpha_{0}
   =
   \Sigma_{j\in S_0}\alpha_{j}
   -2-(|S_0|-1)\beta
   +(|S_0|-1)(\beta-\alpha_{0}),
\]
 Consider the following class
\begin{myeq}\label{inver}
    v_{S_0\setminus\{p\}}
    \left(\frac{v_{[p-1]\setminus\{0\}}}{v_{p-1}}\right)^{|S_0|-1}.
\end{myeq}

\medskip

We now show that there exists a class in 
\(\pi_{\bigstar}\bigl(\widetilde{E_{C_p}\Zp}\wedge H\uFp\bigr)\)
that maps, under the boundary map
\(\partial_{\bigstar}\),
to the \emph{negative suspension} of the class represented in \eqref{inver}.  
Since \eqref{inver} lies in \(\Image(\mathcal{L}_{\bigstar})\), it is zero in  
\(\pi_{\bigstar}(E_{C_p}\Zp_{+}\wedge H\uFp)\); hence it has a representative in \(\Ker(\partial_{\bigstar})\).
  
Consider the following class in \(\bigstar_{\beta}\in \Z\{1,\beta\}\)degree, obtained using the invertibility of the classes 
\(a_{\alpha_r}\) for \(0\le r\le p-1\):
\[
\frac{v_{S_0\setminus\{p\}}}{\prod_{j\in S_0\setminus\{p\}} a_{\alpha_j}}
\left(
   \frac{a_{\alpha_0}\,v_{[p-1]\setminus\{0\}}}{v_{p-1}}
\right)^{|S_0|-1}
=
\frac{v_{S_0\setminus\{p\}}}{v_{p-1}^{\,|S_0|-1}\prod_{j\in S_0\setminus\{p\}} a_{\alpha_j}}
\left(
   a_{\alpha_l}v_{[p-1]\setminus\{\ell\}}
   + \ell\, a_{\beta} v_{p-1}
\right)^{|S_0|-1},
\]
where \(\ell\in S_0\setminus\{p\}\). 
We used the relation
\[
v_{S}\,v_{[p-1]\setminus\{\ell\}}
   = v_{S\setminus\{\ell\}}\, v_{p-1}.
\]

By repeatedly applying the same substitution, one eliminates all factors of \(v_{p-1}\) from the denominator.  
Hence, the resulting class has a representation in
\[
\pi_{\bigstar_{\beta}}\bigl(\widetilde{E_{C_p}\Zp}\wedge H\uFp\bigr).
\]

Thus, the construction ensures that the class in \eqref{inver} lie in $\pi_{\bigstar}\bigl(\widetilde{E_{C_p}\Zp}\wedge H\uFp\bigr)$ and map to zero by \(\partial_{\bigstar}\), and therefore yields a class
\[
k^{(0)}_{S_0}\in \widetilde{H}^{\bigstar}_{\cG}(S^0;\uFp)
\]
The relation holds from the exactness of the top cofiber sequence~\eqref{tatep}. The action of the class follows from the definition of class along with the above relation.
\end{proof}

Similarly, there are classes analogous to $z_{S}$ and $w_{S}$ for the subgroup $H_i$.
\begin{prop}{\label{phiclass}}
For $0 \leq i \leq p-1$, there exist classes 
\[
\psi^{(i)}_{T_i},\; \phi^{(i)}_{T_i} \in \Ker(\partial_{\bigstar}),
\]
where $T_i$ denotes an subset of $[p]\setminus\{i\}$ of cardinality at least $3$ which satisfy the following relations
 \begin{myeq}{\label{phiproperty}}
\psi^{(i)}_{T_i}u^{|T_i|-2}_{\alpha_i}-\kappa_{\alpha_{s_0}}\kappa_{\alpha_{s_1}}\prod_{j\in {T_i}\setminus{\{s_0,s_1\}}}u_{\alpha_j}=0, \quad \phi^{(i)}_{T_i}u^{|T_i|-2}_{\alpha_i}-(\kappa_{\alpha_{s_0}}a_{\alpha_{s_1}}-\kappa_{\alpha_{s_1}}a_{\alpha_{s_0}})\prod_{j\in {T_i}\setminus{\{s_0,s_1\}}}u_{\alpha_j}=0.
    \end{myeq}
Moreover, the action of classes $\psi^{(i)}_{T_i}$ and $\phi^{(i)}_{T_i}$ on $\widetilde{H}^{\bigstar}_{\cG}(E_{C_p}\Zp_{+};\uFp)$ is given by the multiplication of classes $z_{T_i\setminus\{p\}}
    \left(\frac{v_{[p-1]\setminus\{i\}}}{v_{p-1}}\right)^{|T_i|-2}$ and $w_{T_i\setminus\{p\}}
    \left(\frac{v_{[p-1]\setminus\{i\}}}{v_{p-1}}\right)^{|T_i|-2}$ respectively. 
\end{prop}

\begin{proof}
Again, we establish the existence of the claimed classes for $i=0$; the remaining cases follow by symmetry.

\noindent 
For the existence of $\psi^{(0)}_{T_0}$: 
let $T_0\subseteq \{1,2,\ldots,p\}$ be such a subset. We seek a class in degree
\[
\sum_{\substack{j \in T_0}}\alpha_j - 2 - (|T_0|-2)\alpha_0
  = \sum_{\substack{j \in T_0}}\alpha_j - 2 - (|T_0|-2)\beta 
    - (|T_0|-2)(\alpha_0-\beta).
\]
Following the strategy of the preceding proposition, we aim to construct a class in 
$\pi_{\bigstar}(\widetilde{E_{C_p}\Zp}\wedge H\uFp)$ whose boundary under $\partial_{\bigstar}$ is the \emph{negative suspension} of class
\begin{myeq}\label{wclass}
   w_{T_0\setminus\{p\}}
    \left(\frac{v_{[p-1]\setminus\{0\}}}{v_{p-1}}\right)^{|T_0|-2},
\end{myeq}
for $s\neq t\in T_0$.

Consider now the corresponding element in 
$\pi_{\bigstar_{\beta}}(\widetilde{E_{C_p}\Zp}\wedge H\uFp)$, obtained using the invertibility of 
$a_{\alpha_r}$ for $0\le r\le p-1$:
\begin{myeq}\label{psi-construction}
    \frac{w_{T_0\setminus\{p\}}}{\prod_{j\in S_0\setminus\{p\}}a_{\alpha_j}}
    \left(
        a_{\alpha_0}\frac{v_{[p-1]\setminus\{0\}}}{v_{p-1}}
    \right)^{|T_0|-2}
    =
    \frac{w_{T_0\setminus\{p\}}}{
        v_{p-1}^{|T_0|-2}
        \prod_{j\in S_0\setminus\{p\}}a_{\alpha_j}
    }
    \big(
        a_{\alpha_\ell}v_{[p-1]\setminus\{\ell\}}
        + \ell\,a_{\beta}v_{p-1}
    \big)^{|T_0|-2},
\end{myeq}
for some $\ell\in S_0\setminus\{p\}$.
Using the relation
\[
w_{T_0\setminus\{p\}}v_{[p-1]\setminus\{\ell\}}
    = cw_{T_0\setminus\{\ell,p\}}v_{p-1},
\]
where $c\in \Fp$ be some scalar depending on the minimum of $T_{0}\setminus\{p\}$ and $T_{0}\setminus\{\ell,p\}$, and iterating this process, one can eliminate all occurrences of $v_{p-1}$ in the denominator. Hence class in~\eqref{psi-construction} live in $\pi_{\bigstar_{\beta}}(\widetilde{E_{C_p}\Zp}\wedge H\uFp)$ upto relations. Consequently, the resulting class in~\eqref{wclass} lies in 
$\pi_{\bigstar}(\widetilde{E_{C_p}\Zp}\wedge H\uFp)$. But by the boundary map $\partial_{\bigstar}$, it maps to zero (as class in~\eqref{wclass} lies in $\Image \mathcal{L}_{\bigstar}$). Hence, there exists a corresponding class 
\[
\psi^{(0)}_{T_0}\in \Ker \partial_{\bigstar}
\]
representing an element of $\Ker(\partial_{\bigstar})$.

An analogous argument establishes the existence of 
$\phi^{(0)}_{T_0}$.

Again, the relations hold due to the exactness of the top cofiber sequence~\eqref{tatep}, and the action holds due to relations with the definition of class.
\end{proof}

From the top cofiber sequence in \eqref{tatep}, we can conclude the expression of some special classes from $\Ker \partial_{\bigstar}$,  which in turn gives the polynomial part $\mathrm{P}(\widetilde{H}^{\bigstar}_{\cG}(S^0;\uFp))$ of cohomology of a point in $\uFp$--coefficient.

\begin{thm}{\label{cohomologyofpointoddp}}
    For an odd prime $p$, the polynomial part of the cohomology of a point in the coefficient $\uFp$ is given by
    \[
    \mathrm{P}(\widetilde{H}^{\bigstar}_{\cG}(S^0;\uFp))=\Fp[a_{\beta},\kappa_{\beta},u_{\beta},a_{\alpha_i},\kappa_{\alpha_{i}},u_{\alpha_i},w_{T},z_{T},v_{S},k^{(i)}_{S_{i}},\phi^{(i)}_{T_i},\psi^{(i)}_{T_i}]/\sim
    \]
    where $S$ and $T$ are the subset of $[p-1]$ having cardianlity atleast two and three respectively, and for $0\leq i \leq p-1$, $S_{i}$ and $T_{i}$ are the subset of $[p]\setminus \{i\}$ of having cardinality atleast two and three respectively, where $\sim$ denotes the relations which can be pullback from the composition map  $\pi_{\bigstar}H\uFp \to \pi_{\bigstar}\Fun(E_{C_p}\Zp_{+},H\uFp) \to \pi_{\bigstar}\Fun(E\cG_{+},H\uFp) $ which is explicitly given as: 

    \[
\begin{aligned}
a_{\beta} &\,\mapsto\, y_1u_{\beta}, 
&\qquad & \kappa_{\beta} \,\mapsto\, t_1u_{\beta},&
u_{\beta} \,\mapsto\, u_{\beta}, \\[4pt]
a_{\alpha_i} &\,\mapsto\, (y_0-iy_1)u_{\alpha_i}, 
&\qquad & \kappa_{\alpha_i} \,\mapsto\, (t_0-it_1)u_{\alpha_i},&
u_{\alpha_i} \,\mapsto\, u_{\alpha_i}, \\[4pt]
w_{T} &\,\mapsto\, (s_0-s_1)t_0t_1\frac{\prod_{j\in T}u_{\alpha_j}}{u^{|T|-2}_{\beta}}, 
&\qquad & z_{T} \,\mapsto\, (s_0-s_1)(t_0y_1-t_1y_0)\frac{\prod_{j\in T}u_{\alpha_j}}{u^{|T|-2}_{\beta}},\\[4pt]
v_{S} &\,\mapsto\, \frac{\prod_{j\in S}u_{\alpha_j}}{u^{|S|-1}_{\beta}}, 
&\qquad & k^{(i)}_{S_i} \,\mapsto\, \frac{\prod_{j\in S_i}u_{\alpha_j}}{u^{|S_i|-1}_{\alpha_i}},\\[4pt]
\phi^{(i)}_{T_i} &\,\mapsto\, (s_0-s_1)t_0t_1\frac{\prod_{j\in T_i}u_{\alpha_j}}{u^{|T_i|-2}_{\alpha_i}}, 
&\qquad & \psi^{(i)}_{T_i} \,\mapsto\, (s_0-s_1)(t_0y_1-t_1y_0)\frac{\prod_{j\in T_i}u_{\alpha_j}}{u^{|T_i|-2}_{\alpha_i}},
\end{aligned}
\]
where $s_0<s_1$ denotes the consecutive lowest member of the corresponding set. 
\end{thm}
\begin{proof}
    Above polynomial structure is obtained from the computation of kernel of the connecting morphism $ \partial_{\bigstar}:\pi_{\bigstar}(\widetilde{E_{C_p}\Zp}\wedge H\uFp)\to \pi_{\bigstar-1}({E_{C_p}\Zp}_{+}\wedge H\uFp)$ using the fact that any class with $a_{\alpha_r}$ not in the denominator must lie in $\Ker \partial_{\bigstar}$, other classes follow from the Proposition~\ref{kclass} and Proposition~\ref{phiclass}. 
    
    One can prove that the map $P(\pi_{\bigstar}H\uFp) \to P(\pi_{\bigstar}\Fun(E_{C_p}\Zp_{+}, H\uFp))$ is injective, using the fact that $a$, $u$ and $\kappa$ map to themselves and other classes map nontrivially using the ring structure of the map. Since the latter map in the composition is injective on these classes by \eqref{pullp}, it follows that the composite map is injective when we restrict our domain to the polynomial part.    
\end{proof}
 \end{mysubsect} 

\sect{The ring structure of $\widetilde{H}_{\mathcal{K}_4}^\bs(S^0; \uF)$}
 We now consider the case $p=2$.  Although the overall strategy is the same as that used for odd primes, in this case, we carry out the computation of the \emph{entire} $\RO(\K)$–graded Bredon cohomology ring, rather than restricting to its polynomial part. Consider the following Tate Square

\begin{myeq}\label{tate2}
\xymatrix@C=0.4cm{
  \bE_+ \wedge H\uF \ar[r]^-{f} \ar[d]_\simeq 
    & H\uF \ar[d] \ar[r] 
    & \widetilde{\bE} \wedge H\uF \ar[d] \\
  \bE_+ \wedge \Fun(\bE_+, H\uF) \ar[r] 
    & \Fun(\bE_+, H\uF) \ar[r]^-{\mathcal{L}} 
    & \widetilde{\bE} \wedge \Fun(\bE_+, H\uF)
}
\end{myeq}

\begin{enumerate}
\item The computation begins with the homotopy groups
$
\pi_{\bigstar}\Fun(\bE_{+},H\uF),
$
which are given by Theorem~\ref{main2}:
\[
\pi_{\bigstar}\Fun(\bE_{+},H\uF)\cong \widetilde{H}^{-\bigstar}_{\cG}(\bE_{+};\uF).
\]

\item Since \(\widetilde{\bE}\simeq S^{\infty(\alpha_0\oplus\alpha_1)}\), we obtain
\[
\pi_{\bigstar}\big(\widetilde{\bE}\wedge\Fun(\bE_+,H\uF)\big)
\cong
\pi_{\bigstar}\Fun(\bE_+,H\uF)\big[a_{\alpha_0}^{\pm 1},a_{\alpha_1}^{\pm 1}\big].
\]

\item  Since the
ring has zero divisors, the localisation map \(\mathcal{L}_{\bigstar}\) is not injective.
Following the lower cofiber sequence and the left-vertical equivalence in the
Tate-square diagram, one computes \(\pi_{\bigstar}(\bE_{+}\wedge H\uF)\) as follows.

\begin{prop}\label{right}
The homotopy \(\pi_{\bigstar}(\bE_{+}\wedge H\uF)\) decomposes as
\[
\pi_{\bigstar}(\bE_{+}\wedge H\uF)
\;\cong\;
\Sigma^{-1}\Coker(\mathcal{L}_{\bigstar+1})
\;\oplus\; \Ker(\mathcal{L}_{\bigstar}),
\]
where the cokernel summand has the explicit description
\begin{align*}
\Coker(\mathcal{L}_{\bigstar})
&\cong
\bM_{2}(\beta) \otimes\Bigg[\F\!\left\langle\Sigma^{-1}\frac{1}{a_{\alpha_0}}\right\rangle
\big[a_{\alpha_0}^{-1},a_{\alpha_1},u_{\alpha_0},u_{\alpha_1},v^{\pm}\big] \\
&\quad\oplus
\F\!\left\langle\Sigma^{-1}\frac{1}{a_{\alpha_1}}\right\rangle
\big[a_{\alpha_0},a_{\alpha_1}^{-1},u_{\alpha_0},u_{\alpha_1},v^{\pm}\big] \\
&\quad\oplus
\F\!\left\langle\Sigma^{-1}\frac{1}{a_{\alpha_0}a_{\alpha_1}}\right\rangle
\big[a_{\alpha_0}^{-1},a_{\alpha_1}^{-1},u_{\alpha_0},u_{\alpha_1},v^{\pm}\big]\Bigg] \mod(\Image \mathcal{L}_{\bigstar}),
\end{align*}
and the kernel is the ideal of \(\widetilde{H}^{\bigstar}_{\cG}(\bE_{+};\uF)\)
generated by classes of the form
\[
u_{\alpha_i}^{m}\,\Sigma^{-1}\frac{1}{a_{\beta}^{{k_1}}u_{\beta}^{k_2}}, \quad \text{such that} \quad a^{m}_{\alpha_j}u_{\alpha_i}^{m}\,\Sigma^{-1}\frac{1}{a_{\beta}^{{k_1}}u_{\beta}^{k_2}}=0
\]
for $i\neq j \in \{0,1\}$, subject to the relations induced from Theorem~\ref{main}.  In particular,
the relations are the following (for \(i\neq j\)):
\begin{enumerate}
  \item \(v\,u_{\alpha_i}^{m}\,\Sigma^{-1}\frac{1}{a_{\beta}^{k_1}u_{\beta}^{k_2}}
        = u_{\alpha_i}^{m+1}u_{\alpha_j}\,\Sigma^{-1}\frac{1}{a_{\beta}^{k_1}u_{\beta}^{k_2+1}},\)
  \item \(v\,u_{\alpha_i}^{m}\,\Sigma^{-1}\frac{1}{a_{\beta}^{k_1}u_{\beta}^{k_2}}
        = u_{\alpha_i}^{m+1}a_{\alpha_j}\,\Sigma^{-1}\frac{1}{a_{\beta}^{k_1+1}u_{\beta}^{k_2}}
        + u_{\alpha_i}^{m}u_{\alpha_j}a_{\alpha_i}\,\Sigma^{-1}\frac{1}{a_{\beta}^{k_1+1}u_{\beta}^{k_2}}.\)
\end{enumerate}
\end{prop}
\begin{remark}{\label{reduct}}
Note that $\F[a_{\beta},u_\beta]$ term (from $\bM_{2}(\beta)$) can be ignored in the expression of $\Ker\mathcal{L}_{\bigstar}$, as all the corresponding classes can be produced using the remaining classes and relations.     
\end{remark}
\item The computation of \(\pi_{\bigstar}(\widetilde{\bE}\wedge H\uF)\) reduces
to understanding the map \(f_{\bigstar_{\beta}}\).
Following Proposition~\ref{kernel} and Proposition~\ref{fmap}, we have the following:

\begin{prop}
The map
\[
f_{\bigstar_{\beta}}\colon
\pi_{\bigstar_{\beta}}(\bE_{+}\wedge H\uF)
\longrightarrow
\pi_{\bigstar_{\beta}} H\uF
\]
is trivial.  where,
\begin{align*}
    \pi_{\bigstar_{\beta}}(\bE_{+}\wedge H\uF)
\cong
\bM_{2}(\beta)\otimes\Bigg[
\F\!\left\langle\Sigma^{-1}\frac{u_{\alpha_0}}{a_{\alpha_0}}\right\rangle\big[\tfrac{u_{\alpha_0}}{a_{\alpha_0}}\big] \\
\oplus
\F\!\left\langle\Sigma^{-1}\frac{u_{\alpha_1}}{a_{\alpha_1}}\right\rangle\big[\tfrac{u_{\alpha_1}}{a_{\alpha_1}}\big]
\;\oplus\;
\F\!\left\langle\Sigma^{-1}\frac{v}{a_{\alpha_0}a_{\alpha_1}}\right\rangle\big[\tfrac{u_{\alpha_0}}{a_{\alpha_0}},\tfrac{u_{\alpha_1}}{a_{\alpha_1}},\tfrac{v}{a_{\alpha_0}a_{\alpha_1}}\big]
\Bigg]/\!\sim_{1},
\end{align*}

where \(\sim_{1}\) denotes the relations inherited from Proposition~\ref{right},
and \(\pi_{\bigstar_{\beta}} H\uF \cong \bM_{2}(\beta)\) by Proposition~\ref{res.}.
\end{prop}

\item Finally, since
\[
\pi_{\bigstar}\big(\widetilde{\bE}\wedge H\uF\big)
\cong
\pi_{\bigstar_{\beta}}\big(\widetilde{\bE}\wedge H\uF\big)[a_{\alpha_0}^{\pm1},a_{\alpha_1}^{\pm1}],
\]
and the previous calculations imply:

\begin{prop}
The homotopy  \(\pi_{\bigstar}(\widetilde{\bE}\wedge H\uF)\) admit the decomposition
\[
\begin{aligned}
\pi_{\bigstar}(\widetilde{\bE}\wedge H\uF)
&\cong
\bM_{2}(\beta)[a_{\alpha_0}^{\pm1},a_{\alpha_1}^{\pm1}] \\
&\quad\oplus
\bM_{2}(\beta)\otimes\Big[
\F\langle u_{\alpha_0}\rangle[u_{\alpha_0},a_{\alpha_0}^{\pm1},a_{\alpha_1}^{\pm1}]
\oplus
\F\langle u_{\alpha_1}\rangle[u_{\alpha_1},a_{\alpha_0}^{\pm1},a_{\alpha_1}^{\pm1}] \\
&\qquad\qquad\oplus\;
\F\langle v\rangle[u_{\alpha_0},u_{\alpha_1},v,a_{\alpha_0}^{\pm1},a_{\alpha_1}^{\pm1}]
\Big]/\!\sim_{2},
\end{aligned}
\]
where \(\sim_{2}\) denotes corresponding relations induced from \(\sim_{1}\).
\end{prop}
\end{enumerate}


\begin{mysubsect}{The boundary map}

Since the computation of the kernel and the cokernel of the boundary map
\[
\partial_{\bigstar}:\pi_{\bigstar}\widetilde{\bE}\wedge H\uF \longrightarrow \pi_{\bigstar-1}\widetilde{\bE}\wedge \Fun(\bE_{+},H\uF)
\]
control the desired result, it suffices to analyse the restricted map
\[
\partial_{\bigstar_{\beta}}:\pi_{\bigstar_{\beta}}\widetilde{\bE}\wedge H\uF
=\pi_{\bigstar_{\beta}}H\uF \ \oplus\ \Sigma\,\pi_{\bigstar_{\beta}-1}\bE_{+}\wedge H\uF
\ \longrightarrow\ \pi_{\bigstar_{\beta}-1}\bE_{+}\wedge H\uF
\]
together with the $\pi_{\bigstar}H\uF$--module structure.

\begin{prop}\label{boundary}
The boundary map for $\bigstar_{\beta}\in \Z\{1,\beta\}$ is given by
\[
\partial_{\bigstar_{\beta}}:\pi_{\bigstar_{\beta}}\widetilde{\bE}\wedge H\uF
=\pi_{\bigstar_{\beta}}H\uF \oplus\Sigma\,\pi_{\bigstar_{\beta}-1}\bE_{+}\wedge H\uF
\ \longrightarrow\ \pi_{\bigstar_{\beta}-1}\bE_{+}\wedge H\uF
\]
is given by
\[
\partial_{\bigstar_{\beta}}(x+y) \ =\ \Sigma^{-1} y,
\]
where $x \in \pi_{\bigstar_{\beta}}H\uF$ and $y \in \Sigma\,\pi_{\bigstar_{\beta}-1}\bE_{+}\wedge H\uF$.
\end{prop}

\begin{proof}
This follows directly from the long exact sequence of homotopy groups associated to the top cofiber sequence in the Tate diagram~\eqref{tate2}.
\end{proof}
\end{mysubsect}

\begin{mysubsect}{Existence of more special classes}
In Proposition~\ref{inv}, we established the existence of the class $v$ for the subgroup $K$.  
Here we show the existence of classes $k_0$ and $k_1$, which play analogous roles for the subgroups $H_0$ and $H_1$, respectively.

\begin{prop}\label{k_0k_1}
For $i,j \in \{0,1\}$ with $i \neq j$, there exists a class
\[
k_i \ \in\ \pi_{1+\alpha_i-\beta-\alpha_j}(H\uF)
\]
satisfying:
\begin{enumerate}
    \item $k_{i}u_{\alpha_i} = u_{\beta}u_{\alpha_j}$,
    \item $k_{i}a_{\alpha_i} = a_{\beta}u_{\alpha_j} + u_{\beta}a_{\alpha_j}$,
    \item $k_{i}k_{j} = u^2_{\beta}$,
    \item $k_{i}v = u^2_{\alpha_j}$.
\end{enumerate}
\end{prop}

\begin{proof}
Consider the class
\[
\frac{a_{\alpha_1}}{a_{\alpha_0}}u_{\beta} + \frac{u_{\alpha_1}}{a_{\alpha_0}}a_{\beta}
\ \in\ \pi_{\bigstar}\widetilde{\bE}\wedge H\uF.
\]
Since
\[
\partial_{\bigstar}\!\left( \frac{a_{\alpha_1}}{a_{\alpha_0}}u_{\beta} + \frac{u_{\alpha_1}}{a_{\alpha_0}}a_{\beta} \right)
= \Sigma^{-1}\!\left( \frac{a_{\alpha_1}}{a_{\alpha_0}}u_{\beta} + \frac{u_{\alpha_1}}{a_{\alpha_0}}a_{\beta} \right)
= \Sigma^{-1}\frac{u^2_{\alpha_1}}{v} = 0,
\]
it lies in $\ker(\partial_{\bigstar})$.   
This produces the desired class $k_0 \in \pi_{\bigstar}H\uF$ to be the preimage of the above class. Similarly, one may construct a class $k_1.$
The listed relations follow from the exactness of the top cofiber sequence in the  diagram~\eqref{tate2}.
\end{proof}
\end{mysubsect}

\begin{notation}
For simplicity, we denote $k'_0:=\frac{a_{\alpha_1}}{a_{\alpha_0}}u_{\beta} + \frac{u_{\alpha_1}}{a_{\alpha_0}}a_{\beta}$ and $k'_1:=\frac{a_{\alpha_0}}{a_{\alpha_1}}u_{\beta} + \frac{u_{\alpha_0}}{a_{\alpha_1}}a_{\beta}$ be the images of $k_0$ and $k_1$ in $\pi_{\bigstar}\widetilde{\bE}\wedge H\uF$ respectively, under the relevant quotient map.  
\end{notation}

To compute $\Ker \partial_{\bigstar}$,
consider a class
\[
x(k)=
a^{-i}_{\alpha_1}k'^n_1u^m_{\alpha_0}a^\ell_{\alpha_0}
\,\Sigma^{-1}\frac{1}{a^j_{\beta}u^k_{\beta}}\in \pi_{\bigstar}\widetilde{\bE}\wedge H\uF
\] 
where $i,j\geq 1$ and $n,m,\ell\geq 0$ be fixed integers. Here we will define shorthand $\Xi$, which gives the range of $k$ such that $x(k)$ lies in $\Ker \partial_{\bigstar}$.

{\bf I.} Start with a basic case

\[
a^{-i}_{\alpha_1}u^m_{\alpha_0}
\,\Sigma^{-1}\frac{1}{a^j_{\beta}u^k_{\beta}}.
\]
In $\pi_{\bigstar}(\bE_{+}\wedge H\uF)$ we have:
\begin{align*}
\partial_{\bigstar}\!\left(a^{-i}_{\alpha_1}u^m_{\alpha_0}\Sigma^{-1}\frac{1}{a^j_{\beta}u^k_{\beta}}\right)
&= \Sigma^{-1}\frac{u^m_{\alpha_0}}{a^i_{\alpha_1}}
   \,\Sigma^{-1}\frac{1}{a^j_{\beta}u^k_{\beta}}\\
&= \Sigma^{-1}\frac{u^m_{\alpha_0}}{v\,a^i_{\alpha_1}}
   \cdot v\,a_{\beta}\,\Sigma^{-1}\frac{1}{a^{j+1}_{\beta}u^k_{\beta}}\\
&= \Sigma^{-1}\frac{u^m_{\alpha_0}}{v\,a^i_{\alpha_1}}
   \cdot \big(u_{\alpha_0}a_{\alpha_1} + u_{\alpha_1}a_{\alpha_0}\big)
   \,\Sigma^{-1}\frac{1}{a^{j+1}_{\beta}u^k_{\beta}}\\
&= \Sigma^{-1}\frac{u^{m+1}_{\alpha_0}}{v\,a^{i-1}_{\alpha_1}}
   \,\Sigma^{-1}\frac{1}{a^{j+1}_{\beta}u^k_{\beta}}
   \;+\;
   \Sigma^{-1}\frac{u^{m-1}_{\alpha_0}a_{\alpha_0}}{a^i_{\alpha_1}}
   \,\Sigma^{-1}\frac{1}{a^{j+1}_{\beta}u^{k-1}_{\beta}}
\end{align*}

From this calculation, we see:  
\[
\partial_{\bigstar}\!\left(a^{-i}_{\alpha_1}u^m_{\alpha_0}\Sigma^{-1}\frac{1}{a^j_{\beta}u^k_{\beta}}\right) = 0 \quad\Longleftrightarrow\quad m \ge k.
\]

{\bf II.} Consider the general class
\[
a^{-i}_{\alpha_1}k'^n_1u^m_{\alpha_0}a^\ell_{\alpha_0}
\,\Sigma^{-1}\frac{1}{a^j_{\beta}u^k_{\beta}}.
\] 
where $i,j\geq 1$ and $n,m,\ell\geq 0$ be fixed integers.

In $\pi_{\bigstar}(\bE_{+}\wedge H\uF)$ we have:
\begin{align*}
\partial_{\bigstar}\!\left(
a^{-i}_{\alpha_1}k'^n_1u^m_{\alpha_0}a^\ell_{\alpha_0}
\Sigma^{-1}\frac{1}{a^j_{\beta}u^k_{\beta}}
\right)
&= \Sigma^{-1}
\frac{k'^n_1u^m_{\alpha_0}(a_{\alpha_0}u_{\beta})^\ell}{a^i_{\alpha_1}}
\;\Sigma^{-1}\frac{1}{a^j_{\beta}u^{k+l}_{\beta}}\\
&= \Sigma^{-1}
\frac{k'^n_1u^m_{\alpha_0}\,\big(k'_1a_{\alpha_1}+a_{\beta}u_{\alpha_0}\big)^\ell}
     {a^i_{\alpha_1}}
\;\Sigma^{-1}\frac{1}{a^j_{\beta}u^{k+\ell}_{\beta}}.
\end{align*}

{\bf III.} Expansion and term-by-term analysis

Expanding \(\big(k'_1a_{\alpha_1}+a_{\beta}u_{\alpha_0}\big)^\ell\) gives terms of the form
\[
{k'}_1^{b}a_{\alpha_1}^b\;a_{\beta}
^{\,\ell-b}u_{\alpha_0}^{\,\ell-b},
\]
for some $0 \le b \le \ell$. 

The corresponding term in $\pi_{\bigstar} \bE_{+}\wedge H\uF$ is
\[
\Sigma^{-1}
\frac{k'^{n+b}_1\,u^{m+\ell-b}_{\alpha_0}}{a^{\,i-b}_{\alpha_1}}
\;\Sigma^{-1}\frac{1}{a^{j+b-l}_{\beta}u^{k+\ell}_{\beta}}
\;=\;
\Sigma^{-1}
\frac{u^{\,2n+2b+m+l-b}_{\alpha_0}}{v^{n+b}\,a^{\,i-b}_{\alpha_1}}
\;\Sigma^{-1}\frac{1}{a^{j+b-l}_{\beta}u^{k+\ell}_{\beta}},
\]
provided $j+b-\ell > 0$ so that the $a_\beta$-exponent is positive.

{\bf IV.} From the above analysis, we set:
\begin{enumerate}
\item Let $\{\,0 = b_0 < b_1 < \cdots < b_t = \ell\,\}$ be the set of $b$-values appearing in the binomial expansion (these are the exponents of $k'_1a_{\alpha_1}$).  
Choose the \emph{smallest} $b_s$ such that $j + b_s - \ell > 0$.
\item Then $\Xi(n,m,\ell,i,j)$ is a $2^\mathbb{N}$-valued function defined as: 
\begin{myeq}{\label{cond}}
\begin{cases}
[1,2n + b_s + m],  &\text{ if } i - b_s > 0,\\
 \mathbb{N}, &\text{if } i-b_{s}\leq 0
\end{cases}
\end{myeq}
\end{enumerate}

\begin{prop}{\label{xk}}
    A class $x(k)=
a^{-i}_{\alpha_1}k'^n_1u^m_{\alpha_0}a^\ell_{\alpha_0}
\,\Sigma^{-1}\frac{1}{a^j_{\beta}u^k_{\beta}}\in \pi_{\bigstar}\widetilde{\bE}\wedge H\uF$ lie in $\Ker \partial_{\bigstar}$ if and only if $k \in \Xi(n,m,\ell,i,j)$. 
\end{prop}

In fact, in the case of $i-b_s\leq 0$, the class $x(k)$ can be represented in a way that it does not have $a_{\alpha_1}$ factor in the denominator, hence it lies in the second summand of $\Ker \partial_{\bigstar}$ as expressed below. 

\begin{prop}\label{ker}
The kernel and cokernel of $\partial_{\bigstar}$ can be described as follows
\begin{enumerate}
\item
\begin{align*}
\Ker \partial_{\bigstar}
&= \F[a_{\beta},u_{\beta},a_{\alpha_0},u_{\alpha_0},a_{\alpha_1},u_{\alpha_1},k'_0,k'_1,v] \\
&\quad\oplus \F\!\left\langle \Sigma^{-1}\frac{1}{a^j_{\beta}u^k_{\beta}} \right\rangle[a_{\alpha_0},a_{\alpha_1},u_{\alpha_0},u_{\alpha_1},k'_0,k'_1,v] \\
&\quad\oplus \F\!\left\langle a^{-i}_{\alpha_1}k'^n_1u^m_{\alpha_0}a^\ell_{\alpha_0}
\Sigma^{-1}\frac{1}{a^j_{\beta}u^k_{\beta}} \,:\,b_s<i, k \leq 2n+b_s+m \right\rangle[u_{\alpha_1},k'_0,v] \\
&\quad\oplus \F\!\left\langle a^{-i}_{\alpha_0}k'^n_0u^m_{\alpha_1}a^\ell_{\alpha_1}
\Sigma^{-1}\frac{1}{a^j_{\beta}u^k_{\beta}} \,:\,b_s<i, k \leq 2n+b_s+m\right\rangle[u_{\alpha_0},k'_1,v].
\end{align*}
where $k'_0:=\frac{a_{\alpha_1}}{a_{\alpha_0}}u_{\beta} + \frac{u_{\alpha_1}}{a_{\alpha_0}}a_{\beta}$ and $k'_1:=\frac{a_{\alpha_0}}{a_{\alpha_1}}u_{\beta} + \frac{u_{\alpha_0}}{a_{\alpha_1}}a_{\beta}$.
\item
\begin{align*}
\Coker \partial_{\bigstar}
&= \Big[\F\!\left\langle \Sigma^{-1}\frac{1}{va_{\alpha_0}} \right\rangle[a^{-1}_{\alpha_0},a_{\alpha_1},u_{\alpha_0},u_{\alpha_1},v^{-1}] \\
&\quad\oplus \F\!\left\langle \Sigma^{-1}\frac{1}{va_{\alpha_1}} \right\rangle[a_{\alpha_0},a^{-1}_{\alpha_1},u_{\alpha_0},u_{\alpha_1},v^{-1}] \\
&\quad\oplus \F\!\left\langle \Sigma^{-1}\frac{1}{va_{\alpha_0}a_{\alpha_1}} \right\rangle[a^{-1}_{\alpha_0},a^{-1}_{\alpha_1},u_{\alpha_0},u_{\alpha_1},v^{-1}] \Big] \\
&\quad\oplus \F\!\left\langle \Sigma^{-1}\frac{1}{a^{j}_{\beta}u^{k}_{\beta}} \right\rangle
\otimes\Big[
\F\!\left\langle \Sigma^{-1}\frac{1}{va_{\alpha_0}} \right\rangle[a^{-1}_{\alpha_0},v^{-1},u_{\alpha_0},u_{\alpha_1}] \\
&\quad\oplus \F\!\left\langle \Sigma^{-1}\frac{1}{va_{\alpha_1}} \right\rangle[a^{-1}_{\alpha_1},v^{-1},u_{\alpha_0},u_{\alpha_1}] \\
&\quad\oplus \F\!\left\langle \Sigma^{-1}\frac{1}{va_{\alpha_0}a_{\alpha_1}} \right\rangle[a^{-1}_{\alpha_0},a^{-1}_{\alpha_1},v^{-1},u_{\alpha_0},u_{\alpha_1}]
\Big] \\
&\quad\oplus \Ker \mathcal{L}_{\bigstar}.
\end{align*}
modulo the restricted relations from $\pi_{\bigstar}\bE_{+}\wedge H\uF$ and $\Image \partial_{\bigstar}$.
\end{enumerate}
\end{prop}

\begin{proof}
Every class in $\pi_{\bigstar}\bE_{+} \wedge H\uF$ is either a multiple of $a^{-1}_{\alpha_0}$ or $a^{-1}_{\alpha_1}$, or from $\Ker\mathcal{L}_{\bigstar}$ but due to Proposition \ref{boundary}, nothing hits $\Ker\mathcal{L}_{\bigstar}$. Therefore, all classes free of $a^{-1}_{\alpha_0}$, or   $a^{-1}_{\alpha_1}$ must lie in $\Ker \partial_{\bigstar}$.  In fact, together with Propositions \ref{k_0k_1}, the first two summands lie in $\Ker \partial_{\bigstar}$. The final two summands lie in $\Ker \partial_{\bigstar}$ by the $\Xi$-criterion established above in Proposition~\ref{xk}, and they continue to satisfy the $\Xi$-criterion when taking polynomials over the given classes.

For the cokernel, note that no class maps to the one involving $v^{-1}$ or to $\Ker\mathcal{L}_{\bigstar}$, and together with Remark \ref{reduct}, yielding the stated decomposition.
\end{proof}

\begin{mysubsect}{Symmetry and identifications in the coefficient ring} For the sake of symmetry in the structure of $\pi_{\bigstar} H \uF$, we introduce an equivalence relation among certain newly defined classes. This equivalence is encoded by the following correspondence:
\begin{myeq}  \label{equiv1}
    \begin{cases}
        \Sigma^{-1}\dfrac{u_{\alpha_0}}{v a_{\alpha_0}} 
            \ \longmapsto\  
            k_{1} \, \Sigma^{-1} \dfrac{1}{a_{\alpha_0} u_{\alpha_0}}, \\[1.2em]
        \Sigma^{-1}\dfrac{u_{\alpha_0}^{\,2n-k}}{v^{n} a_{\alpha_0}^{\,j}} 
            \ \longmapsto\  
            k^{n}_{1} \, \Sigma^{-1} \dfrac{1}{a_{\alpha_0}^{\,j} u_{\alpha_0}^{\,k}}, 
            & 2n \geq k, \\[1.2em]
        \Sigma^{-1}\dfrac{a_{\alpha_1}}{v a_{\alpha_0}} 
            \ \longmapsto\  
            k_{1} a_{\alpha_1} \, \Sigma^{-1} \dfrac{1}{a_{\alpha_0} u_{\alpha_0}^{\,2}}
            = a_{\beta} \, \Sigma^{-1} \dfrac{1}{a_{\alpha_0} u_{\alpha_0}}, \\[1.2em]
        \Sigma^{-1}\dfrac{u_{\alpha_1}}{v a_{\alpha_0}}
            \ \longmapsto\ 
            k_1u_{\alpha_1} \, \Sigma^{-1} \dfrac{1}{a_{\alpha_0} u^2_{\alpha_0}}
            = u_{\beta} \, \Sigma^{-1} \dfrac{1}{a_{\alpha_0} u_{\alpha_0}}.
    \end{cases}    
\end{myeq}

\begin{remark}
The above correspondence is well-defined in $\pi_{\bigstar} H \uF$ because all terms are of the same degree and compatible with the multiplicative structure. Furthermore, the relations are chosen so that the grading and the multiplicative generators behave symmetrically under this identification.
\end{remark}

In particular, the vanishing relation
\[
\Sigma^{-1} \dfrac{u_{\alpha_0}^{\,\ell} u_{\alpha_1}^{\,m}}{v^{k} a_{\alpha_0}^{\,n}} = 0
\quad\text{whenever}\quad \ell+m \geq 2k
\]
is preserved under the correspondence, ensuring that the quotient by this equivalence inherits the same annihilation patterns.

Since class $\Sigma^{-1}\frac{1}{a^{j}_{\beta}u^{k}_{\beta}} $ exists in $ \pi_{\bigstar}H\uF$ which motivates us to define classes like $\Sigma^{-1}\frac{1}{a^{j}_{\alpha_0}u^{k}_{\alpha_0}} $ and $\Sigma^{-1}\frac{1}{a^{j}_{\alpha_1}u^{k}_{\alpha_1}} $ in $\pi_{\bigstar}H\uF$. To prove the existence of such classes in the corresponding degree, we define the following equivalence in  $\Ker\mathcal{L}_{\bigstar}$:

Since $\frac{u^2_{\alpha_1}}{v^2}\Sigma^{-1}\frac{1}{a_{\beta}a_{\beta}}$ and $\frac{u^2_{\alpha_0}}{v^2}\Sigma^{-1}\frac{1}{a_{\beta}a_{\beta}}$ lie in $\Ker\mathcal{L}_{\bigstar}$ of degree $2-2\alpha_0$, and $2-2\alpha_1$ respectively, we  define classes $\Sigma^{-1}\frac{1}{a_{\alpha_0}u_{\alpha_0}}$ and $\Sigma^{-1}\frac{1}{a_{\alpha_1}u_{\alpha_1}}$ to be their image in $\pi_{\bigstar}H\uF$ in their respective degree, and using the simple steps we can proof the existence of isomorphic preimage of $\F{\langle{\Sigma^{-1}\frac{1}{a^j_{\alpha_0}u^k_{\alpha_0}}}\rangle}_{j,k\geq 1} $ and $\F{\langle{\Sigma^{-1}\frac{1}{a^j_{\alpha_1}u^k_{\alpha_1}}}\rangle}_{j,k\geq 1} $ in $\Ker\mathcal{L}_{\bigstar}$.

We describe an explicit procedure for producing a representative in
$\Ker \mathcal{L}_{\bigstar}$ of the class
$\Sigma^{-1}\frac{1}{a_{\alpha_0}^{\,j}u_{\alpha_0}^{\,k}}$.

\begin{enumerate}
\item If
\[
\frac{u_{\alpha_1}^{\,j+k}}{v^{\,j+k}}
\Sigma^{-1}\frac{1}{a_{\beta}^{\,j}u_{\beta}^{\,k}}
\in \Ker \mathcal{L}_{\bigstar},
\]
then this element represents the desired class.

\item Otherwise, consider
\[
\frac{u_{\alpha_1}^{\,j+k+1}}{v^{\,j+k+1}}
\Sigma^{-1}\frac{1}{a_{\beta}^{\,j+1}u_{\beta}^{\,k}}.
\]
If this element lies in $\Ker \mathcal{L}_{\bigstar}$, multiply by $a_{\alpha_0}$
to correct the degree. If not, repeat this step until such an element is found.
\end{enumerate}

Several explicit representatives obtained by this procedure are listed below.

\begin{center}
    
\begin{tabular}{|c|c|c|}
  \hline
  Class & Preimage \\
  \hline
   $\Sigma^{-1}\frac{1}{a_{\alpha_0}u_{\alpha_0}}$  & $\frac{u^2_{\alpha_1}}{v^2}\Sigma^{-1}\frac{1}{a_{\beta}u_{\beta}}$      \\
  $\Sigma^{-1}\frac{1}{a^2_{\alpha_0}u_{\alpha_0}}$  & $\frac{u^3_{\alpha_1}}{v^3}\Sigma^{-1}\frac{1}{a^2_{\beta}u_{\beta}}$      \\
  $\Sigma^{-1}\frac{1}{a_{\alpha_0}u^2_{\alpha_0}}$  & $a_{\alpha_0}\frac{u^4_{\alpha_1}}{v^4}\Sigma^{-1}\frac{1}{a^2_{\beta}u^2_{\beta}}$      \\
  $\Sigma^{-1}\frac{1}{a^2_{\alpha_0}u^2_{\alpha_0}}$  & $\frac{u^4_{\alpha_1}}{v^4}\Sigma^{-1}\frac{1}{a^2_{\beta}u^2_{\beta}}$      \\
  $\Sigma^{-1}\frac{1}{a_{\alpha_0}u^3_{\alpha_0}}$  & $a^3_{\alpha_0}\frac{u^7_{\alpha_1}}{v^7}\Sigma^{-1}\frac{1}{a^4_{\beta}u^3_{\beta}}$      \\
        
  \hline

\end{tabular} 

\end{center}

\end{mysubsect}

\begin{prop}{\label{equiv}}
    There is an equivalence of rings 
\begin{enumerate}
    \item 
    \begin{myeq}
    \begin{aligned}
        \F\langle{\Sigma^{-1}\frac{1}{va_{\alpha_0}}}\rangle[a^{-1}_{\alpha_0},u_{\alpha_0},a_{\alpha_1},u_{\alpha_1},v^{-1}] \oplus \F\langle{\Sigma^{-1}\frac{1}{va_{\alpha_1}}}\rangle[a_{\alpha_0},u_{\alpha_0},a^{-1}_{\alpha_1},u_{\alpha_1},v^{-1}]\\
        \oplus \F\langle{\Sigma^{-1}\frac{1}{va_{\alpha_0}a_{\alpha_1}}}\rangle[a^{-1}_{\alpha_0},u_{\alpha_0},a^{-1}_{\alpha_1},u_{\alpha_1},v^{-1}]\notag \oplus \Ker l_{\bigstar}
    \end{aligned}
    \end{myeq}
    and 
    \begin{align*}
       \F\langle{\Sigma^{-1}\frac{1}{a^{j}_{\alpha_0}u^{k}_{\alpha_0}}}\rangle[a_{\alpha_1},u_{\alpha_1},a_{\beta},u_{\beta},k_0,k_1,v] \oplus \F\langle{\Sigma^{-1}\frac{1}{a^{j}_{\alpha_1}u^{k}_{\alpha_1}}}\rangle[a_{\alpha_0},u_{\alpha_0},a_{\beta},u_{\beta},k_0,k_1,v]\\
 \oplus       \F\langle{a^{-i}_{\alpha_1}k^{n}_1a^{\ell}_{\beta}u^{m}_{\beta}\Sigma^{-1}\frac{1}{a^{j}_{\alpha_0}u^{k}_{\alpha_0}}\colon b_s<i, k \leq 2n+b_s+m}\rangle[u_{\alpha_1},k_0,v]\\ \oplus \F\langle{a^{-i}_{\alpha_0}k^{n}_0a^{\ell}_{\beta}u^{m}_{\beta}\Sigma^{-1}\frac{1}{a^{j}_{\alpha_1}u^{k}_{\alpha_1}}\colon b_s<i, k \leq 2n+b_s+m}\rangle[u_{\alpha_0},k_1,v].
    \end{align*}
where $\ell,m,n\geq 0$ and $i,j,k\geq 1$.
\item \begin{align*}
    \Big[
\F\!\left\langle \Sigma^{-1}\frac{1}{va_{\alpha_0}} \right\rangle[a^{-1}_{\alpha_0},a_{\alpha_1},v^{-1},u_{\alpha_0},u_{\alpha_1}] 
\quad\oplus \F\!\left\langle \Sigma^{-1}\frac{1}{va_{\alpha_1}} \right\rangle[a_{\alpha_0},a^{-1}_{\alpha_1},v^{-1},u_{\alpha_0},u_{\alpha_1}] \\
\quad\oplus \F\!\left\langle \Sigma^{-1}\frac{1}{va_{\alpha_0}a_{\alpha_1}} \right\rangle[a^{-1}_{\alpha_0},a^{-1}_{\alpha_1},v^{-1},u_{\alpha_0},u_{\alpha_1}]
\Big] \otimes \F\langle{\Sigma^{-1}\frac{1}{a^s_{\beta}u^t_{\beta}}}\rangle
\end{align*}
and
\begin{align*}
\Big[\F\langle{\Sigma^{-1}\frac{1}{a^{j}_{\alpha_0}u^{k}_{\alpha_0}}}\rangle[a^{\pm}_{\alpha_1},u_{\alpha_1},k_0,k_1,v]  \oplus \F\langle{\Sigma^{-1}\frac{1}{a^{j}_{\alpha_1}u^{k}_{\alpha_1}}}\rangle[a^{\pm}_{\alpha_0},u_{\alpha_0},k_0,k_1,v]\Big]\otimes \F\langle{\Sigma^{-1}\frac{1}{a^s_{\beta}u^t_{\beta}}}\rangle
\end{align*} 
where $s,t,j,k \geq 1$.
    \end{enumerate}
    subject to the following relation:
        \begin{enumerate}
        \item $k_0u_{\alpha_0}-u_{\alpha_1}u_{\beta}$
    \item $k_1u_{\alpha_1}-u_{\alpha_0}u_{\beta}$
    \item $vu_{\beta}-u_{\alpha_0}u_{\alpha_1}$
    \item $k_0a_{\alpha_0}-a_{\beta}u_{\alpha_1}-a_{\alpha_1}u_{\beta}$
    \item $k_1a_{\alpha_1}-a_{\beta}u_{\alpha_0}-a_{\alpha_0}u_{\beta}$
    \item $va_{\beta}-a_{\alpha_0}u_{\alpha_1}-a_{\alpha_1}u_{\alpha_0}$
    \item $k_0k_1-u^2_{\beta}$
    \item $k_0v-u^2_{\alpha_1}$
    \item $k_1v-u^2_{\alpha_0}$
    \item $k_0 \, \Sigma^{-1} \frac{1}{a_{\alpha_1} u_{\alpha_1}} =
            k_1 \, \Sigma^{-1} \frac{1}{a_{\alpha_0} u_{\alpha_0}}$,
    \item $k_r \, \Sigma^{-1} \frac{1}{a_{\beta} u_{\beta}} =
            v \, \Sigma^{-1} \frac{1}{a_{\alpha_r} u_{\alpha_r}} \quad r=0,1.$

    \end{enumerate}
\end{prop}

\begin{proof}
\begin{enumerate}
    \item  Proof follows from the equivalence given in $\eqref{equiv1}$, Consider the following class and equivalence
\begin{align*}
    k^{n}_1a^{\ell}_{\beta}u^{m}_{\beta}\Sigma^{-1}\frac{1}{a^{j}_{\alpha_0}u^{k}_{\alpha_0}}&= {v^{-(\ell+m)}}k^{n}_1v^\ell a^{\ell}_{\beta}v^mu^{m}_{\beta}\Sigma^{-1}\frac{1}{a^{j}_{\alpha_0}u^{k}_{\alpha_0}}\\
    &= \frac{u^{2n+m}_{\alpha_0}u^m_{\alpha_1} (a_{\alpha_1}u_{\alpha_0}+u_{\alpha_1}a_{\alpha_0})^{\ell}}{v^{\ell+m+n}}\Sigma^{-1}\frac{1}{a^{j}_{\alpha_0}u^{k}_{\alpha_0}}
    \end{align*}
which have the following type term:
\[
\frac{u^{2n+m+b}_{\alpha_0}u^{m+\ell-b}_{\alpha_1} a^b_{\alpha_1}}{v^{\ell+m+n}}\Sigma^{-1}\frac{1}{a^{j-\ell+b}_{\alpha_0}u^{k}_{\alpha_0}}
\]
for some $0 \leq b \leq \ell$, since the above factor is zero if $j+b-\ell\leq 0$. For $j+b-\ell> 0$, note that if $2n+m+b-k\geq 0$, then class lie in $\Coker \partial_{\bigstar}$ and have the following correspondence
\[
{u^{2n+m+b-k}_{\alpha_0}u^{m+\ell-b}_{\alpha_1} a^b_{\alpha_1}}\Sigma^{-1}\frac{1}{v^{\ell+m+n}a^{j-\ell+b}_{\alpha_0}}\in\Coker \partial_{\bigstar}
\]     
and if $2n+m+b-k< 0$, then class lie in $\Ker\mathcal{L}_{\bigstar}$ and have the following correspondence
\[
\frac{u^{m+\ell-b}_{\alpha_1} a^b_{\alpha_1}}{v^{\ell+m+n}}\Sigma^{-1}\frac{1}{a^{j-\ell+b}_{\alpha_0}u^{k-2n-m-b}_{\alpha_0}}
\]    
as one can identify the class $\Sigma^{-1}\frac{1}{a^{j-\ell+b}_{\alpha_0}u^{k-2n-m-b}_{\alpha_0}}$ in $\Ker\mathcal{L}_{\bigstar}$.

For the equivalence with $a^{-i}_{\alpha_0}$--class, consider the class 
\begin{align*}
    \frac{k^{n}_1a^{\ell}_{\beta}u^{m}_{\beta}}{a^i_{\alpha_1}}\Sigma^{-1}\frac{1}{a^{j}_{\alpha_0}u^{k}_{\alpha_0}}&= \frac{k^{n}_1v^mu^{m}_{\beta}u^\ell_{\alpha_0}a^{\ell}_{\beta}}{v^ma^{i}_{\alpha_1}}\Sigma^{-1}\frac{1}{a^{j}_{\alpha_0}u^{k+\ell}_{\alpha_0}}\\
    &= \frac{k^{n}_1v^mu^{m}_{\beta}({k_1a_{\alpha_1}+u_{\beta}a_{\alpha_0}})^{\ell}}{v^{m}a^{i}_{\alpha_1}}\Sigma^{-1}\frac{1}{a^{j}_{\alpha_0}u^{k+\ell}_{\alpha_0}}\\
    &=\frac{u^{2n+m}_{\alpha_0}u^{m}_{\alpha_1}({k_1a_{\alpha_1}+u_{\beta}a_{\alpha_0}})^{\ell}}{v^{n+m}a^{i}_{\alpha_1}}\Sigma^{-1}\frac{1}{a^{j}_{\alpha_0}u^{k+\ell}_{\alpha_0}}
    \end{align*}

which have the following type term:
\[
\frac{u^{2n+m}_{\alpha_0}u^{m}_{\alpha_1}k^{b}_1 a^b_{\alpha_1}u^{\ell-b}_{\beta}a^{\ell-b}_{\alpha_0}}{a^i_{\alpha_1}v^{m+n}}\Sigma^{-1}\frac{1}{a^{j}_{\alpha_0}u^{k+\ell}_{\alpha_0}} =
\frac{u^{2n+b+m-k}_{\alpha_0}u^{m+\ell-b}_{\alpha_1} }{a^{i-b}_{\alpha_1}}\Sigma^{-1}\frac{1}{v^{\ell+m+n}a^{j+b-\ell}_{\alpha_0}}
\]
Again consider the order set $\{0=b_0<b_1<\cdots<b_t=\ell\}$ consist of the exponent of $k_1a_{\alpha_1}$. Pick the smallest $b_s$ such that $j-\ell+b_s> 0$, then if $i-b_s\leq 0$, then there will be no $a_{\alpha_1}$ term in the denominator and the class up to relation is equal to the earlier one. But for $i-b_{s}> 0$, then class has representation in $\Coker \partial_{\bigstar}$ for $k\leq 2n+m+b_s$.

\item Proof follows using $(1)$, and restriction on $k$ is redundant because of multiplication with $\Sigma^{-1}\frac{1}{a^s_{\beta}u^t_{\beta}}$ factor.
\end{enumerate}
\end{proof}

Using the Proposition $\ref{ker}$ and the above equivalence, we have the cohomology ring of a point. 
\begin{thm}\label{cohomologyofpoint}
The $\RO(\K)$–graded Bredon cohomology of a point is given by
\begin{align*}
    \widetilde{H}^{\bigstar}_{\K}(S^0;\uF) &\cong
    \F[a_{\beta}, u_{\beta}, a_{\alpha_0}, u_{\alpha_0}, a_{\alpha_1}, u_{\alpha_1}, k_0, k_1, v] \\
    &\quad\oplus\ \F\Big\langle \Sigma^{-1} \frac{1}{a_{\beta}^{\,j} u_{\beta}^{\,k}} \Big\rangle
        [\,a_{\alpha_0}, a_{\alpha_1}, u_{\alpha_0}, u_{\alpha_1}, v, k_0, k_1\,] \\
    &\quad\oplus\ \F\Big\langle \Sigma^{-1} \frac{1}{a_{\alpha_0}^{\,j} u_{\alpha_0}^{\,k}} \Big\rangle
        [\,a_{\beta}, a_{\alpha_1}, u_{\beta}, u_{\alpha_1}, v, k_0, k_1\,] \\
    &\quad\oplus\ \F\Big\langle \Sigma^{-1} \frac{1}{a_{\alpha_1}^{\,j} u_{\alpha_1}^{\,k}} \Big\rangle
        [\,a_{\beta}, a_{\alpha_0}, u_{\beta}, u_{\alpha_0}, v, k_0, k_1\,] \\
    &\quad\oplus\ \F\Big\langle a_{\alpha_1}^{-i} k_1^{\,n} u_{\alpha_0}^{\,m} a_{\alpha_0}^{\,\ell} \,
        \Sigma^{-1} \frac{1}{a_{\beta}^{\,j} u_{\beta}^{\,k}} \ \colon\ b_s<i; k \leq 2n+b_s+m \Big\rangle
        [\, u_{\alpha_1}, k_0, v\,] \\
    &\quad\oplus\ \F\Big\langle a_{\alpha_0}^{-i} k_0^{\,n} u_{\alpha_1}^{\,m} a_{\alpha_1}^{\,\ell} \,
        \Sigma^{-1} \frac{1}{a_{\beta}^{\,j} u_{\beta}^{\,k}} \ \colon\ b_s<i; k \leq 2n+b_s+m \Big\rangle
        [\, u_{\alpha_0}, k_1, v\,] \\
    &\quad\oplus\ \F\Big\langle a_{\alpha_1}^{-i} k_1^{\,n}u_{\beta}^{\,m} a_{\beta}^{\,\ell}  \,
        \Sigma^{-1} \frac{1}{a_{\alpha_0}^{\,j} u_{\alpha_0}^{\,k}} \ \colon\ b_s<i; k \leq 2n+b_s+m \Big\rangle
        [\, u_{\alpha_1}, v, k_0\,] \\
    &\quad\oplus\ \F\Big\langle a_{\alpha_0}^{-i} k_0^{\,n} u_{\beta}^{\,m} a_{\beta}^{\,\ell}  \,
        \Sigma^{-1} \frac{1}{a_{\alpha_1}^{\,j} u_{\alpha_1}^{\,k}} \ \colon\ b_s<i; k \leq 2n+b_s+m \Big\rangle
        [\, u_{\alpha_0}, v, k_1\,] \\
    &\quad\oplus\ \F\Big\langle \Sigma^{-1} \frac{1}{a_{\beta}^{\,s} u_{\beta}^{\,t}} \Big\rangle \ \otimes\ 
        \Big[
            \F\Big\langle \Sigma^{-1} \frac{1}{a_{\alpha_0}^{\,j} u_{\alpha_0}^{\,k}} \Big\rangle
                [\,a_{\alpha_1}^{\pm}, u_{\alpha_1}, v, k_0, k_1\,] \\
    &\hspace{5cm} \oplus\ 
            \F\Big\langle \Sigma^{-1} \frac{1}{a_{\alpha_1}^{\,j} u_{\alpha_1}^{\,k}} \Big\rangle
                [\,a_{\alpha_0}^{\pm}, u_{\alpha_0}, v, k_0, k_1\,]
        \Big]
\end{align*}
where $n,m,\ell\geq0$; $i,j,k,s,t\geq1$ and $0\leq b_s\leq \ell$ as defined earlier in \eqref{cond}.

subject to the relations:
\begin{enumerate}
    \item $k_0 u_{\alpha_0} = u_{\alpha_1} u_{\beta}$,
    \item $k_1 u_{\alpha_1} = u_{\alpha_0} u_{\beta}$,
    \item $v u_{\beta} = u_{\alpha_0} u_{\alpha_1}$,
    \item $k_0 a_{\alpha_0} = a_{\beta} u_{\alpha_1} + a_{\alpha_1} u_{\beta}$,
    \item $k_1 a_{\alpha_1} = a_{\beta} u_{\alpha_0} + a_{\alpha_0} u_{\beta}$,
    \item $v a_{\beta} = a_{\alpha_0} u_{\alpha_1} + a_{\alpha_1} u_{\alpha_0}$,
    \item $k_0 k_1 = u_{\beta}^2$,
    \item $k_0 v = u_{\alpha_1}^2$,
    \item $k_1 v = u_{\alpha_0}^2$,
    \item $k_0 \, \Sigma^{-1} \frac{1}{a_{\alpha_1} u_{\alpha_1}} =
            k_1 \, \Sigma^{-1} \frac{1}{a_{\alpha_0} u_{\alpha_0}}$,
    \item $k_r \, \Sigma^{-1} \frac{1}{a_{\beta} u_{\beta}} =
            v \, \Sigma^{-1} \frac{1}{a_{\alpha_r} u_{\alpha_r}}, \quad r=0,1$.
\end{enumerate}
\end{thm}

\begin{remark}
    In the above expression, one can add more factors for a symmetric structure, but all those classes are already present in the above expression up to the above relations.  
\end{remark}

\sect{Equivariant Complex Projective Spaces}

For a complex representation $V$ of an abelian group $G$, one has a complex projective space $\cP(V)$, which is the space of complex lines in $V$ where the action of $G$ on $\cP(V)$ is induced from the action on $V$. Here, the underlying space is complex projective space $\C P^{\dim V-1}$. In this section, we will construct the complex projective space for the complete universe $\cU=\varinjlim_{n} n \cdot \rho_{\C}$
where $\rho_{\C}$ is the complex regular representation.

We first treat the case of an odd prime $p$, deferring the case $p=2$. For the group $\cG$, the irreducible characters are
\[
\chi^{i,j}(a^s b^r)=\exp(\frac{2\pi\iota}{p}(si+rj)),
\qquad 0\le i,j,r,s\le p-1.
\]  We order the characters lexicographically by $(i,j)$, and the tensor product satisfies
\[
\chi^{i_1,j_1}\otimes \chi^{i_2,j_2}
=\chi^{\,i_1+i_2 \bmod p,\; j_1+j_2 \bmod p}.
\]

This ordering induces an increasing filtration of the $\cG$--space $\cP(\rho_{\C})$:
\begin{myeq}{\label{filtration}}
\cP(1_{\C})
\subset \cP(1_{\C}\oplus\chi^{0,1})
\subset \cP(1_{\C}\oplus\chi^{0,1}\oplus\chi^{0,2})
\subset \cdots .
\end{myeq}

Using the above ordering and filtration, we define a $\cG$--representation $V_k$, where
$k=mp+\ell$ with $0\le m,\ell \le p-1$, by
\begin{myeq}
V_k :=
\bigoplus_{i=0}^{m-1}\;\bigoplus_{j=0}^{p-1} \chi^{i,j}
\;\oplus\;
\bigoplus_{t=0}^{\ell} \chi^{m,t}.
\end{myeq}

There is a cofiber sequence
\begin{myeq}{\label{cofiber}}
\cP(V_{k-1})_{+} \longrightarrow \cP(V_k)_{+}
\longrightarrow S^{\chi^{p-m,p-\ell}\otimes V_{k-1}},
\end{myeq}
where the representation on the sphere is
\begin{myeq}
\chi^{p-m,p-\ell}\otimes V_{k-1}
=
\bigoplus_{i=0}^{m-1}\;\bigoplus_{j=0}^{p-1} \chi^{p-m+i,\;p-\ell+j}
\;\oplus\;
\bigoplus_{t=0}^{\ell-1} \chi^{0,\;p-\ell+t}.
\end{myeq}

This representation may be identified, up to exponent, with the previously defined
representations $\alpha_i$ and $\beta$ by computing kernels. Indeed, 
due to the Proposition~\ref{exponentequivalence}, it suffices to identify the summands up to exponent.

For $p=2$, the filtration of $\cP(\rho_{\C})$ is
\begin{myeq}
\cP(1_{\C})
\subset \cP(1_{\C}\oplus 2\alpha_0)
\subset \cP(1_{\C}\oplus 2\alpha_0\oplus 2\alpha_1)
\subset \cP(\rho_{\C}).
\end{myeq}
Analogous to \eqref{cofiber}, the corresponding sphere representations are
\[
\{\,2\alpha_0,\; 2\alpha_1+2\beta,\; 2\alpha_0+2\alpha_1+2\beta\,\}.
\]

The following proposition gives a vanishing criterion for the cohomology groups for a specific type of grading in $\RO(C_p\times C_p)$ for any prime $p$.  

\begin{prop}\label{vanish}
Let $\gamma \in \RO(C_p\times C_p)$ be such that $|\gamma^{H}|\leq -1$ for all $H\leq C_p\times C_p$.  
Then,
\[
\widetilde{H}^{\gamma}_{\cG}(S^0;\uFp) = 0.
\]
where $p$ is any prime.
\end{prop}
\begin{proof}
We prove the statement for odd $p$; the other case is similar.

Consider $\gamma=b+\Sigma^{p-1}_{i=0}c_i\alpha_i+d\beta \in \RO(\cG)$ with the above conditions on the fixed point dimensions. Note that $b$ is negative.

The first step is to reduce $\gamma$ to a virtual representation $\gamma_0$ with
$b<0$ and $c_i,d\le 0$. Suppose $d>0$. We claim that
\begin{myeq}{\label{betafree}}
\widetilde{H}^{\gamma}_{\cG}(S^0;\uFp)
\cong
\widetilde{H}^{\gamma-d\beta}_{\cG}(S^0;\uFp).
\end{myeq}

This follows from the long exact sequence associated to the isotropy cofibration
\[
S(\beta)_{+} \longrightarrow S^0 \longrightarrow S^{\beta}.
\]
In degree $\gamma$, this yields
\begin{myeq}{\label{isotropy1}}
\cdots \to
\widetilde{H}^{\gamma-1}_{\cG}(S(\beta)_{+};\uFp)
\to
\widetilde{H}^{\gamma-\beta}_{\cG}(S^0;\uFp)
\to
\widetilde{H}^{\gamma}_{\cG}(S^0;\uFp)
\to
\widetilde{H}^{\gamma}_{\cG}(S(\beta)_{+};\uFp)
\to \cdots .
\end{myeq}

The outer terms are computed using the cofiber sequence
\[
\cG/K_{+} \longrightarrow \cG/K_{+} \longrightarrow S(\beta)_{+},
\qquad K=\ker(\beta).
\]
Thus, to prove \eqref{betafree}, it suffices to show that
\[
\widetilde{H}^{\bigstar}_{\cG}(\cG/K_{+};\uFp)=0
\quad\text{for}\quad \bigstar=\gamma,\ \gamma-1.
\]
Using the standard identification
\[
\widetilde{H}^{\bigstar}_{\cG}(\cG/K_{+};\uFp)
\cong
\widetilde{H}^{\res_{K}(\bigstar)}_{K}(S^0;\uFp),
\]
the claim follows.

For the above degree $\bigstar$, the corresponding restriction degree will be
\[\res_{K}(\gamma)=(b+2d)+\xi\Sigma_{i=0}^{p-1}c_i \mbox{ and }\res_{K}(\gamma-1)=(b+2d-1)+\xi\Sigma_{i=0}^{p-1}c_i\]
Thus $\widetilde{H}^{\bigstar}_{\cG}(\cG/K_{+};\uFp)$, which is $\widetilde{H}^{\res_{K}(\bigstar)}_{K}(S^0;\uFp)$, vanish for the above degree due to Proposition~\ref{mackey}. Applying the same technique iteratively, we get $\gamma_{0}$ with all its coefficients non-positive.

Assume now that
\(
\gamma_0 = b + \sum_{i=0}^{p-1} c_i \alpha_i + d\beta,
\quad b<0,\ \ c_i,d \le 0.
\)
Applying the same argument to the above cofiber sequences, we obtain an isomorphism
\[
\widetilde{H}^{\gamma_0}_{\cG}(S^0;\uFp)
\cong
\widetilde{H}^{b}_{\cG}(S^0;\uFp).
\]
Since $b<0$ is a negative integer grading, the reduced cohomology of a point vanishes.  
This completes the proof.
\end{proof}

Since the connecting homomorphisms in the cofiber sequence \eqref{cofiber} in Bredon cohomology are
$\widetilde H^{\bigstar}_{C_p\times C_p}(S^0;\uFp)$-module maps, and the suspension degrees of the sphere terms satisfy the hypotheses of Proposition \ref{vanish}, the relevant cohomology groups vanish. Consequently, all boundary maps are trivial, and the desired splitting follows.

Moreover, since cofiber sequences are preserved under smash products, one may inductively compute the Bredon cohomology of the $r$-fold smash power of the equivariant projective space for both prime cases. The argument again reduces to Proposition \ref{vanish}.

\begin{thm}{\label{cohomologyofcp}}
\begin{enumerate}
\item For an odd prime $p$ and $r\geq 2$, the additive structure of the Bredon cohomology of the $r$-fold smash power of the equivariant projective space is given by:
\begin{align*}
    \widetilde{H}^{\bigstar}_{\cG}\!\left({\cP(\cU)^{\wedge (r)}};\uFp\right)
        \;\cong\; 
        \bigoplus_{k=0}^{\infty}\bigoplus_{ i=0 }^{p-1}\bigoplus_{ j=0 }^{p-1}\widetilde{H}^{\bigstar-j\Sigma_{\ell=0}^{p-1}\alpha_{\ell}-i\beta-k\rho_{\C}}_{\cG}\left(\cP(\cU)^{\wedge (r-1)};\uFp\right),
\end{align*}
 and $r=1$ is given as
\begin{myeq}\label{cohpu}
    \widetilde{H}^{\bigstar}_{\cG}\left(\cP(\cU);\uFp\right)= \bigoplus_{k=0}^{\infty} \bigoplus_{i=0}^{p-1}\bigoplus_{j=0}^{p-1}\widetilde{H}^{\bigstar-j\Sigma_{\ell=0}^{p-1}\alpha_{\ell}-i\beta-k\rho_{\C}}_{\cG}\left(S^0;\uFp\right) 
\end{myeq}
with $(i,j,k)\neq (0,0,0)$.    
    \item For $p=2$ and $r\geq 2$, the additive structure of the Bredon cohomology of the $r$-fold smash power of the equivariant projective space is given by:
\begin{align*}
        \widetilde{H}^{\bigstar}_{\K}\!\left({\cP(\cU)^{\wedge (r)}};\uF\right)
        \;\cong\;
        \bigoplus_{i=1}^{  \infty}&
        \widetilde{H}^{\bigstar-i\rho_{\C}}_{\K}\!\left({\cP(\cU)^{\wedge (r-1)}};\uF\right) \\
        \bigoplus_{i=0}^{  \infty}&\Big(\widetilde{H}^{\bigstar-i\rho_{\C}-2\alpha_0}_{\K}\!\left({\cP(\cU)^{\wedge (r-1)}};\uF\right) \\
        &\oplus \widetilde{H}^{\bigstar-i\rho_{\C}-2\alpha_1-2\beta}_{\K}\!\left({\cP(\cU)^{\wedge (r-1)}};\uF\right) \\
        &\oplus \widetilde{H}^{\bigstar-i\rho_{\C}-2\alpha_0-2\alpha_1-2\beta}_{\K}\!\left({\cP(\cU)^{\wedge (r-1)}};\uF\right)\Big),
\end{align*}
with the base case $r=1$ is given by
 \begin{align*}
\widetilde{H}^{\bigstar}_{\K}\!\left({\cP(\cU)};\uF\right)
        \;\cong\;
        \bigoplus_{i=1}^{  \infty}&
        \widetilde{H}^{\bigstar-i\rho_{\C}}_{\K}\!\left({S^0};\uF\right) 
         \bigoplus_{i=0}^{  \infty} \Big(\widetilde{H}^{\bigstar-i\rho_{\C}-2\alpha_0}_{\K}\!\left({{S^0}};\uF\right) \\
        &\oplus \widetilde{H}^{\bigstar-i\rho_{\C}-2\alpha_1-2\beta}_{\K}\!\left({{S^0}};\uF\right) 
        \oplus \widetilde{H}^{\bigstar-i\rho_{\C}-2\alpha_0-2\alpha_1-2\beta}_{\K}\!\left({{S^0}};\uF\right)\Big).
\end{align*}
\end{enumerate}
\end{thm}
By Theorem~\ref{cohomologyofcp}, the cohomology of an $r$--fold smash product of
equivariant projective spaces can be expressed as a suspension of the cohomology of a
point. In the odd prime case, each summand has the form
\begin{myeq}{\label{cohoofcpp}}
\widetilde{H}^{\,
\bigstar
-\sum_{j=0}^{p-1}\sum_{i=1}^{p-1} k^{(j)}_{i}\,(jU_0+i\beta)
-\sum_{j=1}^{p-1} k^{(j)}_{0}\, jU_0
-\sum_{s=1}^{r} d_s \rho_{\C}
}_{\cG}\!\left(S^0;\uFp\right),
\end{myeq}
where \(U_0=\bigoplus_{\ell=0}^{p-1}\alpha_\ell\).

The nonnegative integers \(k^{(j)}_{i}\) and \(k^{(j)}_{0}\) satisfy
\[
\sum_{j=0}^{p-1}\sum_{i=1}^{p-1} k^{(j)}_{i}
+\sum_{j=1}^{p-1} k^{(j)}_{0} \le r,
\]
and at least
$
r-\Biggl(
\sum_{j=0}^{p-1}\sum_{i=1}^{p-1} k^{(j)}_{i}
+\sum_{j=1}^{p-1} k^{(j)}_{0}
\Biggr)
$
of the integers \(d_1,\dots,d_r\) are strictly positive.

\medskip

For the group \(\K\), the analogous \(r\)-fold smash product of equivariant projective spaces has the summand of shifted suspension of cohomology of a point which have the form
\begin{myeq}
\widetilde{H}^{\,
    \bigstar
    -\sum_{m=1}^r i_m\rho_{\C}
    -k_1(2\alpha_0)
    -k_2(2\alpha_1+2\beta)
    -k_3(2\alpha_0+2\alpha_1+2\beta)
}_{\cG}(S^0;\uF),
\end{myeq}
where \(k_1,k_2,k_3\) are nonnegative integers with
\(k_1+k_2+k_3\le r\), and at least \(r-(k_1+k_2+k_3)\) of the \(d_s\geq 0\) are
strictly positive.

\begin{prop}{\label{pnof}} 
\begin{enumerate}
    \item For an odd prime $p$, the Mackey functor-valued cohomology $\underline{H}^{\bigstar}_{\cG}\left(\cP(\cU)^{\wedge {(r)}};\uFp\right)$ does not contain the constant Mackey functor $\uFp$  for  $\bigstar=r\beta+2r(p-1)$.
    \item For $p=2$, the Mackey functor-valued cohomology  $\underline{H}^{\bigstar}_{\K}\left(\cP(\cU)^{\wedge {(r)}};\uF\right)$ does not contain the constant Mackey functor $\uF$ for $\bigstar=2r\alpha_0+2r$.
\end{enumerate}

 \end{prop} 

\begin{proof}
We prove the statement for odd \(p\); the \(p=2\) case is analogous. Suppose, towards a contradiction, that the constant Mackey functor \(\uFp\)
occurs as a summand of
$
\underline{H}^{\bigstar}_{\cG}\bigl(\cP(\cU)^{\wedge r};\uFp\bigr)
$
for
$
\bigstar = r\beta + 2r(p-1).
$
Restricting this copy of \(\uFp\) to the subgroup \(K=\langle b\rangle\)
gives the constant Mackey functor \(\uFp\) on \(K\); in particular
\(\res_K(\uFp)=\uFp\).

By the description in \eqref{cohoofcpp}, any summand in degree \(\bigstar\), gives a degree of the form
\[
\alpha
= r\beta + 2r(p-1)
  - \sum_{j=0}^{p-1}\sum_{i=1}^{p-1} k^{(j)}_{i}\,(jU_0 + i\beta)
  - \sum_{j=1}^{p-1} k^{(j)}_{0}\, jU_0
  - \sum_{s=1}^{r} d_s \rho_{\C},
\]

Consider the following computation
\[
\bigl|\res_K(\alpha)\bigr|
= 2rp
  - \sum_{j=0}^{p-1}\sum_{i=1}^{p-1} k^{(j)}_{i}\, (2pj + 2i)
  - \sum_{j=1}^{p-1} 2pj\, k^{(j)}_{0}
  - 2p^2\sum_{s=1}^{r} d_s.
\]
and
\[
\bigl|(\res_K(\alpha))^{K}\bigr|
= 2rp
  - \sum_{j=0}^{p-1}\sum_{i=1}^{p-1} k^{(j)}_{i}\,2i
  - 4\sum_{s=1}^{r} d_s.
\]
If $\bigl|\res_K(\alpha)\bigr|=0$, we find on substituting
\[
\bigl|(\res_K(\alpha))^{K}\bigr|
= \sum_{j=0}^{p-1}\sum_{i=1}^{p-1} 2pj\,k^{(j)}_{i}
  + \sum_{j=1}^{p-1} 2pj\, k^{(j)}_{0}
  + (2p^2-4)\sum_{s=1}^{r} d_s\geq 0.
\]

Claim: $\bigl|(\res_K(\alpha))^{K}\bigr|>0$. 

If $\bigl|(\res_K(\alpha))^{K}\bigr|=0$ implies $d_{s}=0$, $k^{(j)}_{i}=0$ for all $j\neq 0$ which implies $\Sigma_{i=1}^{p-1}k^{(0)}_{i}=r$, but $\bigl|\res_K(\alpha)\bigr|=2(rp-\Sigma^{p-1}_{i=1}ik^{(0)}_{i})> 0$ as $\max \{\Sigma^{p-1}_{i=1}ik^{(0)}_{i}\}=(p-1)r$, contradicting Proposition~\ref{mackey}. Therefore, no copy of the constant Mackey
functor \(\uFp\) can occur in degree \(\bigstar=r\beta+2r(p-1)\), completing
the proof for odd \(p\).
\end{proof}

In the following, we prove that a certain cohomology operation cannot be lifted to $\cG$ and $\K$ level, 
which is the similar result given by Caruso 
\cite{Car99} for 
$C_p$ group.  
\begin{thm}{\label{nonlift}}
Let $\theta$ be a cohomology operation of degree $2r(p-1)$ and $2r$ for an odd $p$ and $p=2$ case, respectively which does not involve the B\"ockstein $\beta$.  
Then there does not exist an equivariant cohomology operation $\tilde{\theta}$ lifting $\theta$.
\end{thm}

\begin{proof}
Let $p$ be an odd prime. Assume, toward a contradiction, that there exists an equivariant lift
\[
   \tilde{\theta}\colon H\underline{\mathbb{F}}_{p}
   \longrightarrow
   \Sigma^{2r(p-1)} H\underline{\mathbb{F}}_{p}
\]
whose underlying non-equivariant map agrees with $\theta$. 
Such a map induces, for every $G$-space $X$, a natural transformation of $\RO(\cG)$-graded Mackey functors
\[
   \underline{H}^{\bigstar}_{\mathcal G}(X;\underline{\mathbb F}_{p})
   \longrightarrow
   \underline{H}^{\bigstar+2r(p-1)}_{\mathcal G}(X;\underline{\mathbb F}_{p}),
\]
functorial in $X$.

Now take $X=\mathcal P(\mathcal U)^{\wedge (r)}$ and $\bigstar=r\beta$. 
By Theorem~\ref{cohoofcpp}, there is an inclusion of Mackey functors
\[
   \underline{H}^{0}_{\mathcal G}(S^{0};\underline{\mathbb F}_{p})
   \;\cong\;
   \underline{\mathbb F}_{p}
   \;\hookrightarrow\;
   \underline{H}^{\,r\beta}_{\mathcal G}
   \bigl(\mathcal P(\mathcal U)^{\wedge (r)};\underline{\mathbb F}_{p}\bigr).
\]

On the other hand, Proposition~\ref{pnof} shows that
\[
   \underline{H}^{\,r\beta+2r(p-1)}_{\mathcal G}
   \bigl(\mathcal P(\mathcal U)^{\wedge (r)};\underline{\mathbb F}_{p}\bigr)
\]
contains no direct summand isomorphic to $\underline{\mathbb F}_{p}$. 
Consequently, the composite
\[
   \underline{\mathbb F}_{p}
   \;\hookrightarrow\;
   \underline{H}^{\,r\beta}_{\mathcal G}
   \bigl(\mathcal P(\mathcal U)^{\wedge (r)};\underline{\mathbb F}_{p}\bigr)
   \;\xrightarrow{\;\tilde{\theta}\;}
   \underline{H}^{\,r\beta+2r(p-1)}_{\mathcal G}
   \bigl(\mathcal P(\mathcal U)^{\wedge (r)};\underline{\mathbb F}_{p}\bigr)
\]
is necessarily trivial as a morphism of Mackey functors.

Passing to underlying non-equivariant cohomology, the inclusion sends 
$\underline{\mathbb F}_{p}(\mathcal G/e)$ to the nonzero class 
\[
   \mathbb F_{p}\{x_{1}\otimes\cdots\otimes x_{r}\}.
\]
Since $\theta$ acts nontrivially on this element, the induced map on underlying cohomology is nonzero. 
This contradicts the triviality of the corresponding Mackey functor morphism. 
Hence no such lift $\tilde{\theta}$ can exist.

For $p=2$ case, the argument is analogous. One shows that there is no nontrivial Mackey functor morphism
\[
   \underline{H}^{2r\alpha_{0}}_{\K}
   \bigl(\mathcal P(\mathcal U)^{\wedge (r)};\underline{\mathbb F}\bigr)
   \longrightarrow
   \underline{H}^{2r\alpha_{0}+2r}_{\K}
   \bigl(\mathcal P(\mathcal U)^{\wedge (r)};\underline{\mathbb F}\bigr),
\]
and the same contradiction argument applies.
\end{proof}


\end{document}